\documentclass[reqno,a4paper]{amsart}
\usepackage[english]{babel}

\usepackage{graphicx,mathrsfs,tikz,latexsym,ifthen,amsmath,amsfonts,amssymb,amsthm,stmaryrd,fancyhdr,amscd,amsbsy,amstext}
\usepackage{enumerate,enumitem,empheq,mathtools,mathpazo,verbatim,a4wide} 
\mathtoolsset{showonlyrefs,showmanualtags}

\usepackage{epigraph}

\usepackage[utf8x]{inputenc}

\usepackage[toc,page]{appendix}
\usepackage[mathscr]{euscript} 
\allowdisplaybreaks

\usetikzlibrary{positioning,shapes,shadows,arrows}

\usepackage{color}
\definecolor{MyDarkBlue}{rgb}{0.15,0.25,0.45}

\newcommand{\triend}{\parbox{2mm}{\hfill} \hfill\text{\hspace{0.2mm}}\hfill$\triangle$}
\newcommand{\ocend}{\parbox{2mm}{\hfill} \hfill\text{\hspace{0.2mm}}\hfill$\oslash$}

\usepackage{hyperref}

\usepackage[hyperpageref]{backref} 

\newtheorem{theorem}{Theorem}
\newtheorem{proposition}[theorem]{Proposition}
\newtheorem{lemma}[theorem]{Lemma}
\newtheorem{corollary}[theorem]{Corollary}

\newtheorem{corollary*}{Corollary}
\newtheorem*{theorem*}{Theorem}
\newtheorem*{proposition*}{Proposition}
\newtheorem*{conjecture*}{Conjecture}

\numberwithin{equation}{section}
\numberwithin{theorem}{section}

\theoremstyle{remark}
\newtheorem{ex}[theorem]{Example}
\newenvironment{example}{\begin{ex}}{\triend\end{ex}}

\theoremstyle{remark}
\newtheorem{rem}[theorem]{Remark}
\newenvironment{remark}{\begin{rem}}{\triend\end{rem}}

\theoremstyle{definition}
\newtheorem{defin}[theorem]{Definition}
\newenvironment{definition}{\begin{defin}}{\ocend\end{defin}}

\newcommand{\fH}{{\mathbf{H}}}
\newcommand{\fI}{{\mathbf{1}}}
\newcommand{\fU}{{\mathbf{U}}}
\newcommand{\Sh}{\mathsf{Sh}}
\newcommand{\semi}{\mathsf{ss}}

\newcommand{\bC}{\ensuremath{\mathbb{C}}}

\newcommand{\bF}{\ensuremath{\mathbb{F}}}

\newcommand{\bP}{\ensuremath{\mathbb{P}}}
\newcommand{\bQ}{\ensuremath{\mathbb{Q}}}
\newcommand{\bR}{\ensuremath{\mathbb{R}}}

\newcommand{\bV}{\ensuremath{\mathbb{V}}}

\newcommand{\bZ}{\ensuremath{\mathbb{Z}}}

\renewcommand{\SS}{{\operatorname{SS}}}

\newcommand{\cSh}{\operatorname{\bf{Sh}}^c}

\newcommand{\relmid}{\mathrel{}\middle|\mathrel{}}

\newcommand{\la}{\left\langle}
\newcommand{\ra}{\right\rangle}

\newcommand{\lc}{\left\{}
\newcommand{\rc}{\right\}}

\newcommand{\Xcal}{{\mathcal X}}
\newcommand{\Acal}{{\mathcal A}}
\newcommand{\Hcal}{{\mathcal H}}
\newcommand{\Ecal}{{\mathcal E}}
\newcommand{\Bcal}{{\mathcal B}}
\newcommand{\Kcal}{{\mathcal K}}
\newcommand{\Ncal}{{\mathcal N}}
\newcommand{\Mcal}{\mathcal{M}}
\newcommand{\Ccal}{{\mathcal C}}

\newcommand{\Ocal}{{\mathcal O}}
\newcommand{\Gcal}{{\mathcal G}}
\newcommand{\Fcal}{{\mathcal F}}
\newcommand{\Ical}{{\mathcal I}}
\newcommand{\Qcal}{{\mathcal Q}}
\newcommand{\Rcal}{{\mathcal R}}
\newcommand{\Lcal}{{\mathcal L}}

\newcommand{\Tcal}{{\mathcal T}}

\newcommand{\Ucal}{\mathcal{U}}

\newcommand{\Scal}{\mathcal{S}}

\newcommand{\slfrak}{\mathfrak{sl}}

\newcommand{\glfrakhat}{\widehat{\mathfrak{gl}}}
\newcommand{\slfrakhat}{\widehat{\mathfrak{sl}}}

\newcommand{\gfrak}{\mathfrak{g}}

\newcommand{\glfrak}{\mathfrak{gl}}

\newcommand{\Sfrak}{\mathfrak{S}}

\newcommand{\K}{{\mathbb{K}}}
\newcommand{\R}{{\mathbb{R}}}
\newcommand{\N}{{\mathbb{N}}}
\newcommand{\C}{{\mathbb{C}}}
\newcommand{\Z}{{\mathbb{Z}}}
\newcommand{\Q}{{\mathbb{Q}}}

\newcommand{\PP}{{\mathbb{P}}}
\newcommand{\A}{{\mathbb{A}}}
\newcommand{\G}{{\mathbb{G}}}
\newcommand{\F}{{\mathbb{F}}}

\newcommand{\V}{{\mathbb{V}}}

\newcommand{\Xscr}{{\mathscr{X}}}

\newcommand{\Dscr}{{\mathscr{D}}}
\newcommand{\Bscr}{{\mathscr{B}}}

\newcommand{\Osf}{\mathsf{O}}

\newcommand{\Csf}{\mathsf{C}}

\newcommand{\tildeDHbf}{\widetilde{\mathbf{D}}\mathbf{H}}
\newcommand{\DHbf}{\mathbf{DH}}
\newcommand{\DUbf}{\mathbf{DU}}
\newcommand{\DCbf}{\mathbf{DC}}
\newcommand{\Kbf}{\mathbf{K}}
\newcommand{\Hbf}{\mathbf{H}}
\newcommand{\kbf}{\mathbf{k}}
\newcommand{\onebf}{\mathbf{1}}
\newcommand{\Ubf}{\mathbf{U}}
\newcommand{\Pbf}{\mathbf{P}}

\newcommand{\abf}{\mathbf{a}}
\newcommand{\Sbf}{\mathbf{S}}
\newcommand{\Cbf}{\mathbf{C}}

\newcommand{\Zbf}{\mathbf{Z}}
\newcommand{\Bbf}{\mathbf{B}}
\newcommand{\bbf}{\mathbf{b}}

\newcommand{\rk}{\operatorname{rk}}

\newcommand{\Spec}{\operatorname{Spec}}

\newcommand{\rootlines}{\sqrt[n]{(\mathcal{L},s)/\mathscr{X}}}
\newcommand{\rootline}{\sqrt[n]{\mathcal{L}/\mathscr{X}}}

\newcommand{\rootdiv}{\sqrt[n]{\mathscr{D}/\mathscr{X}}}
\newcommand{\schemerootdiv}{\sqrt[n]{D/X}}

\newcommand{\Rep}{\mathsf{Rep}}
\newcommand{\Tor}{\mathsf{Tor}}
\newcommand{\tor}{\mathsf{tor}}
\newcommand{\Coh}{\mathsf{Coh}}

\newcommand{\bun}{\mathsf{bun}}
\newcommand{\QCoh}{\mathsf{QCoh}}
\newcommand{\Db}{\mathsf{D}^{\mathsf{b}}}

\newcommand{\Ksf}{\mathsf{K}_0}
\newcommand{\Ksfnum}{\mathsf{K}_0^{\mathsf{num}}}

\newcommand{\nintpart}[1]{{}_n \lfloor #1 \rfloor}
\newcommand{\nfrapart}[1]{{}_n \{ #1 \}}

\newcommand{\Hom}{\mathsf{Hom}}
\newcommand{\Ext}{\mathsf{Ext}}
\newcommand{\Aut}{\mathsf{Aut}}
\newcommand{\End}{\mathsf{End}}
\newcommand{\Pic}{\mathsf{Pic}}

\hypersetup{anchorcolor=black, linktocpage=true,
colorlinks=true,
citecolor=red,
linkcolor=MyDarkBlue,
urlcolor=black,
pdfauthor={Francesco Sala and Olivier Schiffmann},
pdftitle={The circle quantum group and the infinite root stack of a curve}, 
breaklinks=true,
plainpages=true
}

\title[$\Ubf_\upsilon(\slfrak(S^1))$ and the infinite root stack of a curve]{The circle quantum group and the infinite root stack of a curve}

\author[F.~Sala]{Francesco Sala}
\address[Francesco Sala]{Università di Pisa, Dipartimento di Matematica, Largo Bruno Pontecorvo 5, 56127 Pisa (PI), Italy}
\address{Kavli IPMU (WPI), UTIAS, The University of Tokyo, Kashiwa, Chiba 277-8583, Japan}
\curraddr{}
\email{\href{mailto:francesco.sala@unipi.it}{francesco.sala@unipi.it}}

\author[O.~Schiffmann]{Olivier Schiffmann}
\address[Olivier Schiffmann]{Laboratoire de Math\'ematiques, Faculté des Sciences d'Orsay, Université Paris-Sud, B\^at. 307, 91405 Orsay Cedex, France}
\curraddr{}
\email{\href{mailto:olivier.schiffmann@math.u-psud.fr}{olivier.schiffmann@math.u-psud.fr}}

\thanks{The work of the first-named author is partially supported by World Premier International Research Center Initiative (WPI), MEXT, Japan, by JSPS KAKENHI Grant number JP17H06598 and by JSPS KAKENHI Grant number JP18K13402.}

\begin{document}

\begin{flushright}
IPMU--18-0105
\end{flushright}

\maketitle\thispagestyle{empty}

\vspace{-1cm}

\begin{center}
{\footnotesize Appendix~\ref{sec:tatsuki} by Tatsuki Kuwagaki}
\end{center}

\vskip 1cm

\begin{abstract}
In the present paper, we give a definition of the quantum group $\Ubf_\upsilon(\slfrak(S^1))$ of the circle $S^1\coloneqq \R/\Z$, and its fundamental representation. Such a definition is motivated by a realization of a quantum group $\Ubf_\upsilon(\slfrak(S^1_\Q))$ associated to the rational circle $S^1_\Q\coloneqq \Q/\Z$ as a direct limit of $\Ubf_\upsilon(\slfrakhat(n))$'s, where the order is given by divisibility of positive integers. The quantum group $\Ubf_\upsilon(\slfrak(S^1_\Q))$ arises as a subalgebra of the Hall algebra of coherent sheaves on the infinite root stack $X_\infty$ over a fixed smooth projective curve $X$ defined over a finite field. Via this Hall algebra approach, we are able to realize geometrically the fundamental and the tensor representations, and a family of symmetric tensor representations, depending on the genus $g_X$, of $\Ubf_\upsilon(\slfrak(S^1_\Q))$. Moreover, we show that $\Ubf_\upsilon(\slfrakhat(+\infty))$ and $\Ubf_\upsilon(\slfrakhat(\infty))$ are subalgebras of $\Ubf_\upsilon(\slfrak(S^1_\Q))$. As proved by T. Kuwagaki in the appendix, the quantum group $\Ubf_\upsilon(\slfrak(S^1))$ naturally arises as well in the mirror dual picture, as a Hall algebra of constructible sheaves on the circle $S^1$.
\end{abstract}

\epigraph{\textit{Once you want to free your mind about\newline a concept of harmony and music being correct,\newline you can do whatever you want.}}{Giorgio Moroder}

\tableofcontents

\bigskip\section{Introduction}

The present paper is the first of a series aimed at studying the Hall algebra of the category of coherent sheaves on the infinite root stack $X_\infty$ over a fixed smooth projective curve $X$ defined over a finite field $\F_q$. 

To make this introduction more accessible, we will start by defining algebraically the resulting Hopf algebras $\Ubf_\upsilon(\slfrak(S^1))$ and $\Ubf_\upsilon(\slfrak(S^1_\Q))$ associated, respectively, with the circle $S^1\coloneqq \R/\Z$ and the rational circle $S^1_\Q\coloneqq \Q/\Z$.
There are similar Hopf algebras associated to $\R$ and $\Q$. We will motivate our definitions by showing that $\Ubf_\upsilon(\slfrak(S^1_\Q))$ is also obtained as a certain direct limit of $\Ubf_\upsilon(\slfrakhat(n))$. In addition, we will introduce the fundamental representations of these algebras.

The discussion of the choice of such a limit, as that of the construction of a family of symmetric tensor representations of $\Ubf_\upsilon(\slfrak(S^1_\Q))$, is postponed to the second part of the introduction, where we will show that the algebraic constructions stem naturally from the study of the Hall algebra of coherent sheaves on the infinite root stack. 

\subsection{Definition of $\Ubf_\upsilon(\slfrak(S^1))$}

Let us introduce some notation. We call \textit{interval} of $S^1$ a half-open interval $J=[a,b[\subseteq S^1$. We say that an interval $J$ is \textit{rational} if $J\subseteq S^1_\Q$. We say that an interval $J$ (resp.\ rational interval $J$) is \textit{strict} if $J\neq S^1$ (resp.\ $J\neq S^1_\Q$). For an interval $J$, we denote by $\chi_J$ its characteristic function and for two characteristic functions $f, g$, we define
\begin{align}\label{eq:eulerform-introduction}
\langle f, g\rangle \coloneqq\sum_x\, f_{-}(x)(g_{-}(x)-g_{+}(x))\ ,
\end{align}
where we have set $h_{\pm}(x)\coloneqq\lim_{t \to 0, t >0} h(x\pm t)$, and $(f,g)\coloneqq\langle f, g\rangle+\langle g, f\rangle$. Finally, let $\widetilde \Q \coloneqq\Q[\upsilon, \upsilon^{-1}]/(\upsilon^2-q)$.

\begin{definition}\label{def:circlequantumgroup}
$\Ubf_\upsilon(\slfrak(S^1))$ is the topological $\widetilde \Q$-Hopf algebra generated by elements $E_J, F_J, K_{J'}^{\pm 1}$, where $J$ (resp.\ $J'$) runs over all strict intervals (resp.\ intervals), modulo the following set of relations:
\begin{itemize}\itemsep0.4cm
\item \textit{Drinfeld-Jimbo relations}:
\begin{itemize}\itemsep0.2cm
\item for any intervals $I, I_1, I_2$ and strict  interval $J$,
\begin{align}
[K_{I_1},K_{I_2}] &=0\ ,\\
K_{I}\, E_{J}\, K_{I}^{-1} &=\upsilon^{( \chi_{I},\chi_{J})}\, E_{J}\ ,\\[2pt]
K_{I}\, F_{J}\, K_{I}^{-1} &=\upsilon^{-( \chi_{I},\chi_{J})}\, F_{J} \ ;
\end{align}
\item if $J_1, J_2$ are strict intervals such that $J_1 \cap J_2 = \emptyset$, 
\begin{align}
[E_{J_1},F_{J_2}]=0\ ;
\end{align}
\item for any strict interval $J$,
\begin{align}
[E_J,F_J]=\frac{K_J-K_J^{-1}}{\upsilon-\upsilon^{-1}} \ ;
\end{align} 
\end{itemize}
\item \textit{join relations}:
\begin{itemize}\itemsep0.2cm
\item if $J_1,J_2$ are strict intervals of the form $J_1=[a,b[$ and $J_2=[b,c[$ such that $J_1 \cup J_2$ is again a interval,
\begin{align}
K_{J_1}\, K_{J_2}=K_{J_1\cup J_2}\ ;
\end{align}
\item if $J_1, J_2$ are strict intervals of the form $J_1=[a,b[$ and $J_2=[b,c[$ such that $J_1 \cup J_2$ is again a strict interval,
\begin{align}
\begin{aligned}
E_{J_1\cup J_2} &=\upsilon^{1/2}\,E_{J_1}\,E_{J_2} - \upsilon^{-1/2}\,E_{J_2}\, E_{J_1}\ ,  \\[2pt]
F_{J_1\cup J_2} &=\upsilon^{-1/2}\,F_{J_2}\,F_{J_1} - \upsilon^{1/2}\, F_{J_1}\, F_{J_2}  \ ;
\end{aligned}
\end{align}
\end{itemize}
\item \textit{nest relations}:
\begin{itemize}\itemsep0.2cm
\item if $J_1, J_2$ are strict intervals such that $\overline{J_1} \cap \overline{J_2} =\emptyset$,
\begin{align}
[E_{J_1},E_{J_2}]=0\quad\text{and}\quad [F_{J_1}, F_{J_2}]=0\ ;
\end{align}
\item if $J_1, J_2$ are strict intervals such that $J_1\subseteq J_2$,
\begin{align}
\begin{aligned}
\upsilon^{\langle \chi_{J_1}, \chi_{J_2}\rangle} \, E_{J_1}\, E_{J_2} &=\upsilon^{\langle \chi_{J_2}, \chi_{J_1}\rangle} \, E_{J_2}\, E_{J_1} \ ,\\[2pt]
\upsilon^{\langle \chi_{J_1}, \chi_{J_2}\rangle} \, F_{J_1}\, F_{J_2} &=\upsilon^{\langle \chi_{J_2}, \chi_{J_1}\rangle} \, F_{J_2}\, F_{J_1} \ .
\end{aligned}
\end{align}
\end{itemize}
\end{itemize}
The coproduct is:
\begin{align}\label{E:coprod}
\begin{aligned}
\tilde \Delta(K_J)&=K_J\otimes K_J\ ,\\[2pt]
\tilde \Delta (E_{[a,\, b[})&=E_{[a,\, b[} \otimes 1 + \sum_{a<c<b} \upsilon^{-1/2}\, (\upsilon-\upsilon^{-1})\, E_{[a,\, c[}\, K_{[c,\, b[} \otimes E_{[c,\, b[} + K_{[a,\, b[} \otimes E_{[a,\, b[} \ ,\\[2pt]
\tilde \Delta (F_{[a,\, b[})&=1\otimes F_{[a, \, b[} - \sum_{a<c<b} \upsilon^{-1/2}\, (\upsilon-\upsilon^{-1})\, F_{[c,\, b[}\otimes F_{[a,\, c[}\, K_{[c,\, b[}^{-1}  + F_{[a,\, b[}\otimes K_{[a,\, b[}^{-1}\ .
\end{aligned}
\end{align}
Here the sums on the right-hand-side run over all possible real values $c\in [a, b[$.
\end{definition}
The quantum group $\Ubf_\upsilon(\slfrak(S^1))$ has some interesting unusual features; for example it is neither Noetherian nor finitely generated, but it is a quadratic algebra, i.e. the usual Serre relations are not present. In addition, there is no notion of a \textit{simple} root space. The formulas for the coproduct  involve some wildly infinite (uncountable) sums hence take values in some very large product space; nevertheless
$\Ubf_\upsilon(\slfrak(S^1))$ is a Hopf algebra in this topological sense, see Section~\ref{sec:Hallalgebrainfty} for the case of $\Ubf_\upsilon(\slfrak(S^1_{\Q}))$.
\begin{defin}
$\Ubf_\upsilon(\slfrak(S^1_\Q))$ is the subalgebra of $\Ubf_\upsilon(\slfrak(S^1))$ generated by elements $E_J, F_J, K_{J'}^\pm$ such that $J$ (resp.\ $J'$) is a strict rational interval (resp.\ rational interval). It is a topological Hopf algebra if we define the coproduct by the same formula as \eqref{E:coprod} in which $c$ is in addition required to belong to $\mathbb{Q}$. \hfill$\oslash$
\end{defin}

Now, let us motivate algebraically the above definitions. A reasonable way to define a quantum group whose Dynkin diagram would be the circle is as a limit. More precisely, we approximate the circle by the cyclic quiver $A_{n-1}^{(1)}$ with $n$ vertices and we consider the quantum enveloping algebra $\Ubf_\upsilon(\slfrakhat(n))$ of the Kac-Moody Lie algebra\footnote{To be precise, here we are considering the quantum enveloping algebra of the \textit{derived algebra} of $\slfrakhat(n)$. Such a quantum group is the correct one to consider from the point of view of Hall algebras. See for example \cite[Section~3.3]{book:schiffmann2012} for details.} of type $A_{n-1}^{(1)}$. Then, we decide a way to make the approximation more and more fine, passing from $n$ vertices to $m$ vertices, with $m>n$. This corresponds to considering embeddings $\Ubf_\upsilon(\slfrakhat(n))\to \Ubf_\upsilon(\slfrakhat(m))$: by \textit{letting $n$ tend to infinity}, one gets the quantum group of the circle as a limit (with respect to the directed system given by these embeddings). In the past, similar direct limits were also considered, but rather for the real line instead of the circle. For instance, we may increase by one the number of vertices at each step of the approximation, by inserting one extra vertex (on the right) to the underlying finite quiver $A_{n-1}$ of $A_{n-1}^{(1)}$,  thereby obtaining $\Ubf_\upsilon(\slfrakhat(+\infty))$. Alternatively, we may increase by two the number of vertices at each step by adding two vertices (one on the left, one on the right) to $A_{n-1}$, thereby obtaining $\Ubf_\upsilon(\slfrakhat(\infty))$ (see \cite{art:hernandez2011} for a different definition of $\Ubf_\upsilon(\slfrakhat(\infty))$).

In this paper, at each step of the approximation, we pass from $n$ vertices to $kn$ vertices, by ``subdividing" any vertex into $k$ vertices. More precisely, we consider a directed system
\begin{align}
\big(\Ubf_\upsilon(\slfrakhat(n)), \Omega_{kn,n}\colon \Ubf_\upsilon(\slfrakhat(n))\to\Ubf_\upsilon(\slfrakhat(kn))\big)\ ,
\end{align}
where the directed set is the positive integers with the order given by divisibility, and the definition of $\Omega_{kn, n}$ is given in Theorem \ref{thm:circle-hall} below. We stress here that the direct systems giving $\Ubf_\upsilon(\slfrakhat(+\infty))$ and $\Ubf_\upsilon(\slfrakhat(\infty))$ are direct systems of Hopf algebras while $\Omega_{kn,n}$ is only a morphism of algebras and \textit{not} of coalgebras. 
\begin{theorem}
The algebra $\Ubf_\upsilon(\slfrak(S^1_\Q))$ is isomorphic to the direct limit of algebras
\begin{align}
\lim_{\genfrac{}{}{0pt}{}{\to}{n}}\, \Ubf_\upsilon(\slfrakhat(n))
\end{align}
where the order is given by divisibility of positive integers. Both $\Ubf_\upsilon(\slfrakhat(+\infty))$ and $\Ubf_\upsilon(\slfrakhat(\infty))$ are subalgebras of $\Ubf_\upsilon(\slfrak(S^1_\Q))$.
\end{theorem}

The existence and definition of the coproduct of $\Ubf_\upsilon(\slfrak(S^1_\Q))$ and $\Ubf_\upsilon(\slfrak(S^1))$ is best seen from the Hall algebra picture, which we explain in the next section. We conclude this part of the introduction by defining the fundamental representations of both algebras. The \textit{fundamental representation} of $\Ubf_\upsilon(\slfrak(S^1))$ is the $\widetilde{\Q}$-vector space
\begin{align}
\V_{S^1}\coloneqq\bigoplus_{y \in \R}\, \widetilde \Q\, \vec u_y \ ,
\end{align} 
with the action given by
\begin{align}
F_{[a,b[}\bullet \vec u_y =& \delta_{\{b+y\}, 0}\, \upsilon^{1/2}\,\vec u_{y+b-a}\ , \\[4pt]
E_{[a,b[}\bullet \vec u_y =& \delta_{\{a+y\}, 0}\, \upsilon^{-1/2}\,\vec u_{y+a-b}\ , \\[4pt]
K_{[a',b'[}^\pm\bullet \vec u_y =& \upsilon^{\pm(\delta_{\{b'+y\},0}-\delta_{\{a'+y\},0})}\, \vec u_y\ ,
\end{align}
for intervals $[a,b[, [a', b'[\subseteq S^1$, with $[a,b[$ strict, and $y\in \Q$. Similarly, the \textit{fundamental representation} of $\Ubf_\upsilon(\slfrak(S^1_\Q))$ is as a $\widetilde{\Q}$-vector space
\begin{align}
\V_{S^1_\Q}\coloneqq\bigoplus_{y \in \Q}\, \widetilde \Q\, \vec u_y \subset \V_{S^1}
\end{align}  
and the $\Ubf_\upsilon(\slfrak(S^1_\Q))$-action is induced by the above one.

\begin{remark}
	We want to emphasize that, in the main body of the paper, we deal only with the Hall algebra realization of $\Ubf_\upsilon(\slfrak(S^1_\Q))$. This choice is due to the study, from the Hall algebra perspective, of the category of coherent sheaves on the so-called \textit{infinite root stack} of Talpo and Vistoli over a curve (which is equivalent to the category of parabolic sheaves over the curve with rational weights). There is a similar Hall algebra realization of the quantum group  
$\Ubf_\upsilon(\slfrak(S^1))$ involving a suitable category of parabolic torsion sheaves with \textit{real} weights \footnote{Most of the results and techniques used in this paper can be exported to the real setting using a real-to-rational weights approximation procedure à la Mehta--Seshadri \cite{art:mehtaseshadri1980}; we thank Mattia Talpo for suggesting such an approach.}. Nevertheless, we decided to focus on the rational case and the infinite root stack in the main body of this paper. The appendix written by Kuwagaki contains Hall algebra realizations of both $\Ubf_\upsilon(\slfrak(S^1))$ and $\Ubf_\upsilon(\slfrak(S^1_\Q))$ by means of the category of constructible sheaves on $S^1$ with prescribed singular support; in that setting the passage from rational weights to real weights is very natural, see Section~\ref{sec:S1-mirror}. The two Hall algebra realizations of $\Ubf_\upsilon(\slfrak(S^1_\Q))$ are related by a mirror symmetry derived equivalence; as suggested by the referee, it is very likely that the same holds for $\Ubf_\upsilon(\slfrak(S^1))$ as well.
\end{remark}

\subsubsection*{Definition of $\Ubf_\upsilon(\slfrak(\R))$} Replacing the circle $\R/\Z$ by $\R$, we can define in an entirely similar fashion a
quantum group $\Ubf_\upsilon(\slfrak(\R))$. The generators $E_I,F_I, K^{\pm 1}_{I}$ are now indexed by half-open intervals $I=[a,b[$ in $\R$, but all relations remain the same as in Definition~\ref{def:circlequantumgroup}. The quantum group $\Ubf_\upsilon(\slfrak(\Q))$ is the subalgebra of $\Ubf_\upsilon(\slfrak(\R))$ generated by elements associated to rational intervals $[a,b[\subset \Q$; it is a topological Hopf algebra by defining the coproduct by the same formula as \eqref{E:coprod} in which $c$ is in addition required to belong to $\mathbb{Q}$. 

Similarly, one can define the quantum group $\Ubf_\upsilon(\slfrak(\Z))$ generated by elements associated to integral intervals $[a,b[\subset \Z$. It is easy to see that this quantum group coincides with the quantum enveloping algebra $\Ubf_\upsilon(\slfrak(\infty))$ of the Lie algebra $\slfrak(\infty)$.

\subsubsection*{Fock spaces}

In \cite{art:salaschiffmann2020} we provide an algebraic construction of the Fock space $\Fcal_\K$ for $\Ubf_\upsilon(\slfrak(\K))$, for $\K\in \{\Z, \Q, \R\}$, as the vector space generated by $\K$-pyramids. For $\K=\Z$ this construction coincides with the action of $\Ubf_\upsilon(\slfrak(\infty))$ on the usual Fock space generated by partitions, since $\Ubf_\upsilon(\slfrak(\infty))\simeq \Ubf_\upsilon(\slfrak(\Z))$ and partitions correspond to $\Z$-pyramids. In addition, by adapting to our setting the ``folding procedure" of Hayashi \cite{art:hayashi1990} and Misha-Miwa \cite{art:misramiwa1990} we obtain an action of $\Ubf_\upsilon(\slfrak(S^1_\K))$ on $\Fcal_\K$ for $\K\in \{\Q, \R\}$.

\subsubsection*{Continuum Kac--Moody algebras and continuum quantum groups}

By taking the semi-classical limit of the quantum groups defined above, we  obtain the Lie algebra $\slfrak(\K)$ with $\K\in \{\Z, \Q, \R, S^1_\Q, S^1\}$. Together with Andrea Appel, in \cite{art:appelsalaschiffmann2020}, by mimicking the construction of Kac--Moody algebras, we define Lie algebras associated to a large class of oriented one-manifolds, which we called \textit{continuum Kac--Moody algebras}, which generalize $\slfrak(\K)$ (i.e., the continuum Kac--Moody algebra of $\K$ is simply $\slfrak(\K)$). In \cite{art:appelsala2020}, Appel and the first-named author prove that there is a structure of a topological bialgebra on the continuum Kac--Moody algebras and construct the corresponding quantizations, which go under the name of continuum quantum groups. These bialgebras generalize $\Ubf_\upsilon(\slfrak(\K))$.

\subsection{$\Ubf_\upsilon(\slfrak(S^1_\Q))$  and Hall algebras}

\subsubsection*{Hall algebras}

The study of Hall algebras associated to quivers goes back to the foundational work of Ringel \cite{art:ringel1990} and Green \cite{art:green1995}, who in particular proved that Hall algebra
of the category of representations over some finite field $\F_q$ of a quiver $\Qcal$ 
contains the Drinfeld-Jimbo quantum group $\Ubf^+_{\upsilon}(\mathfrak{g}_{\Qcal})$ of the Kac-Moody algebra $\gfrak_{\Qcal}$ associated to $\Qcal$. Here $\upsilon=\sqrt{q}$. The study of Hall algebras of (smooth, projective) curves was later initiated by Kapranov \cite{art:kapranov1997}, who 
in particular showed that these were related to quantum loop algebras. For instance, the Hall algebra of the category of coherent sheaves over $\mathbb{P}^1$ contains Drinfeld's new realization of the quantized enveloping algebra $\Ubf^+_\upsilon(\slfrakhat(2))$. On the other hand, the Hall algebra of coherent sheaves on a elliptic curve over a finite field realizes a positive part of the quantum toroidal algebra of type $\glfrak(1)$ as exploited by Burban and the second-named author in \cite{art:burbanschiffmann2012}. 

Another important family of Hall algebras consists of those associated with the categories of coherent sheaves on weighted projective lines, introduced by \cite{art:geiglelenzing1987}. For instance, weighted projective lines of tame and tubular type provide quantized enveloping algebras of affine or toroidal algebras (in the latter case, of type D and E) --- see \cite{art:schiffmann2004,art:burbanschiffmann2013}. 

\subsubsection*{Root stacks}

Weighted projective lines are examples of \textit{root stacks} over $\PP^1$. Root stacks were introduced independently by \cite{art:cadman2007} and \cite{art:abramovichgrabervistoli2008} in great generality. The study of the Hall algebra associated with the abelian category of coherent sheaves on a root stack over a curve has been initiated by \cite{art:lin2014}\footnote{\cite{art:lin2014} deals with the so-called \textit{spherical} Hall algebras of abelian categories of parabolic coherent sheaves on a curve. As proved by \cite{art:bornevistoli2012}, such an abelian category is equivalent to that of a root stack over the curve.}, where in particular a shuffle presentation was obtained.

In this paper we fix a smooth projective curve $X$ defined over a finite field $\F_q$.  We consider the $n$-th root stacks $X_n$ obtained from the original curve $X$ by --- roughly speaking --- ``replacing" a fixed rational point $p$ by the trivial gerbes $p_n\coloneqq\Bcal\mu_n$ for any positive integer $n$. The Serre subcategory $\Tor_{p_n}(X_n)$ of torsion sheaves on $X_n$ \textit{set-theoretically} supported at $p_n$
 is isomorphic to the category $\mathsf{Rep}_k^{\mathsf{nil}}(A_{n-1}^{(1)})$ of nilpotent representation of the cyclic quiver $A_{n-1}^{(1)}$ with $n$ vertices. Hence, by the result of Ringel and Green mentioned before, $\Ubf^+_\upsilon\big(\slfrakhat(n)\big)$ is a subalgebra of the Hall algebra $\Hbf^{\tor, \mathsf{tw}}_{n, p_n}$ of $\Tor_{p_n}(X_n)$. The Hall multiplication between torsion sheaves and vector bundles yields an action of $\Ubf_\upsilon(\slfrakhat(n))$ on the  Hall subalgebra $\Ubf^>_n$ of the category of vector bundles on $X_n$ (the action by \textit{Hecke operators}). By construction, the action of $\Ubf_\upsilon(\slfrakhat(n))$ on $\Ubf^>_n$ preserves the graded subspaces $\Ubf^>_n[r]$, corresponding to vector bundles on $X_n$ of fixed rank $r$. In particular, $\Ubf^>_n[1]$ is identified with the standard representation of $\Ubf_\upsilon(\slfrakhat(n))$ (in the space $\C^n[x^{\pm 1}]$) and $\Ubf^>_n[r]$ is a certain quotient of the tensor representation $\Ubf^>_n[1]^{\widehat{\otimes}\, r}$. Such a representation depends on the genus $g_X$ of $X$ and we call the \textit{$r$-th symmetric tensor representation} of $\Ubf_\upsilon(\slfrakhat(n))$.

For any positive integers $n, k$, we have morphisms $\pi_{kn, n}\colon X_{kn}\to X_n$, which form an inverse system. Now, \textit{let $n$ tend to infinity}. More precisely, consider the inverse limit of the $X_n$'s with respect to divisibility of positive integers: the resulting stack is the \textit{infinite root stack} $X_\infty$, introduced by Talpo and Vistoli \cite{art:talpovistoli2014}. Again, intuitively, one can imagine that $X_\infty$ is obtained from $X$ by ``replacing" a fixed rational point $p$ by the trivial gerbe $p_\infty\coloneqq\Bcal\mu_\infty$. The abelian category $\Coh(X_\infty)$ of coherent sheaves on $X_\infty$ is equivalent to the direct limit of categories $\Coh(X_n)$ with respect to the fully faithful functors $\pi_{kn, n}^\ast$. Thanks to this, $\Coh(X_\infty)$ exhibits some interesting new features; for instance, we can associate with any coherent sheaf on $X_\infty$ a rank and a degree, which is a linear combination of characteristic functions of intervals in $S^1_\Q$. In particular, line bundles on $X_\infty$ have \textit{rational} degrees. Moreover, it is worth mentioning that $\Coh(X_\infty)$ is equivalent to the category of parabolic sheaves with rational weights.

Let $\Tor_{p_\infty}(X_\infty)$ be the category of torsion sheaves on $X_\infty$ set-theoretically supported at $p_\infty$. Such a category is obtained as a direct limit of $\Tor_{p_n}(X_n)$ with respect to the functors $\pi_{kn, n}^\ast$. Thanks to this realization, one can show that any torsion sheaf in $\Tor_{p_\infty}(X_\infty)$ decomposes as a direct sum of torsion sheaves naturally associated with rational intervals of $S^1_\Q$, and the Euler form restricted to the Grothendieck group of $\Tor_{p_\infty}(X_\infty)$ coincides exactly with \eqref{eq:eulerform-introduction}. 

In addition, for any $k, n$ the functors $\pi_{kn, n}^\ast$ give rise to algebra embeddings $\Omega_{kn, n}\colon \Hbf^{\tor, \mathsf{tw}}_{n, p_n}\to \Hbf^{\tor, \mathsf{tw}}_{kn, p_{kn}}$, which induce a directed system. The limit of such a directed system is isomorphic (as an algebra) to the Hall algebra $\Hbf^{\tor, \mathsf{tw}}_{\infty, p_{\infty}}$ of the category $\Tor_{p_\infty}(X_\infty)$. 
\begin{theorem}\label{thm:circle-hall}
The embeddings $\Omega_{kn,n}$ restricts to embeddings 
\begin{align}
\Omega_{kn, n}\big\vert_{\Ubf^+_\upsilon(\slfrakhat(n))}\colon \Ubf^+_\upsilon\big(\slfrakhat(n))\big)\to \Ubf^+_\upsilon\big(\slfrakhat(kn)\big)\ .
\end{align}
These embeddings induce a directed system whose  limit, denoted $\Cbf_\infty$, is canonically identified with a topological sub $\widetilde \Q$-Hopf algebra of $\Hbf^{\tor, \mathsf{tw}}_{\infty, p_{\infty}}$. Moreover $\Cbf_\infty$ is isomorphic $\Ubf^+_\upsilon(\slfrak(S^1_\Q))$.
\end{theorem}
The Hopf algebra $\Cbf_\infty$ is endowed with a non degenerate Hopf pairing (the Green pairing); the reduced Drinfeld double of $\Cbf_\infty$ is isomorphic to the whole $\Ubf_\upsilon(\slfrak(S^1_\Q))$, as topological $\widetilde \Q$-Hopf algebras.

As before, the Hall multiplication between torsion sheaves and vector bundles yields an action of $\Ubf_\upsilon(\slfrak(S^1_\Q))$ on the (spherical) Hall subalgebra $\Ubf^>_\infty$ of the category of vector bundles on $X_\infty$. We have the following.
\begin{theorem} 
As a $\Ubf_\upsilon(\slfrak(S^1_\Q))$-module, $\Ubf^>_\infty[1]$ is isomorphic to $\V_{S^1_\Q}$. Moreover, the (surjective) multiplication map $\Ubf^>_\infty[1]^{\otimes\, r} \to \Ubf^>_\infty[r]$ is a $\Ubf_\upsilon(\slfrak(S^1_\Q))$-intertwiner.
\end{theorem}

We may identify $\V_{S^1_\Q}$ with $\C[x_1^{\pm 1}]\otimes V_\infty$, where $V_\infty$ is an infinite-dimensional vector space with basis $\{\vec v_x\, \vert\, x\in S^1_\Q\}$, via the assignment $\vec u_y \mapsto x^{\lfloor y \rfloor}\, \vec v_{\{y\}}$. Using this identification, we realize the above multiplication map explicitly as an appropriate shuffle product
\begin{align}
\V^{\otimes\, r} \to \mathbb{V}^{\widehat \otimes\, r}\simeq
\C[x_1^{\pm 1}, \ldots, x_r^{\pm 1}][[x_1/x_2,\ldots, x_{r-1}/x_r]]\otimes V_\infty^{\otimes\, r}\ ,
\end{align}
whose kernel involves the zeta function of $X$ as a main ingredient. We call the representation of $\Ubf_\upsilon(\slfrak(S^1_\Q))$ on $\Ubf^>_\infty[r]$ the \textit{$r$-th symmetric tensor representation of genus $g_X$}.

\subsubsection*{Some remarks concerning $\Ubf_\upsilon(\glfrak(S^1_\Q))$}

As proved by the second-named author \cite{art:schiffmann2002}, for any positive integer $n$, the whole $\Hbf^{\tor,\mathsf{tw}}_{n, p_n}$ is isomorphic to the positive part of $\Ubf_\upsilon(\widehat{\glfrak}(n))$, where the positive part of the quantum Heisenberg algebra corresponds to the center of $\Hbf^{\tor,\mathsf{tw}}_{n, p_n}$. In the limit $n \to \infty$ this quantum Heisenberg algebra still exists, but it does not lie in the center anymore. Indeed, as shown in Proposition~\ref{prop:trivialcenter}, $\Hbf^{\tor,\mathsf{tw}}_{\infty, p_\infty}$ fails to have a center. As an attempt to recover a center, we introduce in Section \ref{sec:completion} a metric completion of $\Hbf_{\infty, p_\infty}^{\mathsf{tw}, \tor}$. This can be interpreted as a first step towards the definition of $\Ubf_\upsilon(\glfrak(S^1_\Q))$.

\subsubsection*{Higher genus Fock spaces}

As we will show in a sequel to this paper, when $X=\PP^1$ the representations $\Ubf^>_n[r]$ themselves admit a stable limit as $r$ tends to infinity, which is isomorphic to the level one Fock space of $\Ubf_\upsilon(\slfrakhat(n))$, as studied in \cite{art:kashiwaramiwastern1995}; the extension of this procedure to arbitrary curves and to the stable limit as $n$ tends to infinity thus produces a family of ``higher genus" Fock spaces for $\Ubf_\upsilon(\slfrakhat(n))$ and $\Ubf_\upsilon(\slfrak(S^1_\Q))$. In the genus zero case, we expect to compare such a construction with the algebraic construction of the Fock space given in \cite{art:salaschiffmann2020}.

\subsection*{Contents of the paper}

This paper is organised as follows. In Section~\ref{sec:rootstacks} we recall the theory of root stacks and coherent sheaves over them. In Section~\ref{sec:preliminariesorbcurves}, we restrict ourselves to root stacks over a smooth projective curve $X$: we provide a characterization of the category of coherent sheaves and its (numerical) Grothendieck group. Section~\ref{sec:hallalgebran} is devoted to the study of the Hall algebra of $X_n$, in particular its Hecke algebra and its spherical subalgebra: part of the results overlaps those in \cite{art:lin2014}. In Section~\ref{sec:hallalgebramn} we compare Hall algebras of $X_n$ for different $n$'s. In Section~\ref{sec:hallalgebrainfinite}, we study in the detail the Hall algebra of $X_\infty$, its Hecke algebra and spherical subalgebra: we define $\Ubf_\upsilon(\slfrak(S^1_\Q))$ and its representations, and a metric completion of $\Hbf_{\infty, p_\infty}^{\tor, \mathsf{tw}}$. Finally, in Section~\ref{sec:comparisons}, we compare our algebras with $\Ubf_\upsilon(\slfrakhat(+\infty))$ and $\Ubf_\upsilon(\slfrakhat(\infty))$. We finish the paper with an appendix in which for completeness we reprove some of the results in \cite{art:lin2014} used in the main body of the paper. The last appendix, due to T.~Kuwagaki, is devoted to a geometric realization of $\Ubf_\upsilon(\slfrak(S^1))$.

\subsection*{Acknowledgements}

We are grateful to Mattia Talpo for many enlightening conversations about root stacks and their categories of coherent sheaves, to Jyun-Ao Lin for helpful explanations about his work and to the referee for a careful reading and useful comments. We also thank Andrea Appel and Mikhail Kapranov for helpful discussions and comments; and David Hernandez for pointing out his paper \cite{art:hernandez2011}. Last but not least, we thank Tatsuki Kuwagaki for writing the appendix \ref{sec:tatsuki}.

\subsection*{Notations and Convention}

Let $k$ be a field. Schemes and stacks will always be over $k$. Our main reference for the theory of stacks is \cite{book:laumonmoretbailly2000}. A morphism of stacks is \textit{representable} if the base change of an algebraic space is an algebraic space. An \textit{algebraic stack} is a stack (in groupoids) with a ``smooth presentation" (i.e., a representable smooth epimorphism from an algebraic space) and representable quasi-compact and separated diagonal. An algebraic stack is \textit{Deligne-Mumford} if it has a presentation which is moreover \'etale. We always assume that our algebraic stacks are of finite type over $k$.

We say that an algebraic stack $\Xscr$ has \textit{finite inertia} if the natural morphism from the inertia stack $\Ical(\Xscr)\to \Xscr$ is finite. If $\Xscr$ has finite inertia, by \cite{art:keelmori1997} it admits a coarse moduli space $\pi\colon \Xscr\to X$, which is an algebraic space; moreover $\pi$ is a proper morphism and $\pi_\ast \Ocal_{\Xscr}=\Ocal_X$. We say that $\Xscr$ is \textit{tame} if $\pi_\ast \colon \QCoh(\Xscr)\to \QCoh(X)$ is an exact functor.  In this case, $\pi$ is cohomologically affine and therefore $\pi$ is a \textit{good moduli space morphism} in the sense of Alper \cite{art:alper2013}. In particular, for any quasi-coherent sheaf $F$ on $X$ the adjunction morphism $F\to \pi_\ast\pi^\ast F$ is an isomorphism and hence the projection formula holds for $\pi$.
 
Since we will work with coherent sheaves on Noetherian algebraic stacks, we use the notions of support of a coherent sheaf, purity and torsion filtration defined in \cite[Section~2.2.6.3]{art:lieblich2007}. We shall call a zero-dimensional coherent sheaf a \textit{torsion sheaf} and a pure coherent sheaf of maximal dimension a \textit{torsion-free sheaf}. We use the letters $\Ecal, \Fcal, \Gcal, \ldots,$ for sheaves on a algebraic stack, and the letters $E, F, G, \ldots,$ for sheaves on a scheme.
 
We shall use the following group schemes:
\begin{align}
\mu_r=\mu_r(k)\coloneqq\{a\in k\,\vert\, a^r=1\} \qquad\text{and}\qquad \G_m=\G_m(k)\coloneqq\{a\in k\, \vert \, a\text{ is a unit} \}\ .
\end{align} 
 
We will fix a finite field $\F_q$ with $q$ elements. Set $\widetilde \Q =\Q[\upsilon, \upsilon^{-1}]/(\upsilon^2-q)$, and for any integer $d$, define $[d]_\upsilon\coloneqq(\upsilon^d-\upsilon^{-d})/(\upsilon-\upsilon^{-1})$. 

Finally, fix a positive integer $n$. Given an integer $d$, we define
\begin{align}
\nintpart{d}\coloneqq \left\lfloor \frac{d}{n}\right\rfloor\quad\text{and}\quad \nfrapart{d}\coloneqq n\left\{\frac{d}{n}\right\}\ ,
\end{align}
which are, respectively, the integer part and $n$ times the fractional part of $d/n$.

\bigskip\section{Root stacks}\label{sec:rootstacks}

We shall give a brief survey of the theory of root stacks over algebraic stacks as presented in \cite{art:cadman2007} (see also \cite[Appendix~B]{art:abramovichgrabervistoli2008}, \cite[Section~2.1]{art:bayercadman2010}, and \cite[Section~1.3]{art:fantechimannnironi2010}). We will give also an overview of the relation between quasi-coherent sheaves on root stacks (resp.\ the infinite root stack) and parabolic sheaves (resp.\ with parabolic sheaves with rational weights) following \cite{art:borne2007, art:bornevistoli2012, art:talpovistoli2014, art:talpo2017} (see also \cite{phd:talpo2015}).

\subsection{Preliminaries}

Let $\Xscr$ be an algebraic stack. Recall that there is an equivalence of categories between the category of line bundles on $\Xscr$ and the category of morphisms $\Xscr\to \Bscr\G_m$, where the morphisms in the former category are taken to be isomorphisms of line bundles. There is also an equivalence between the category of pairs $(\Lcal,s)$, with $\Lcal$ a line bundle on $\Xscr$ and $s$ a global section of $\Lcal$, and the category of morphisms $\Xscr\to [\A^1/\G_m]$, where $\G_m$ acts on $\A^1$ by multiplication \cite[Example~5.13]{art:olsson2003}.

Throughout this section, $\Xscr$ will be an algebraic stack, $\Lcal$ a line bundle on $\Xscr$, $s$ a global section of $\Lcal$, and $n$ a positive integer. The pair $(\Lcal,s)$ defines a morphism $\Xscr\to [\A^1/\G_m]$ as above. Let $\theta_n\colon [\A^1/\G_m]\to [\A^1/\G_m]$ be the morphism of stacks, induced by the morphisms
\begin{align}
x\in\A^1 \ \longmapsto \ x^n\in\A^1 \qquad \text{and} \qquad t\in\G_m \ \longmapsto \ t^n\in\G_m\ ,
\end{align}
which sends a pair $(\Lcal,s)$ to its $n$-th tensor power $(\Lcal^{\otimes\, n}, s^{\otimes\, n})$.
\begin{definition}\label{def:rootlines}
Let $\rootlines$ be the algebraic stack obtained as the fibre product
\begin{align}
  \begin{tikzpicture}[xscale=1.5,yscale=-1.2]
    \node (A0_0) at (0, 0) {$\rootlines$};
    \node (A0_2) at (2, 0) {$[\A^1/\G_m]$};
    \node (A1_1) at (1, 1) {$\square$};
    \node (A2_0) at (0, 2) {$\Xscr$};
    \node (A2_2) at (2, 2) {$[\A^1/\G_m]$};
    \node (Comma) at (2.7, 1) {$.$};
    \path (A0_0) edge [->]node [left] {$\scriptstyle{\pi_n}$} (A2_0);
    \path (A0_0) edge [->]node [auto] {$\scriptstyle{}$} (A0_2);
    \path (A0_2) edge [->]node [auto] {$\scriptstyle{\theta_n}$} (A2_2);
    \path (A2_0) edge [->]node [auto] {$\scriptstyle{}$} (A2_2);
  \end{tikzpicture} 
\end{align}
We say that $\rootlines$ is the \textit{root stack} obtained from $\Xscr$ by the $n$-th root construction.
\end{definition}
 
\begin{remark}
By \cite[Example~2.4.2]{art:cadman2007}, if $s$ is a nowhere vanishing section then $\rootlines\simeq \Xscr$. This shows that all of the ``new" stacky structure in $\rootlines$ is concentrated at the zero locus of $s$.
\end{remark}
\begin{definition}\label{def:rootline}
Let $\rootline$ be the algebraic stack obtained as the fibre product
\begin{align}
  \begin{tikzpicture}[xscale=1.5,yscale=-1.2]
    \node (A0_0) at (0, 0) {$\rootline$};
    \node (A0_2) at (2, 0) {$\Bscr\G_m$};
    \node (A1_1) at (1, 1) {$\square$};
    \node (A2_0) at (0, 2) {$\Xscr$};
    \node (A2_2) at (2, 2) {$\Bscr\G_m$};
    \node (Comma) at (2.6, 1) {$,$};
    \path (A0_0) edge [->]node [left] {$\scriptstyle{\rho_n}$} (A2_0);
    \path (A0_0) edge [->]node [auto] {$\scriptstyle{}$} (A0_2);
    \path (A0_2) edge [->]node [auto] {$\scriptstyle{}$} (A2_2);
    \path (A2_0) edge [->]node [auto] {$\scriptstyle{}$} (A2_2);
  \end{tikzpicture} 
\end{align}
where $\Xscr\to \Bscr\G_m$ is determined by $\Lcal$ and $\Bscr\G_m\to \Bscr\G_m$ is given by the map $t\in\G_m\mapsto t^n\in\G_m$.
\end{definition}
By construction, the morphism $\rootlines\to [\A^1/\G_m]$ corresponds to a line bundle $\Lcal_n$ on $\rootlines$ with a global section $s_n$. Moreover, there is an isomorphism $\Lcal_n^{\otimes\, n}\simeq \pi_n^\ast\Lcal$ which sends $s_n^{\otimes\, n}$ to $\pi_n^\ast s$. Similarly, the morphism $\rootline\to \Bscr\G_m$ corresponds to a line bundle $\Lcal_n$ on $\rootline$ such that $\Lcal_n^{\otimes\, n}\simeq \pi_n^\ast\Lcal$.
\begin{remark}
As described in \cite[Example~2.4.3]{art:cadman2007}, $\rootline$ is a closed substack of $\sqrt[n]{(\Lcal,0)/\Xscr}.$ In general, let $\Dscr$ be (reduced) the vanishing locus of $s\in \Gamma(\Xscr,\Lcal)$.  One has a chain of inclusions of closed substacks
\begin{align}
\sqrt[n]{\Lcal_{\vert \Dscr}/\Dscr}\ \subset \
\sqrt[n]{(\Lcal_{\vert \Dscr},0)/\Dscr}\ \subset \ \rootlines\ .
\end{align}
In addition, $\sqrt[n]{\Lcal_{\vert \Dscr}/\Dscr}$ is isomorphic to the reduced stack $\big(\, \sqrt[n]{(\Lcal_{\vert \Dscr},0)/\Dscr}\, \big)_{\mathsf{red}}$. So, $\sqrt[n]{\Lcal_{\vert \Dscr}/\Dscr}$ is an effective Cartier divisor of $\rootlines$.
\end{remark}
Locally, $\rootline$ is a quotient of $\Xscr$ by a trivial action of $\mu_n$, though this is not true globally. In general, $\rootline$ is a $\mu_n$-gerbe over $\Xscr$. Its cohomology class in the \'etale cohomology group $H^2(\Xscr;\mu_n)$ is obtained from the class $[\Lcal] \in H^1(\Xscr; \G_m)$ via the boundary homomorphism $\delta\colon$ $H^1(\Xscr; \G_m)\to H^2(\Xscr; \mu_n)$ given by the Kummer exact sequence
\begin{align}
1\ \longrightarrow \ \mu_n\ \longrightarrow \ \G_m\ \xrightarrow{(-)^n}
\ \G_m\ \longrightarrow \ 1\ . 
\end{align}
Since the class $\delta([\Lcal])$ has trivial image in $H^2(\Xscr; \G_m)$, the gerbe $\rootline$ is called \textit{essentially trivial} \cite[Definition~2.3.4.1 and Lemma~2.3.4.2]{art:lieblich2007}. The gerbe is trivial if and only if $\Lcal$ has a $n$-th root in $\Pic(\Xscr)$. More generally, the gerbe $\rootline$ is isomorphic, as a $\mu_n$-banded gerbe, to $\sqrt[n]{\Lcal'/\Xscr}$ if and only if $[\Lcal]=[\Lcal']$ in $\Pic(\Xscr)/n\Pic(\Xscr)$.

In the following we state some of the results proved in \cite[Sections~2 and 3]{art:cadman2007} and in \cite[Section~4]{art:bornevistoli2012} (see also \cite[Proposition~3.3]{art:berghluntsschnurer2016}).
\begin{proposition}
The root constructions described in Definitions~\ref{def:rootlines} and \ref{def:rootline} have the following properties:
\begin{itemize}\itemsep0.2em
\item The projections $\pi_n\colon \rootlines\to\Xscr$ and $\pi_n\colon \rootline\to\Xscr$ are faithfully flat and quasi-compact.
\item The pushforwards $(\pi_n)_\ast\colon \QCoh(\rootlines)\to \QCoh(\Xscr)$ and $(\rho_n)_\ast\colon \QCoh(\rootline)\to \QCoh(\Xscr)$ are exact; $\rootlines$ and $\rootline$ are tame provided that the same holds for $\Xscr$.
\item If $\Xscr$ is an algebraic space, then it is a coarse moduli space for both $\rootlines$ and $\rootline$ under the projections to $\Xscr$. More generally, if $\Xscr$ is an algebraic stack having a coarse moduli space $\Xscr\to X$, then the compositions $\rootlines \to \Xscr\to X$ and $\rootline\to \Xscr\to X$ are coarse moduli spaces for $\rootlines$ and $\rootline$ respectively. 
\item If $\Xscr$ is a Deligne-Mumford stack and $n$ is coprime with the characteristic of $k$, then $\rootlines$ and $\rootline$ are also Deligne-Mumford stacks. 
\end{itemize}
\end{proposition}

\begin{theorem}[{\cite[Theorem~3.1.1 and Corollary~3.1.2]{art:cadman2007}}]
Let us assume that the zero locus of $s$ is nonempty and connected. Then any line bundle $\Fcal$ on $\rootlines$ is of the form
\begin{align}
\Fcal \simeq \pi_n^\ast \Mcal \otimes \Lcal_n^{\otimes \, k}\ ,
\end{align}
where $\Mcal$ is a line bundle on $\Xscr$ and $0\leq k\leq n-1$ an integer. Moreover $s$ is unique and $\Mcal$ is unique up to isomorphism. Furthermore,
\begin{align}
(\pi_n)_\ast \big(\pi_n^\ast \Mcal \otimes \Lcal_n^{\otimes \, k} \big) \simeq \Mcal\otimes \Lcal^{\otimes\, \nintpart{k}}\ .
\end{align}
\end{theorem}
We have the following morphism of short exact sequences (cf.\ \cite[Formulas~(3) and (4)]{art:fantechimannnironi2010})
\begin{align}\label{eq:diagramPic}
\begin{aligned}
  \begin{tikzpicture}[xscale=3.5,yscale=-1.2]
    \node (A0_0) at (0.5, 0) {$0$};
    \node (A0_1) at (1, 0) {$\Z$};
    \node (A0_2) at (2,0) {$\Z$};
    \node (A0_3) at (3,0) {$\Z_n$};
    \node (A0_4) at (3.5,0)  {$0$};
    \node (A1_0) at (0.5, 1) {$0$};
    \node (A1_1) at (1, 1) {$\Pic(\Xscr)$};
    \node (A1_2) at (2,1) {$\Pic(\rootlines)$};
    \node (A1_3) at (3,1) {$\Z_n$};
    \node (A1_4) at (3.5,1)  {$0$};
   \path (A0_0) edge [->]node [auto] {$\scriptstyle{}$} (A0_1);
      \path (A0_1) edge [->]node [auto] {$\scriptstyle{\cdot n}$} (A0_2);
      \path (A0_2) edge [->]node [left] {$\scriptstyle{}$} (A0_3);
         \path (A0_3) edge [->]node [left] {$\scriptstyle{}$} (A0_4);
   \path (A1_0) edge [->]node [left] {$\scriptstyle{}$} (A1_1);
      \path (A1_1) edge [->]node [auto] {$\scriptstyle{\pi_n^\ast}$} (A1_2);
      \path (A1_2) edge [->]node [auto] {$\scriptstyle{q}$} (A1_3);
         \path (A1_3) edge [->]node [left] {$\scriptstyle{}$} (A1_4);
    \path (A0_1) edge [->]node [auto] {$\scriptstyle{}$} (A1_1);
    \path (A0_2) edge [->]node [auto] {$\scriptstyle{}$} (A1_2);
    \path (A0_3) edge [->]node [auto] {$\scriptstyle{}$} (A1_3);
  \end{tikzpicture} 
  \end{aligned}
\end{align}
where the first and the second vertical morphisms are defined by $1\mapsto \Lcal$ and $1\mapsto \Lcal_n$. Here $q$ is the map which associate with a line bundle of the form $\pi_n^\ast \Mcal \otimes \Lcal_n^{\otimes \, k}$ the \textit{remainder} $k$. 

\begin{remark}
Note that any line bundle $\Fcal$ on $\rootline$ is of the form
\begin{align}
\Fcal \simeq \pi_n^\ast \Mcal \otimes \Lcal_n^{\otimes \, k}\ ,
\end{align}
where $\Mcal$ is a line bundle on $\Xscr$ and $0\leq k\leq n-1$ an integer. Moreover $k$ is unique and $\Mcal$ is unique up to isomorphism. A diagram similar to \eqref{eq:diagramPic} holds for $\rootline$.
\end{remark}

\begin{defin}\label{def:rootdiv}
Let $\Xscr$ be a smooth algebraic stack, $\Dscr\subset \Xscr$ an effective Cartier divisor and $n$ a positive integer. We denote by $\rootdiv$ the root stack $\sqrt[n]{(\Ocal_\Xscr(\Dscr),s)/\Xscr}$, where $s$ is the \textit{canonical section} of $\Ocal_\Xscr(\Dscr)$, i.e., it is the image of $1$ with respect to the natural morphism $\Ocal_{\Xscr}\to \Ocal_{\Xscr}(\Dscr)$. \hfill$\oslash$
\end{defin}
The morphism $\rootdiv \to [\A^1/\G_m]$ corresponds to an effective divisor $\Dscr_n$, i.e., the reduced closed substack $\pi_n^{-1}(\Dscr)_{\mathsf{red}}$, where $\pi_n\colon\rootdiv \to \Xscr$ is the natural projection morphism. Moreover
\begin{align}
\Ocal_{\rootdiv}(n\, \Dscr_n)\simeq \pi_n^\ast\Ocal_{\Xscr}(\Dscr) \ .
\end{align}
Let $\Xscr$ and $\Dscr$ be smooth. By \cite[Proposition~3.9]{art:berghluntsschnurer2016}, $\rootdiv$ and $\Dscr_n$ are smooth; $\Dscr_n$ is the root stack $\sqrt[n]{\Ocal_{\Xscr}(\Dscr)\vert_{\Dscr}/\Dscr}$ and hence is a $\mu_{n}$-gerbe over $\Dscr$.

\subsection{Root stacks and logarithmic schemes}

Let $X$ be a scheme over $k$. Denote by $\mathfrak{Div}(X)$ the category whose objects are pairs $(L,s)$, consisting of a line bundle $L$ on $X$ and a global section $s$ of $L$. An arrow between $(L,s)$ and $(L',s')$ is an isomorphism of $\Ocal_X$-modules from $L$ to $L'$ carrying $s$ into $s'$. The category $\mathfrak{Div}(X)$ also has a symmetric monoidal structure given by tensor products: $(L,s)\otimes (L',s')=(L\otimes L', s\otimes s')$. The neutral element is $(\Ocal_X, 1)$. Denote by $\mathfrak{Pic}(X)$ the category of line bundles on $X$, with symmetric monoidal structure given by tensor product. In contrast with standard usage, the arrows in $\mathfrak{Pic}(X)$ are arbitrary morphisms of $\Ocal_X$-modules.

From now on, let $X$ be a Noetherian scheme over $k$, $D\subset X$ an effective Cartier divisor and $n$ a positive integer number. Moreover, let $s$ be the global section of $\Ocal_X(D)$ given as image of $1$ with respect to the natural morphism $\Ocal_X\to \Ocal_X(D)$. By using the pair $(\Ocal_X(D), s)$ and $n$ we can construct the $n$-th root stack $\schemerootdiv$ (cf.\ Definition~\ref{def:rootdiv}). By using the same data, we can define a \textit{logarithmic structure} on $X$. The assignment
\begin{align}
L_D\colon &\N \to \mathfrak{Div}(X) \ , \\
& 1 \mapsto (\Ocal_X(D), s)
\end{align}
is a symmetric monoidal functor from a monoid to a symmetric monoidal category \cite[Definition~2.1]{art:bornevistoli2012}. Here, $\N$ is considered as a discrete symmetric monoidal category: the arrows are all identities while the tensor product is the sum. By \cite[Proposition~3.21]{art:bornevistoli2012} $L_D$ defines a canonical \textit{Deligne-Faltings structure $(A_\N,L_D)$} on $X$ (i.e., a \textit{logarithmic structure} on $X$). Moreover, the inclusion $\N\to \frac{1}{n}\, \N$ is a \textit{Kummer homomorphism} \cite[Definition~4.1]{art:bornevistoli2012}, hence it defines a \textit{system of denominators} $B_{\frac{1}{n}\, \N}/A_{\N}$ \cite[Definition~4.3 and Remark~4.8]{art:bornevistoli2012}. In \cite[Section~4]{art:bornevistoli2012} the authors define the stack of roots $X(B_{\frac{1}{n}\, \N}/A_{\N})$ associated with $(A_{\N}, B_{\frac{1}{n}\, \N}, L_D)$; one has $X(B_{\frac{1}{n}\, \N}/A_{\N})\simeq \schemerootdiv$. This alternative nature of $\schemerootdiv$ allows us to describe the abelian category $\QCoh(\schemerootdiv\,)$ by means of parabolic sheaves on $X$ associated with $(A_{\N}, B_{\frac{1}{n}\, \N}, L_D)$. 

Let us introduce the category $\mathfrak{Wt}(\frac{1}{n}\N)$ of \textit{weights} associated with $\frac{1}{n}\, \N$: objects are elements of $\frac{1}{n}\, \Z$, and an arrow $q\colon a\to b$ is an element $q\in \frac{1}{n}\,\N$ such that $a+q=b$. In the following, we will consider $X$ with its logarithmic structure induced by $L_D$.
\begin{definition}
A \textit{parabolic sheaf on $X$ with denominators in $\frac{1}{n}\, \N$} is a pair $(E, \rho^E)$:
\begin{itemize}
\item[(a)] a functor $E\colon \mathfrak{Wt}(\frac{1}{n}\, \N)\to \QCoh(X)$. We denote it by $v\mapsto E_v$ at the level of objects, $a\mapsto E_a$ at the level of arrows. 
\item[(b)] for any $u\in \Z$ and $v\in \frac{1}{n}\,\Z$, an isomorphism of $\Ocal_X$-modules 
\begin{align}
\rho^E_{u,v}\colon E_{u+v}\xrightarrow{\sim} \Ocal_X(u\,D)\otimes_{\Ocal_X} E_v\ ,
\end{align}
which we will call \textit{pseudo-period isomorphism}. Moreover, these data are required to satisfy the conditions stated in \cite[Definition~5.6]{art:bornevistoli2012}.
\end{itemize}
\end{definition}
We can describe a parabolic sheaf by the following diagram
\begin{align}
\begin{aligned}
  \begin{tikzpicture}[xscale=1.5,yscale=-1.2]
    \node (A0_1) at (1, 0) {$0$};
    \node (A0_2) at (2, 0) {$\frac{1}{n}$};
    \node (A0_3) at (3, 0) {$\frac{2}{n}$};
    \node (A0_4) at (4, 0) {$\cdots$};
   \node (A0_5) at (5, 0) {$\frac{n-1}{n}$};
     \node (A0_6) at (6, 0) {$1$};    
    \node (A1_1) at (1, 1) {$E_0$};
    \node (A1_2) at (2, 1) {$E_\frac{1}{n}$};
    \node (A1_3) at (3, 1) {$E_\frac{2}{n}$};
    \node (A1_4) at (4, 1) {$\cdots$};
   \node (A1_5) at (5, 1) {$E_\frac{n-1}{n}$};
     \node (A1_6) at (7, 1) {$E_1\simeq \Ocal_X(D)\otimes_{\Ocal_X}E_0$};        
  \path (A1_1) edge [->]node [left] {$\scriptstyle{}$} (A1_2);
 \path (A1_2) edge [->]node [left] {$\scriptstyle{}$} (A1_3);
  \path (A1_5) edge [->]node [left] {$\scriptstyle{}$} (A1_6);
  \end{tikzpicture} 
  \end{aligned}
\end{align}
\begin{definition}
A parabolic sheaf $(E, \rho^E)$ on $X$ is \textit{coherent} if $E_v$ is coherent for any $v$.
\end{definition}
Parabolic sheaves form the abelian category $\mathsf{Par}(X, D, n)$ (cf.\ \cite[Section~5]{art:bornevistoli2012}). Moreover, those which are coherent form an abelian subcategory.
\begin{theorem}[{\cite[Theorem~6.1]{art:bornevistoli2012} and \cite[Proposition~1.3.8]{art:talpo2017}}]\label{thm:rootvsparabolic}
There is a canonical tensor equivalence $\Phi$ of abelian categories between the category $\QCoh(\schemerootdiv\,)$ of quasi-coherent sheaves on $\schemerootdiv$ and the category $\mathsf{Par}(X,D,n)$ of parabolic sheaves on $X$ with denominators in $\frac{1}{n}\, \N$. Moreover, the equivalence restricts to a canonical tensor equivalence between the category $\Coh(\schemerootdiv\,)$ of coherent sheaves on $\schemerootdiv$ and the category of parabolic coherent sheaves on $X$.
\end{theorem}
\begin{remark}
The equivalence $\Phi\colon \QCoh(\schemerootdiv\,)\to \mathsf{Par}(X,D,n)$ is given by
\begin{align}
\Phi(\Fcal)_v\coloneqq(\pi_n)_\ast\big(\Fcal \otimes_{\Ocal_{\schemerootdiv}} \Lcal_{n}^{\otimes\, vn}\big)
\end{align}
for $v\in \frac{1}{n}\,\Z$.
\end{remark}
\begin{definition}
A coherent parabolic sheaf $(E, \rho^E)$ on $X$ is said to be \textit{pure of dimension $d$} if $E_v$ is pure of dimension $d$ for any $v\in \frac{1}{n}\, \Z$.
\end{definition}
\begin{proposition}[{see \cite[Proposition~3.1.14]{art:talpo2017}}]
The equivalence $\Phi$ induces a canonical tensor equivalence between pure coherent sheaves of dimension $d$ on $\schemerootdiv$ and pure coherent parabolic sheaves on $X$ of dimension $d$.
\end{proposition}\label{prop:injectivity}
\begin{proposition}[{\cite[Proposition~3.1.18]{art:talpo2017}}]\label{prop:parabolicvectorbundle}
Let $(E, \rho^E)$ be a torsion free parabolic sheaf on $X$. Then for any arrow $a\to b$ in $\mathfrak{Wt}(\frac{1}{n}\, \N)$ the corresponding morphism $E_a\to E_b$ is injective.
\end{proposition}

\subsection{Infinite root stack}\label{sec:infinity}

Let $X$ be a Noetherian scheme over $k$ and $D\subset X$ an effective Cartier divisor. For any two positive integers $n,m$ with $n\vert m$, we have a morphism $\pi_{m,n}\colon \sqrt[m]{D/X} \to \schemerootdiv$. 
\begin{proposition}[{\cite[Proposition~2.2.10]{art:talpo2017}}]\label{prop:projectionformula-finite}
Let $n$ be a positive integer and let $\pi_{m, n}\colon \sqrt[m]{D/X} \to \schemerootdiv$ be the natural projection. Then
\begin{itemize}\itemsep0.2em
\item the functor ${\pi_{m, n}}_\ast\colon \QCoh(\sqrt[m]{D/X}\,)\to\QCoh(\schemerootdiv\,)$ is exact.
\item ${\pi_{m,n}}_\ast\Ocal_{\sqrt[m]{D/X}}\simeq \Ocal_{\schemerootdiv}$.
\item If $\Fcal\in \QCoh(\schemerootdiv\,)$ and $\Ecal\in \QCoh(\sqrt[m]{D/X}\,)$, we have a functorial isomorphism $\Fcal\otimes {\pi_{m, n}}_\ast \Ecal\simeq {\pi_{m, n}}_\ast\big(\pi_{m,n}^\ast \Fcal\otimes \Ecal\big)$.
\item ${\pi_{m, n}}_\ast$ sends coherent sheaves to coherent sheaves.
\end{itemize}
\end{proposition}
\begin{corollary}{\cite[Corollary~2.2.22]{art:talpo2017}}
The functor $\pi_{m,n}^\ast \colon \QCoh(\schemerootdiv\,)\to\QCoh(\sqrt[m]{D/X}\,)$ is fully faithful.
\end{corollary}
Let us see $\Z_{>0}$ as a filtered partially ordered set with the ordering given by divisibility. Then the set $\{\schemerootdiv\}_{n\in \Z_{>0}}$ of root stacks forms an \textit{inverse system} of stacks over the category $(\mathsf{Sch}/X)$ in the sense of \cite[Section~2.2]{art:talpovistoli2014}, therefore one can take the \textit{canonical inverse limit} in the sense of \cite[Definition~2.2]{art:talpovistoli2014}. 
\begin{definition}\label{def:infinitestack}{\cite[Definition~3.3 and Proposition~3.5]{art:talpovistoli2014}}
The \textit{infinite root stack} $\sqrt[\infty]{D/X}$ associated with the logarithmic structure $(A_{\N}, L_D)$ on $X$ is the inverse limit 
\begin{align}
\sqrt[\infty]{D/X} \coloneqq\lim_{\genfrac{}{}{0pt}{}{\longleftarrow}{n}}\, \sqrt[n]{D/X} \ .
\end{align}
\end{definition}
\begin{remark}
By \cite[Proposition~2.1.9]{phd:talpo2015} $\sqrt[\infty]{D/X}$ is a stack over the category $(\mathsf{Sch}/X)$ (with the fpqc topology or any coarser one). Moreover, it has a fpqc presentation $U\to \sqrt[\infty]{D/X}$.

In addition, there exists a projection morphism $\pi_\infty\colon \sqrt[\infty]{D/X} \to X$ such that for any positive integer $n$ factorizes as $\pi_\infty = \pi_n\circ \pi_{\infty, n}$, where $\pi_{\infty, n}\colon \sqrt[\infty]{D/X}\to \schemerootdiv$ is the natural projection to $\schemerootdiv$.
\end{remark}
As explained in \cite[Section~4]{art:talpovistoli2014}, since $\sqrt[\infty]{D/X}$ is only a fpqc stack with a fpqc presentation $U\to \sqrt[\infty]{D/X}$, one has to be very careful in order to define quasi-coherent sheaves on it. The definition given in \textit{loc.cit.} makes use of the fqpc presentation: a quasi-coherent sheaf on $\sqrt[\infty]{D/X}$ is a quasi-coherent sheaf on $U$ together with descent data with respect to the groupoid $U\times_{\sqrt[\infty]{D/X}} U \rightrightarrows U$ (cf.\ \cite[Proposition~4.6]{art:talpovistoli2014}). Thus we have the abelian category $\QCoh(\sqrt[\infty]{D/X}\,)$ of quasi-coherent sheaves on $\sqrt[\infty]{D/X}$.

Since all projections $\pi_{m,n}$ with $n\vert m$ are flat\footnote{This follows from the fact that the logarithmic structure of $X$ is locally free, i.e., the stalks of $A_{\N}$ are all free monoids and that the logarithmic structure comes from the map $L_\N$ (cf.\ \cite[Comment after Lemma~4.1.4]{phd:talpo2015}).}, by \cite[Proposition~4.19]{art:talpovistoli2014}, $\sqrt[\infty]{D/X}$ is coherent, hence we can talk about coherent sheaves on it: again a coherent sheaf on $\sqrt[\infty]{D/X}$ is a coherent sheaf on a fpqc presentation $U\to \sqrt[\infty]{D/X}$ together with a descent data. Coherent sheaves form the abelian category $\Coh(\sqrt[\infty]{D/X}\,)$.
\begin{proposition}[{\cite[Proposition~4.16]{art:talpovistoli2014}}]\label{prop:projectionformula}
Let $n$ be a positive integer and let $\pi_{\infty, n}\colon \sqrt[\infty]{D/X} \to \schemerootdiv$ be the natural projection. Then
\begin{itemize}\itemsep0.2em
\item the functor ${\pi_{\infty, n}}_\ast\colon \QCoh(\sqrt[\infty]{D/X}\,)\to\QCoh(\schemerootdiv\,)$ is exact.
\item ${\pi_{\infty,n}}_\ast\Ocal_{\sqrt[\infty]{D/X}}\simeq \Ocal_{\schemerootdiv}$.
\item If $\Fcal\in \QCoh(\schemerootdiv\,)$ and $\Ecal\in \QCoh(\sqrt[\infty]{D/X}\,)$, we have a functorial isomorphism $\Fcal\otimes {\pi_{\infty, n}}_\ast \Ecal\simeq {\pi_{\infty, n}}_\ast\big(\pi_{\infty,n}^\ast \Fcal\otimes \Ecal\big)$.
\item ${\pi_{\infty, n}}_\ast$ sends coherent sheaves to coherent sheaves.
\end{itemize}
\end{proposition}
\begin{corollary}
Let $n$ be a positive integer and let $\pi_{\infty, n}\colon \sqrt[\infty]{D/X} \to \schemerootdiv$ be the natural projection. Then the pullback functor $\pi_{\infty,n}^\ast$ is fully faithful.
\end{corollary}
\begin{proposition}{\cite[Proposition~6.1]{art:talpovistoli2014}}\label{prop:coherentsheavesXinfty}
The pullbacks $\pi_{m,n}^\ast\colon \Coh(\schemerootdiv\,)\to \Coh(\sqrt[m]{D/X}\,)$ fit into a directed system of categories, where $n\vert m$. Moreover, the pullbacks $\pi_{\infty, n}^\ast\colon \Coh(\schemerootdiv\,)\to \Coh(\sqrt[\infty]{D/X}\,)$ along the projections $\pi_{\infty, n}\colon \sqrt[\infty]{D/X}\to \schemerootdiv$ are compatible with the structure maps of the system, and if in addition $X$ is quasi-compact and quasi-separated the induced functor
\begin{align} 
\lim_{\genfrac{}{}{0pt}{}{\longrightarrow}{n}}\, \Coh(\schemerootdiv\,)\to \Coh(\sqrt[\infty]{D/X}\,)
\end{align} 
is an equivalence.
\end{proposition}
\begin{remark}\label{rem:coherentsheaves}
The statement that $\displaystyle\lim_{\genfrac{}{}{0pt}{}{\longrightarrow}{n}}\, \Coh(\schemerootdiv\,)\to \Coh(\sqrt[\infty]{D/X}\,)$ is an equivalence means that any coherent sheaf on $\sqrt[\infty]{D/X}$ comes \textit{uniquely} from a coherent sheaf on $\schemerootdiv$ for some $n$. Here, ``uniquely" means that if there are two such coherent sheaves on $\schemerootdiv$ and $\sqrt[m]{D/X}$ giving the same coherent sheaf on $\sqrt[\infty]{D/X}$, they are isomorphic on $\sqrt[k]{D/X}$ for some integer $k$ such that $m\vert k$ and $n\vert k$ (cf.\ the proof of \cite[Proposition~6.1]{art:talpovistoli2014}). A similar result holds for morphisms between coherent sheaves on $\sqrt[\infty]{D/X}$.
\end{remark}
Let us introduce the category $\mathfrak{Wt}(\Q)$ of \textit{weights} associated with $\Q$: objects are elements of $\Q$, and an arrow $q\colon a\to b$ is an element $q\in \Z$ such that $a+q=b$.
\begin{definition}
A \textit{parabolic sheaf on $X$ with rational weights} is a pair $(\Ecal, \rho^\Ecal)$:
\begin{itemize}
\item[(a)] a functor $\Ecal\colon \mathfrak{Wt}(\Q)\to \QCoh(X)$. We denote it by $v\mapsto \Ecal_v$ at the level of objects, $a\mapsto \Ecal_a$ at the level of arrows. 
\item[(b)] for any $u\in \Z$ and $v\in \Q$, an isomorphism of $\Ocal_X$-modules 
\begin{align}
\rho^\Ecal_{u,v}\colon \Ecal_{u+v}\xrightarrow{\sim} \Ocal_X(u\,D)\otimes_{\Ocal_X} \Ecal_v\ ,
\end{align}
which we will call \textit{pseudo-period isomorphism}. Moreover, these data are required to satisfy the conditions stated in \cite[Definition~7.1]{art:talpovistoli2014}.
\end{itemize}
\end{definition}
A morphism of parabolic sheaves with rational weights is a natural transformation compatible with the pseudo-periods isomorphism. Parabolic sheaves with rational weights form the abelian category $\mathsf{Par}(X, D, \Q)$ with a tensor product and internal Homs (cf.\ \cite[Section~2.2.3]{phd:talpo2015}). Moreover, those which are coherent form an abelian subcategory.
\begin{proposition}[{\cite[Theorem~7.3]{art:talpovistoli2014}}]
There is a canonical tensor equivalence $\Phi$ of abelian categories between the category $\QCoh(\sqrt[\infty]{D/X}\,)$ of quasi-coherent sheaves on $\sqrt[\infty]{D/X}$ and the category $\mathsf{Par}(X,D,\Q)$ of parabolic sheaves on $X$ with rational weights. 
\end{proposition}
\begin{remark}
The equivalence $\Phi\colon \QCoh(\sqrt[\infty]{D/X}\,)\to \mathsf{Par}(X,D,\Q)$ is given by
\begin{align}
\Phi(\Fcal)_v\coloneqq{\pi_\infty}_\ast\big(\Fcal \otimes_{\Ocal_{\sqrt[\infty]{D/X}}} \pi_{\infty,n}^\ast\Lcal_{n}^{\otimes\, s}\big)
\end{align}
if $v=s/n$. Note that by Proposition~\ref{prop:projectionformula} we get
\begin{align}
\Phi(\Fcal)_v={\pi_n}_\ast\circ {\pi_{\infty,n}}_\ast\big(\Fcal \otimes_{\Ocal_{\sqrt[\infty]{D/X}}} \pi_{\infty,n}^\ast\Lcal_{n}^{\otimes\, s}\big)\simeq {\pi_n}_\ast\big ({\pi_{\infty,n}}_\ast\Fcal \otimes_{\Ocal_{\schemerootdiv}} \Lcal_{n}^{\otimes\, s}\big) \ .
\end{align}
\end{remark}

\bigskip\section{Preliminaries on coherent sheaves on (infinite) root stacks over a curve}\label{sec:preliminariesorbcurves}

\subsection{$n$-th root stack}\label{sec:preliminariesorbcurves-nrootstack}

Let $X$ be a smooth geometrically connected projective curve over a field $k$. Let $p\in X$ be a closed point of degree one and $n\geq 2$ an integer. Let $\pi_n\colon X_n\to X$ be the $n$-th root stack $\sqrt[n]{(\Ocal_X(p),s)/X}$ associated with $p$, seen as an effective Cartier divisor (cf.\ Definition~\ref{def:rootdiv}). Then $X_n$ is a smooth tame algebraic stack and the functors
\begin{align}
{\pi_n}_\ast &\colon \QCoh(X_n)\to \QCoh(X)\ , \quad {\pi_n}_\ast \colon \Coh(X_n)\to \Coh(X)\ , \\
\pi_n^\ast &\colon \QCoh(X) \to \QCoh(X_n) \ , \quad \pi_n^\ast \colon \Coh(X) \to \Coh(X_n)
\end{align}
are exact.

We shall denote by $p_n$ the effective Cartier divisor $\pi_n^{-1}(p)_{\mathsf{red}}$ in $X_n$ and we denote the corresponding line bundle equivalently by $\Lcal_n$ and $\Ocal_{X_n}(p_n)$. By construction, $p_n\simeq \Bscr\mu_n\coloneqq[\Spec k/\mu_n]$ is the trivial $\mu_n$-gerbe over $p$ (here, the $\mu_n$-action on $\Spec k$ is trivial). We shall call $p_n$ the \textit{stacky point} of $X_n$.
\begin{remark}
The line bundle $\Lcal_n^{\otimes\, \ell}$ on $X_n$ corresponds to the parabolic locally free sheaf on $X$
\begin{align}
\begin{aligned}
  \begin{tikzpicture}[xscale=3.5,yscale=-1.2]
    \node (A0_1) at (2, 0) {$0$};
    \node (A0_2) at (2.5, 0) {$\cdots$};
    \node (A0_3) at (3, 0) {$\frac{n-1}{n}-\{\frac{\ell}{n}\}$};
    \node (A0_4) at (4, 0) {$1-\{\frac{\ell}{n}\}$};
    \node (A0_5) at (4.5, 0) {$\cdots$};
     \node (A0_6) at (5, 0) {$1$};    
    \node (A1_1) at (2, 1) {$\Ocal_X\big(\nintpart{\ell}\, p\big)$};
    \node (A1_2) at (2.5, 1) {$\cdots$};
    \node (A1_3) at (3, 1) {$\Ocal_X\big(\nintpart{\ell}\, p\big)$};
   \node (A1_4) at (4, 1) {$\Ocal_X\big(\big(\nintpart{\ell}+1\big)\, p\big)$};
    \node (A1_5) at (4.5, 1) {$\cdots$};
    \node (A1_6) at (5, 1) {$\Ocal_X\big(\big(\nintpart{\ell}+1\big)\, p\big)$};   
 \path (A1_3) edge [->]node [left] {$\scriptstyle{}$} (A1_4);
  \end{tikzpicture} 
  \end{aligned}
\end{align}
\end{remark}
Due to the lack of literature on coherent sheaves on tame algebraic stacks\footnote{For example, we believe that it should be possible to extend some results about duality of coherent sheaves from Deligne-Mumford stacks \cite{art:nironi2008-II} to tame algebraic stacks.}, we shall in the following use freely Theorem~\ref{thm:rootvsparabolic} and translate into the language of coherent sheaves on root stacks some of the results about parabolic coherent sheaves stated in \cite[Section~1]{art:lin2014} and references therein. For example, there exists a canonical line bundle on $X_n$ 
\begin{align}
\omega_{X_n}=\pi_n^\ast \omega_X\otimes \Lcal_n^{\otimes\, n-1}\ 
\end{align}
and Serre duality holds:
\begin{theorem}
Let $\Ecal, \Fcal$ be coherent sheaves on $X_n$. Then there exists a functorial isomorphism
\begin{align}
\Ext^1(\Ecal, \Fcal)\xrightarrow{\sim} \Hom(\Fcal, \Ecal\otimes \omega_{X_n})^\vee \ .
\end{align}
\end{theorem}
In addition, we have 
\begin{align}
\mathsf{H}^i(X_n, \Fcal)\simeq \mathsf{H}^i(X, {\pi_n}_\ast (\Fcal)) 
\end{align}
for $i=0, 1$ and any coherent sheaf $\Fcal$ on $X_n$.

Let $\Tor_x(X_n)$ be the category of torsion sheaves on $X_n$ supported at a \textit{non-stacky} point $x\in X_n$ (i.e., a point $x\in X\setminus \{p\} \simeq X_n\setminus p_n$). $\Tor_x(X_n)$ is equivalent to the category $\Tor_x(X)$ of torsion sheaves on $X$ supported at $x$ (here we are using the fact that $\pi_n$ is an isomorphism out of $p$). Therefore, $\Tor_x(X_n)$ is equivalent to the category $\Rep_{k(x)}^{\mathsf{nil}}\big(A_0^{(1)}\big)$  of nilpotent representations of the 1-loop quiver $A_0^{(1)}$ over the residue field $k(x)$. We shall denote by $\Tcal_x^{(\ell)}$ the unique indecomposable torsion sheaf of $\Tor_x(X_n)$ of length $\ell$. If fits into the short exact sequence
\begin{align}\label{eq:torsion-nonorbifold}
0\to \Ocal_{X_n}(-\ell\, x)\to \Ocal_{X_n}\to \Tcal_x^{(\ell)}\to 0\ .
\end{align}
We set $\Tcal_x\coloneqq\Tcal_x^{(1)}$. Note that $\Tcal_x^{(\ell)}\otimes \Lcal_n^{\otimes\, s}\simeq \Tcal_x^{(\ell)}$ for any $s=0, \ldots, n-1$.

Recall that coherent sheaves on the stacky point $p_n\simeq \Bscr\mu_n\coloneqq [\Spec k/\mu_n]$ correspond to parabolic torsion sheaves $T_\bullet$ on $X$ with zero morphisms $T_{i/n}\to T_{(i+1)/n}$ for $0\leq i\leq n-1$. On the other hand, consider the stack $[\Spec(k[x]_{(0)})/\mu_n]$, where the $\mu_n$-action is given by $t\cdot x\coloneqq t^{-1} x$ and $t\cdot a\coloneqq a$ for $a\in k$: coherent sheaves on $[\Spec(k[x]_{(0)})/\mu_n]$ correspond to finite-dimensional $k[x]_{(0)}$-modules together with a compatible $\Z_n$-grading, hence to parabolic torsion sheaves $T_\bullet$ on $X$ with morphisms $T_{i/n}\to T_{(i+1)/n}$, for $0\leq i\leq n-1$, corresponding to the multiplication by $x$. Thus, if we denote by $\Tor_{p_n}(X_n)$ the category of coherent sheaves on $X_n$ supported at $[\Spec(k[x]_{(0)})/\mu_n]$, $\Tor_{p_n}(X_n)$ is equivalent to the category $\Rep_k^{\mathsf{nil}}\big(A_{n-1}^{(1)}\big)$ of nilpotent representations of the cyclic quiver $A_{n-1}^{(1)}$ with $n$ vertices over $k$. Here we assume that the orientation of the arrows of $A_{n-1}^{(1)}$ is ``increasing", i.e., the arrows are $i\to (i+1)$, for any $i\in \{1, \ldots, n\}$.

Let $i\in \{1, \ldots, n\}$ and $j\in \Z_{>0}$. Define the torsion sheaf ${}_n\Scal_i^{(j)}\in \Tor_{p_n}(X_n)$ by the short exact sequence
\begin{align}\label{eq:torsion-orbifold}
0\to \Ocal_{X_n}(-i\, p_n)\to \Ocal_{X_n}((j-i)\, p_n)\to {}_n\Scal_i^{(j)}\to 0\ .
\end{align}
Set formally ${}_n\Scal_0^{(j)}\coloneqq{}_n\Scal_n^{(j)}$ for any positive integer $j$, moreover define ${}_n\Scal_i\coloneqq{}_n\Scal_i^{(1)}$ for $i=1, \ldots, n$. In addition, ${}_n\Scal_i^{(j)}\otimes \Lcal_n^{\otimes\, s}\simeq {}_n\Scal_{i-s\bmod{n}}^{(j)}$ for any $s=0, \ldots, n-1$.
\begin{remark}
Let $i\in \{1, \ldots, n\}$ and let $j\geq 2$ be an integer. Then ${}_n\Scal_i^{(j)}$ fits into the short exact sequences
\begin{align}
0\to  {}_n\Scal_i^{(j-1)} \to  {}_n\Scal_i^{(j)} \to  {}_n\Scal_{\nfrapart{i+1-j}}^{(1)}\to 0\ ,\\[4pt]
0\to  {}_n\Scal_{i}^{(1)} \to  {}_n\Scal_i^{(j)} \to  {}_n\Scal_{i-1}^{(j-1)}\to 0\ .
\end{align}
\end{remark}
\begin{remark}
The torsion sheaf ${}_n\Scal_1$ corresponds to the parabolic torsion sheaf on $X$
\begin{align}
\begin{aligned}
  \begin{tikzpicture}[xscale=2.2,yscale=-1.2]
    \node (A0_1) at (1, 0) {$0$};
    \node (A0_2) at (2, 0) {$\frac{1}{n}$};
    \node (A0_3) at (2.5, 0) {$\cdots$};    
    \node (A0_4) at (3, 0) {$\frac{n-1}{n}$};    
    \node (A0_7) at (4, 0) {$1$};    
    \node (A1_1) at (1, 1) {$T_p$};
    \node (A1_2) at (2, 1) {$0$};
    \node (A1_3) at (2.5, 1) {$\cdots$};
    \node (A1_4) at (3, 1) {$0$};
        \node (A1_5) at (4, 1) {$T_p\otimes\Ocal_X(p)$};
 \path (A1_1) edge [->]node [left] {$\scriptstyle{}$} (A1_2);
 \path (A1_2) edge [->]node [left] {$\scriptstyle{}$} (A1_3);
  \path (A1_3) edge [->]node [left] {$\scriptstyle{}$} (A1_4);
 \path (A1_4) edge [->]node [left] {$\scriptstyle{}$} (A1_5);
  \end{tikzpicture} 
  \end{aligned}
\end{align}
while, for $i\in\{2, \ldots, n\}$, the torsion sheaf ${}_n\Scal_i$ corresponds
\begin{align}
\begin{aligned}
  \begin{tikzpicture}[xscale=2.2,yscale=-1.2]
    \node (A0_1) at (2, 0) {$0$};
    \node (A0_2) at (2.5, 0) {$\cdots$};
    \node (A0_3) at (3, 0) {$\frac{i-2}{n}$};    
    \node (A0_4) at (4, 0) {$\frac{i-1}{n}$};    
    \node (A0_7) at (5, 0) {$\frac{i}{n}$};    
\node (A0_8) at (5.5, 0) {$\cdots$};    
     \node (A0_9) at (6, 0) {$1$};    
    \node (A1_1) at (2, 1) {$0$};
    \node (A1_2) at (2.5, 1) {$\cdots$};
    \node (A1_3) at (3, 1) {$0$};
    \node (A1_4) at (4, 1) {$T_p$};
        \node (A1_5) at (5, 1) {$0$};
    \node (A1_6) at (5.5, 1) {$\cdots$};    
    \node (A1_7) at (6, 1) {$0$};   
 \path (A1_1) edge [->]node [left] {$\scriptstyle{}$} (A1_2);
 \path (A1_2) edge [->]node [left] {$\scriptstyle{}$} (A1_3);
  \path (A1_3) edge [->]node [left] {$\scriptstyle{}$} (A1_4);
 \path (A1_4) edge [->]node [left] {$\scriptstyle{}$} (A1_5);
 \path (A1_5) edge [->]node [left] {$\scriptstyle{}$} (A1_6);
 \path (A1_6) edge [->]node [left] {$\scriptstyle{}$} (A1_7);
  \end{tikzpicture} 
  \end{aligned}
\end{align}
\end{remark}

We denote by $T_x^{(\ell)}$ the unique indecomposable torsion sheaf of $\Tor_x(X)$ supported at $x\in X$ of length $\ell$ and set $T_x\coloneqq T_x^{(1)}$. We have
\begin{align}
{\pi_n}_\ast \big(\Tcal_x^{(\ell)}\big) \simeq T_x^{(\ell)}\text{ for $x\neq p$,}\quad {\pi_n}_\ast \big({}_n\Scal_i^{(j)}\big)\simeq T_p^{(1+\nintpart{j-i})}\ .
\end{align}
\begin{proposition} $ $
\begin{itemize}\itemsep0.2em
\item Any torsion sheaf on $X_n$ is a direct sum of torsion sheaves of the form ${}_n\Scal_i^{(j)}$ for $i=1, \ldots, n$ and $j\in \Z_{>0}$ and torsion sheaves supported at non-stacky points of $X_n$.
\item Any coherent sheaf on $X_n$ is a direct sum of a torsion sheaf and a locally free sheaf.
\item Any locally free sheaf $\Fcal$ on $X_n$ admits a filtration
\begin{align}\label{eq:filtrationlinebundles}
0=\Fcal_0\subset \Fcal_1\subset \cdots \subset \Fcal_r=\Fcal
\end{align}
such that the factor $\Fcal_s/\Fcal_{s-1}$ is a line bundle on $\Fcal$ for $s=1, \ldots, r$.
\end{itemize}
\end{proposition}
For any locally free sheaf $\Fcal$ on $X_n$, we shall call \textit{rank} the number $r$ of nonzero factors in \eqref{eq:filtrationlinebundles}. This corresponds to the rank of ${\pi_n}_\ast \Fcal$. Moreover we shall call \textit{rank} of a coherent sheaf the rank of its locally free part. 

Thanks to the previous proposition, the Grothendieck group $\Ksf(X_n)$ of $X_n$ is generated by the classes of line bundles and torsion sheaves on $X_n$. We denote by $[\Fcal]$ the $\Ksf$-theory class of a coherent sheaf $\Fcal$ on $X_n$. 

We denote by $\Ksf^+(X_n)$ the subset of $\Ksf(X_n)$ consisting of the classes of coherent sheaves on $X_n$. The Euler form of any two coherent sheaves $\Fcal, \Ecal$ on $X_n$ is defined as
\begin{align}
\langle \Fcal, \Ecal\rangle \coloneqq\dim \Hom(\Fcal, \Ecal)-\dim \Ext^1(\Fcal, \Ecal) \ ,
\end{align}
and we extend it to $\Ksf(X_n)$ by linearity. As usual, the symmetric Euler form is $(x,y)\coloneqq\langle x, y\rangle +\langle y, x \rangle$ for any $x, y\in \Ksf(X_n)$. 

Let $\Ksfnum(X_n)$ be the \textit{numerical Grothendieck group} of $X_n$, that is the quotient of $\Ksf(X_n)$ by the kernel of $\langle\,,\,\rangle$.  We shall call \textit{class} of a coherent sheaf $\Fcal$, and denote by $\overline \Fcal$, the image of $[\Fcal]$ in $\Ksfnum(X_n)$ with respect to the projection map $\Ksf(X_n) \to \Ksfnum(X_n)$.

Let $\delta_n\coloneqq\sum_{i=1}^{n}\, \overline{{}_n\Scal_i}$ in $\Ksfnum(X_n)$, then $\delta_n =\overline{\Tcal_x}$ for any non-stacky point $x$.
\begin{lemma}
We have
\begin{align}\label{eq:line1}
&\langle \overline{\Ocal_{X_n}}, \overline{\Lcal_n^{\otimes\, i}}\rangle  =1-g_X+\delta_{i,n} \quad \text{for $i=0, 1, \ldots, n$}\ , \quad \langle \overline{\Ocal_{X_n}}, \delta_n\rangle=1 \ , \quad  \langle \delta_n, \overline{\Ocal_{X_n}}\rangle=-1\\[8pt] \label{eq:line2}
&\langle \overline{\Ocal_{X_n}}, \overline{{}_n\Scal_j}\rangle =
\begin{cases}
0 & \text{for } j=2, \ldots, n \ , \\
1 & \text{for } j=1 \ ,
\end{cases}
\ , \quad
\langle \overline{{}_n\Scal_j}, \overline{\Ocal_{X_n}}\rangle  =
\begin{cases}
-1 & \text{for } j=n\ ,\\
0 & \text{otherwise} \ ,
\end{cases}
 \\[8pt] \label{eq:line3}
&\langle \delta_n, \delta_n\rangle=0\ , \quad\langle \delta_n, \overline{{}_n\Scal_j}\rangle=0\ , \quad \langle \overline{{}_n\Scal_j}, \delta_n\rangle=0\ , \quad \langle \overline{{}_n\Scal_i}, \overline{{}_n\Scal_j}\rangle  = 
\begin{cases}
1 & \text{for } i=j \ , \\
-1 & \text{for } j=i+1\ , \\
0 & \text{otherwise} \ .
\end{cases}
\end{align}
\end{lemma}
By using short exact sequences of the form \eqref{eq:torsion-orbifold} we can derive the following equality in $\Ksfnum(X_n)$
\begin{align}
\overline{\pi_n^\ast M\otimes \Lcal_n^{\otimes\, \ell}} &=\overline{\Ocal_{X_n}}+\deg(M)\delta_n +\sum_{k=n-\ell+1}^n\, \overline{{}_n\Scal_k}
=\overline{\Ocal_{X_n}}+(1+\deg(M))\delta_n -\sum_{k=1}^{n-\ell}\, \overline{{}_n\Scal_k}\\
&=\overline{\Ocal_{X_n}}+\deg(M)\sum_{k=1}^n\, \overline{{}_n\Scal_k} +\sum_{k=n-\ell+1}^n\, \overline{{}_n\Scal_k}
\end{align}
for any line bundle $M$ on $X$ and any integer $0\leq \ell\leq n-1$. Moreover, one has in $\Ksfnum(X_n)$
\begin{align}
\overline{\Tcal_x^{(s)}}= s\deg(x)\, \delta_n= s\deg(x)\,\sum_{k=1}^n\, \overline{{}_n\Scal_k} \ ,
\end{align}
for any $s\in \Z_{>0}$ and a non-stacky point $x$. By using these equalities together with the additivity of the rank and the degree with respect to short exact sequences, we can obtain the following result.
\begin{proposition}
The class of a coherent sheaf $\Fcal$ on $X_n$ can be written as
\begin{align}
\overline{\Fcal}=\rk\big({\pi_n}_\ast\Fcal\big)\, \overline{\Ocal_{X_n}}+\sum_{i=1}^{n}\, \deg\big({\pi_n}_\ast(\Fcal\otimes \Lcal_n^{\otimes\, i-1})\big)\, \overline{{}_n\Scal_i}\ .
\end{align}
\end{proposition}
\begin{remark}
Recall that for a coherent sheaf $F$ on $X$, we have $\overline{F}=\rk(F)\, \overline{\Ocal_X}+\deg(F)\, \delta_1$, where $\delta_1$ is the class of the torsion sheaf $T_x$ at a point $x\in X$. Then $\overline{\pi_n^\ast F}=\rk(F)\, \overline{\Ocal_{X_n}}+ \deg(F)\,\delta_n=\rk(F)\, \overline{\Ocal_{X_n}}+ \deg(F)\, \sum_{k=1}^n\, \overline{{}_n\Scal_k}$.
\end{remark}
It follows from the above discussion that $\Ksfnum(X_n) \simeq \Z \oplus \Z^{n}$, with the map $\Ksf(X_n) \to \Ksfnum(X_n)$ given at the level of classes of coherent sheaves by
\begin{align}
[\Ecal] \mapsto \overline \Ecal = \big(\rk(\Ecal), \mathrm{\mathbf{deg}}\,(\Ecal)\big) \ ,
\end{align}
where
\begin{align}
\mathrm{\mathbf{deg}}\,(\Ecal)\coloneqq\Big( \deg\big({\pi_n}_\ast(\Ecal)\big), \deg\big({\pi_n}_\ast(\Ecal\otimes \Lcal_n)\big), \ldots, \deg\big({\pi_n}_\ast(\Ecal\otimes \Lcal_n^{\otimes\, n-1})\big)\Big)\ .
\end{align}
We have
\begin{align}
\overline{{}_n\Scal_i}=(0, \mathbf e_{n,i})\text{ for $i=1, \ldots, n$}\ .
\end{align}
Here $ \mathbf e_{n,i}$ denotes the $i$-th coordinate vector of $\Z^n$ for $i=1, \ldots, n$. Therefore, $\delta_n=(0, \sum_i \mathbf e_{n, i})$; we shall denote the vector $\sum_i \mathbf e_{n, i}$ by $\delta_n$ as well, any time there is no possible confusion.

Moreover, the restriction 
\begin{align}\label{eq:dim}
\underline{\dim}_n\colon\Ksf(\Tor_{p_n}(X_n))\to \Z^n
\end{align}
of the map $\Ksf(X_n) \to \Ksfnum(X_n)$ is an isomorphism, since $\Tor_{p_n}(X_n)\simeq \Rep_k^{\mathsf{nil}}\big(A_{n-1}^{(1)}\big)$ (cf.\ \cite[Corollary~3.2-(2)]{book:schiffmann2012}).
\begin{remark}
Recall that $\Ksfnum(X)=\Z^2$ and the map $\Ksf(X)\to \Ksfnum(X)$ associates with a $\Ksf$-class $[E]$ of a coherent sheaf $E$ on $X$ the pair $\overline E=(\rk(E), \deg(E))$. By using the projection formula for $\pi_n$, one can show $\overline{\pi_n^\ast E}=(\rk(E), \deg(E)\, \delta_n)$.
\end{remark}
\begin{remark}\label{rem:Ktheorylinebundle}
Let $M$ be a line bundle on $X$ and $0\leq \ell\leq n-1$ an integer. Then
\begin{multline}
\overline{\pi_n^\ast M\otimes \Lcal_n^{\otimes \, \ell}}=\big(1,  \deg(M), \ldots, \deg(M), \underbrace{\deg(M)+1, \ldots, \deg(M)+1}_{\genfrac{}{}{0pt}{}{\ell\text{\tiny -th}}{\text{\tiny times}}}\big)\\
=\Big(1, \deg(M)\delta_n+ \sum_{i=n-\ell+1}^n\, \mathbf e_{n,i}\Big)\in \Ksfnum(X_n)\ .
\end{multline}
\end{remark}
Let us introduce the following bilinear forms
\begin{align}
\langle \cdot, \cdot \rangle & \colon \Z_{\geq 0}^n\otimes \Z_{\geq 0}^n\to \Z \ ,\quad \langle \mathbf d, \mathbf e \rangle\coloneqq\sum_{i=1}^{n}\, (d_i e_i-d_i e_{i+1}) \ , \\
\big(\cdot, \cdot \big)  & \colon \Z_{\geq 0}^n\otimes \Z_{\geq 0}^n\to \Z \ ,\quad \big(\mathbf d, \mathbf e \big)\coloneqq\sum_{i=1}^{n}\, (2d_ie_i-d_{i+1}e_i-d_ie_{i+1})=\mathbf d\, \widehat A_{n-1}\, \mathbf e\ ,
\end{align}
where we identify $d_{n+1}=d_1$ and $e_{n+1}=e_1$. Here $\widehat A_{n-1}$ is the Cartan matrix of the affine $A_{n-1}^{(1)}$ Dynkin diagram.
\begin{corollary}[Riemann-Roch Formula]
Let $\Ecal,\Fcal$ be two coherent sheaves on $X_n$ with $\overline \Fcal= (r, \mathbf d)$ and $\overline \Ecal=(s, \mathbf e)$. Then
\begin{align}\label{eq:Eulerform}
\langle \Fcal, \Ecal\rangle&=rs(1-g_X)+r e_1-d_n s+\langle \mathbf d, \mathbf e \rangle\ , \\
\big( \Fcal, \Ecal\big)&=2rs(1-g_X)+r(e_1-e_n)+s(d_1-d_n)+\big(\mathbf d, \mathbf e\big) \ .
\end{align}
\end{corollary}
\begin{remark}
Recall that for two coherent sheaves $E, F$ on $X$ we have
\begin{align}
\langle F, E\rangle=\rk(F)\rk(E)\, (1-g_X)+\rk(F)\deg(E)-\rk(E)\deg(F) \ .
\end{align}
Then $\langle \pi_n^\ast F, \pi_n^\ast E\rangle=\langle F, E\rangle$.
\end{remark}

Define the locally free sheaf on $X_n$
\begin{align}
\Gcal_n\coloneqq\bigoplus_{\ell=1}^{n}\, \Lcal_n^{\otimes \, \ell-n}\simeq \bigoplus_{\ell=1}^{n}\, \Lcal_n^{\otimes \, \ell}\otimes \pi_n^\ast\Ocal_X(-p)\ .
\end{align}
It is a so-called \textit{generating sheaf}\footnote{For an overview of the theory of generating sheaves, see  \cite[Section~2]{art:nironi2008-I} , \cite[Section~2]{art:bruzzosala2015} and \cite[Section~3.2]{art:talpo2017}. Our choice of the generating sheaf differs from \cite[Section~7.2]{art:nironi2008-I} and \cite[Sections~3.2.1 and 3.2.2]{art:talpo2017} by the twisting by $\pi_n^\ast\Ocal_X(-p)$.}  for $X_n$. 
\begin{definition}
Let $\Fcal$ be a coherent sheaf on $X_n$ of class $\overline \Fcal=\big(r,\mathbf d \big)$.  We call the \textit{$n$-th degree} of $\Fcal$ the quantity 
\begin{align}
\deg_n(\Fcal)\coloneqq\frac{1}{n}\deg\big({\pi_n}_\ast(\Fcal\otimes \Gcal_n^\vee)\big)=\frac{1}{n}\sum_{\ell=1}^{n}\,\deg\big({\pi_n}_\ast(\Fcal\otimes \Lcal_n^{\otimes\, n-\ell})\big)=\frac{1}{n}\sum_{i=1}^n\, d_i \ ,
\end{align}
and the \textit{$n$-th slope} of $\Fcal$ the quantity
\begin{align}
\mu_n(\Fcal)\coloneqq\frac{\deg_n(\Fcal)}{\rk(\Fcal)}=\frac{1}{rn}\sum_{i=1}^n\, d_i \ ,
\end{align}
if $\rk(\Fcal)\neq 0$, otherwise $\mu_n(\Fcal)=\infty$.
\end{definition}
\begin{remark}\label{rem:deg}
Note that the $n$-th degree of a coherent sheaf on $X_n$ corresponds to the parabolic degree\footnote{See \cite[Section~1.2]{art:lin2014} for the definition of the parabolic degree.} of the corresponding parabolic sheaf on $X$ with weights $a_i\coloneqq(n-i)/n$ for $i=0, \ldots, n-1$.
\end{remark}
\begin{remark}
Let $M$ be a line bundle on $X$ and $0\leq \ell\leq n-1$ an integer. Then
\begin{align}
\deg_n(\pi_n^\ast M\otimes \Lcal_n^{\otimes \, \ell})=\deg(M)+ \frac{\ell}{n}\ .
\end{align}
Therefore fixing the $n$-th degree of a line bundle on $X_n$ is equivalent to fixing its class. Indeed, the class of a line bundle $\Lcal$ on $X_n$ of $n$-th degree $x\in \frac{1}{n}\,\Z$ is
\begin{align}
\overline \Lcal = \Big(1, \lfloor x \rfloor\, \delta_n +\sum_{i=n-\ell+1}^n\, \mathbf e_{n,i}\Big)
\end{align}
if $\ell=n\{x\}$.
\end{remark}

As one can expect $\deg_n(\Ocal_{X_n})=0$, moreover
\begin{align}
\deg_n(\omega_{X_n})=\deg(\omega_X)+\frac{n-1}{n}=2(g_X-1)+\frac{n-1}{n}\ .
\end{align}
Recall that if $E$ is a coherent sheaf on $X$, we have $\chi(E)=\deg(E)+\rk(E)\chi(\Ocal_X)$ where, as usual, $\chi(E)\coloneqq\langle \overline{\Ocal_X}, \overline E\rangle$. In the stacky case, in order to get a formula similar to this one, it is better to introduce the \textit{averaged Euler characteristic} for $X_n$ given by (cf.\ \cite[Section~2.9]{art:geiglelenzing1987})
\begin{align}
\chi_n(\Ecal)\coloneqq\sum_{i=0}^{n-1}\, \left\langle \overline{\Ocal_{X_n}}, \overline{\Ecal\otimes \omega_{X_n}^{\otimes\, -i}}\right\rangle
\end{align}
for any coherent sheaf $\Ecal$ on $X_n$. It is easy to show that 
\begin{align}
\chi_n(\Ocal_{X_n})=\frac{n^2+n-2g_Xn^2}{2}\ ,
\end{align}
and for any coherent sheaf $\Ecal$ on $X_n$
\begin{align}
\chi_n(\Ecal)=n\, \deg_n(\Ecal)+\rk(\Ecal)\chi_n(\Ocal_{X_n})\ .
\end{align}
\begin{definition}
We call \textit{virtual genus} of $X_n$ the quantity
\begin{align}
g_{X_n}\coloneqq1-\frac{1}{n}\chi_n(\Ocal_{X_n})=\frac{2ng_X-n+1}{2}\ .
\end{align}
\end{definition}
Similarly to the non-stacky case, we get $\deg_n(\omega_{X_n})=(1/n)(2g_{X_n}-2)$.

We finish the section by introducing the notion of semistability for coherent sheaves on $X_n$ and characterizing the category of semistable sheaves.
\begin{definition}\label{def:semistability}
Let $\Ecal$ be a coherent sheaf on $X_n$. We say that $\Ecal$ is \textit{semistable} if for any subsheaf $0\neq \Fcal\subsetneq \Ecal$, we have $\mu_n(\Fcal)\leq \mu_n(\Ecal)$. We say that $\Ecal$ is \textit{stable} if the previous inequality is always strict.
\end{definition}
\begin{remark}
Because of Remark~\ref{rem:deg}, a coherent sheaf on $X_n$ is semistable (resp.\ stable) if and only if the corresponding parabolic coherent sheaf on $X$ is semistable \footnote{See \cite[Section~1.2]{art:lin2014} for the definition of semistability for parabolic coherent sheaves.} (resp.\ stable).
\end{remark}
Let $\nu\in \Q\cup\{\infty\}$. We shall denote by $\Coh^{\mathsf{ss}, \nu}(X_n)$ the category of semistable coherent sheaves on $X_n$ of $n$-th slope $\nu$. In particular, $\Coh^{\mathsf{ss}, \infty}(X_n)=\Tor(X_n)$.
\begin{proposition}[{\cite[Proposition~1.6]{art:lin2014}}]\label{prop:cohss}
The following hold:
\begin{itemize}\itemsep0.2cm
\item $\Coh^{\mathsf{ss}, \nu}(X_n)$ is abelian, Artinian and Noetherian.
\item $\Hom(\Ecal, \Fcal)=0$ if $\Ecal\in \Coh^{\mathsf{ss}, \nu}(X_n)$, $\Fcal\in \Coh^{\mathsf{ss}, \mu}(X_n)$ and $\nu>\mu$.
\item Let $\Ecal$ be a coherent sheaf on $X_n$. Then there exists a unique filtration, called \textit{Harder-Narasimhan filtration}, 
\begin{align}
0=\Ecal_0\subsetneq \Ecal_1\subsetneq \cdots \subsetneq \Ecal_{t-1}\subsetneq \Ecal_t=\Ecal
\end{align}
such that $\mathsf{HN}_i(\Ecal)\coloneqq\Ecal_i/\Ecal_{i-1}$ is semistable for $i=1, \ldots, t$ and 
\begin{align}
\mu_n(\mathsf{HN}_1(\Ecal))> \cdots >\mu_n(\mathsf{HN}_t(\Ecal))\ .
\end{align}
\end{itemize}
\end{proposition}
The HN-type of a coherent sheaf $\Ecal$ is $\underline \alpha=(\alpha_1, \ldots, \alpha_t)$ with $\alpha_i\coloneqq\overline{\mathsf{HN}_i(\Ecal)}$ for $i=1, \ldots, t$. Note that $\sum_i\, \alpha_i=\overline \Ecal$.

Let $\nu, \nu'\in \Q\cup \{\infty\}$, $\nu\leq \nu'$, and let $\Coh^{[\nu, \nu']}(X_n)$ be the full subcategory of $\Coh(X_n)$ consisting of coherent sheaves $\Ecal$ for which $\nu\leq \mu_n(\mathsf{HN}_t(\Ecal))<\cdots <\mu_n(\mathsf{HN}_1(\Ecal))\leq \nu'$. Similarly, one can define the full subcategories $\Coh^{\geq \nu}(X_n)$, $\Coh^{>\nu}(X_n)$, $\Coh^{\leq \nu}(X_n)$, and $\Coh^{<\nu}(X_n)$. The following result follows from the same arguments as in \cite[Lemma~1.8]{art:lin2014} and \cite[Lemma~2.2]{art:burbanschiffmann2012}.
\begin{lemma}\label{lem:finiteness}
For any $\nu, \nu'\in \Q\cup \{\infty\}$, $\nu< \nu'$, the following hold.
\begin{itemize}\itemsep0.2cm
\item $\Coh^{[\nu, \nu']}(X_n)$ is stable under extensions.
\item For any $\alpha\in \Ksfnum(X_n)$, the number of isomorphism classes of coherent sheaves in $\Coh^{[\nu, \nu']}(X_n)$ (resp.\ $\Coh^{\geq \nu}(X_n)$, $\Coh^{\leq \nu}(X_n)$) of class $\alpha$ is finite.
\item Let $\alpha\in \Ksfnum(X_n)$ and $\Ecal\in \Coh(X_n)$. The number of isomorphism classes of subsheaves $\Fcal\subset \Ecal$ (resp.\ quotient sheaves $\Ecal\twoheadrightarrow \Fcal$) of class $\alpha$ is finite.
\item Let $\alpha\in \Ksfnum(X_n)$ and  $\Fcal, \Gcal\in \Coh(X_n)$. The number of pairs of subsheaves $(\Hcal\subset \Fcal, \Kcal\subset \Gcal)$ (resp.\ pairs of quotient sheaves $(\Fcal\twoheadrightarrow \Hcal, \Gcal\twoheadrightarrow \Kcal)$) satisfying $\overline \Hcal+\overline \Kcal=\alpha$ is finite.
\end{itemize}
\end{lemma}
\begin{remark}
We want to stress that one of the key ingredient to the proof of the previous lemma is that the moduli space of semistable coherent sheaves on $X_n$ of fixed rank and $n$-th degree is a scheme of finite type.
\end{remark}

\subsection{Relations between different root stacks}\label{sec:different-coh}

Let $n,k$ be positive integers and consider the root stacks $X_n$ and $X_{kn}$. First, note that $X_{kn}\simeq \sqrt[k]{(\Lcal_n,s_n)/X_n}$, where $s_n$ is the canonical section of $\Lcal_n$. Hence, the morphism $\pi_{kn}\colon X_{kn}\to X$ factors through the natural projection $\pi_{kn,n}\colon X_{kn}\to X_n$. Moreover, the following formulas hold:
\begin{align}\label{eq:sn-n}
\pi_{kn,n}^\ast\Lcal_n\simeq \Lcal_{kn}^{\otimes\, k}\quad\text{and}\quad{\pi_{kn,n}}_\ast\Lcal_{kn}^{\otimes\, \ell}\simeq \Lcal_{n}^{\otimes\, {}_k\lfloor\ell\rfloor}\ .
\end{align}
Consider the torsion sheaf ${}_n\Scal_i^{(j)}\in \Tor_{p_n}(X_n)$, for $i= 1, \ldots, n$ and $j\in\Z_{>0}$, and the torsion sheaf ${}_{kn}\Scal_h^{(\ell)}\in \Tor_{p_{kn}}(X_{kn})$, for $h=1, \ldots, kn$ and $\ell\in\Z_{>0}$. By applying the exact functors $\pi_{kn,n}^\ast$ and ${\pi_{kn,n}}_\ast$ to the short exact sequence \eqref{eq:torsion-orbifold} we get
\begin{align}\label{eq:simple-independent}
\pi_{kn,n}^\ast \big({}_n\Scal_i^{(j)}\big)\simeq {}_{kn}\Scal_{ki}^{(kj)}\quad\text{and}\quad {\pi_{kn,n}}_\ast \big({}_{kn}\Scal_h^{(\ell)}\big)\simeq {}_n\Scal_{-{}_k\lfloor -h\rfloor}^{({}_k \lfloor \ell-h\rfloor-{}_k \lfloor -h\rfloor)}\ .
\end{align}
Here the exactness of $\pi_{kn,n}^\ast$ follows from the flatness of $\pi_{kn,n}$, since it can be interpreted as a root stack morphism (cf.\ \cite[Proposition~2.2.9]{art:talpovistoli2014}). More generally, a coherent sheaf $\Fcal$ on $X_n$, which corresponds to the parabolic sheaf on $X$
\begin{align}
\begin{aligned}
  \begin{tikzpicture}[xscale=1.5,yscale=-1.2]
    \node (A0_1) at (1, 0) {$0$};
    \node (A0_2) at (2, 0) {$\frac{1}{n}$};
    \node (A0_3) at (3, 0) {$\frac{2}{n}$};
    \node (A0_4) at (4, 0) {$\cdots$};
   \node (A0_5) at (5, 0) {$\frac{n-1}{n}$};
     \node (A0_6) at (6, 0) {$1$};    
    \node (A1_1) at (1, 1) {$F_0$};
    \node (A1_2) at (2, 1) {$F_\frac{1}{n}$};
    \node (A1_3) at (3, 1) {$F_\frac{2}{n}$};
    \node (A1_4) at (4, 1) {$\cdots$};
   \node (A1_5) at (5, 1) {$F_\frac{n-1}{n}$};
     \node (A1_6) at (6, 1) {$F_1$};        
  \path (A1_1) edge [->]node [auto] {$\scriptstyle{f_0}$} (A1_2);
 \path (A1_2) edge [->]node [auto] {$\scriptstyle{f_1}$} (A1_3);
  \path (A1_5) edge [->]node [auto] {$\scriptstyle{f_{n-1}}$} (A1_6);
  \end{tikzpicture} 
  \end{aligned}
\end{align}
is sent, via $\pi_{kn,n}^\ast$, to the coherent sheaf on $X_{kn}$ corresponding to the following parabolic sheaf on $X$:
\begin{multline}
 \begin{tikzpicture}[xscale=1.6,yscale=-1.2]
\node (A0_1) at (1, 0) {$0$};
\node (A0_2) at (2, 0) {$\frac{1}{kn}$};
\node (A0_3) at (3, 0) {$\frac{2}{kn}$};
\node (A0_4) at (3.5, 0) {$\cdots$};
\node (A0_5) at (4, 0) {$\frac{k-1}{kn}$};
\node (A0_6) at (5, 0) {$\frac{1}{n}$};
\node (A0_7) at (6, 0) {$\frac{k+1}{kn}$};
    \node (A0_8) at (6.5, 0) {$\cdots$};
\node (A1_1) at (1, 1) {$F_0$};
\node (A1_2) at (2, 1) {$F_0$};
\node (A1_3) at (3, 1) {$F_0$};
\node (A1_4) at (3.5, 1) {$\cdots$};
\node (A1_5) at (4, 1) {$F_0$};
\node (A1_6) at (5, 1) {$F_\frac{1}{n}$};
\node (A1_7) at (6, 1) {$F_\frac{1}{n}$};
    \node (A1_8) at (6.5, 1) {$\cdots$};
\path (A1_1) edge [->]node [auto] {$\scriptstyle{\mathsf{id}}$} (A1_2);
\path (A1_2) edge [->]node [auto] {$\scriptstyle{\mathsf{id}}$} (A1_3);
\path (A1_2) edge [->]node [auto] {$\scriptstyle{\mathsf{id}}$} (A1_3);
\path (A1_5) edge [->]node [auto] {$\scriptstyle{f_{1}}$} (A1_6);
\path (A1_6) edge [->]node [auto] {$\scriptstyle{\mathsf{id}}$} (A1_7);
\end{tikzpicture} \\
  \begin{tikzpicture}[xscale=1.6,yscale=-1.2]
    \node (A0_8) at (6.5, 0) {$\cdots$};
    \node (A0_9) at (7, 0) {$\frac{kn-k-1}{kn}$};
    \node (A0_10) at (8, 0) {$\frac{n-1}{n}$};
    \node (A0_11) at (9, 0) {$\frac{kn-k+1}{kn}$};
     \node (A0_12) at (9.5, 0) {$\cdots$};
   \node (A0_13) at (10, 0) {$\frac{kn-1}{kn}$};
     \node (A0_14) at (11, 0) {$1$};    
    \node (A1_8) at (6.5, 1) {$\cdots$};
    \node (A1_9) at (7, 1) {$F_\frac{n-2}{n}$};
    \node (A1_10) at (8, 1) {$F_\frac{n-1}{n}$};
    \node (A1_11) at (9, 1) {$F_\frac{n-1}{n}$};
     \node (A1_12) at (9.5, 1) {$\cdots$};
   \node (A1_13) at (10, 1) {$F_\frac{n-1}{n}$};
     \node (A1_14) at (11, 1) {$F_1$};          
 \path (A1_9) edge [->]node [auto] {$\scriptstyle{f_{n-2}}$} (A1_10);
 \path (A1_10) edge [->]node [auto] {$\scriptstyle{\mathsf{id}}$} (A1_11);
 \path (A1_13) edge [->]node [auto] {$\scriptstyle{f_{n-1}}$} (A1_14);
  \end{tikzpicture} 
\end{multline}
Here, we used the Projection Formula \eqref{prop:projectionformula-finite} and Formula \eqref{eq:sn-n}. Thus, if $\overline \Fcal=\big(r,\mathbf d \big)$, we get
\begin{align}
\overline{\pi_{kn,n}^\ast \Fcal}=&(\rk(\pi_{kn,n}^\ast\Fcal), \mathrm{\mathbf{deg}}\,(\pi_{kn,n}^\ast\Fcal))\\[2pt]
 =&\Big(\rk\big({\pi_{kn}}_\ast(\pi_{kn,n}^\ast\Fcal)\big), \deg\big({\pi_{kn}}_\ast(\pi_{kn,n}^\ast\Fcal)\big), \\
& \deg\big({\pi_{kn}}_\ast((\pi_{kn,n}^\ast\Fcal)\otimes \Lcal_{kn})\big), \ldots, \deg\big({\pi_{kn}}_\ast((\pi_{kn,n}^\ast\Fcal)\otimes \Lcal_{kn}^{\otimes\, kn-1})\big) \Big) \\[2pt]
=&\Big(r, \underbrace{d_1, \ldots, d_1}_{\text{\tiny $k$ times}}, \ldots, \underbrace{d_{n-1}, \ldots, d_{n-1}}_{\text{\tiny $k$ times}}, \underbrace{d_n, \ldots, d_n}_{\text{\tiny $k$ times}}\Big)\ .
\end{align}
Then (cf.\ \cite[Proposition~4.2.11]{art:talpo2017})
\begin{align}\label{eq:rank-independent}
\rk(\pi_{kn,n}^\ast\Fcal)&=\rk(\Fcal)\ , \\ \label{eq:degree-independent}
\deg_{kn}(\pi_{kn,n}^\ast\Fcal)&=\deg_n(\Fcal)\ .
\end{align}
If $\Ecal$ and $\Fcal$ are coherent sheaves on $X_n$, we have
\begin{align}
\Hom(\pi_{kn,n}^\ast \Ecal, \pi_{kn,n}^\ast \Fcal) \simeq \Hom(\Ecal, \Fcal)
\end{align}
since $\pi_{kn,n}$ is a fully faithful functor, and
\begin{align}\label{eq:Ext1different}
\begin{split}
&\Ext^1(\pi_{kn,n}^\ast \Ecal, \pi_{kn,n}^\ast \Fcal)\simeq \Hom(\pi_{kn,n}^\ast \Fcal, \pi_{kn,n}^\ast \Ecal\otimes \omega_{X_{kn}})\\
 &\simeq \Hom(\Fcal, (\pi_{kn,n})_\ast(\pi_{kn,n}^\ast \Ecal\otimes \omega_{X_{kn}}))\simeq \Hom(\Fcal, \Ecal\otimes \omega_{X_n})\simeq \Ext^1(\Ecal, \Fcal)\ .
\end{split}
\end{align}
Thus $\langle \pi_{kn,n}^\ast\Ecal, \pi_{kn,n}^\ast\Fcal\rangle= \langle \Ecal, \Fcal\rangle$.

\subsection{Infinite root stack}\label{sec:infiniterootstack}

Let $X_\infty$ be the inverse limit $\displaystyle \lim_{\genfrac{}{}{0pt}{}{\longleftarrow}{n}}\, X_n$ (cf.\ Definition~\ref{def:infinitestack}). The stacky structure of $X_\infty$ is concentrated at the trivial gerbe $\Bcal\mu_\infty$ over $p$, where $\displaystyle\mu_\infty\coloneqq \lim_{\genfrac{}{}{0pt}{}{\longleftarrow}{n}}\, \mu_n$. 

By Proposition~\ref{prop:coherentsheavesXinfty}, we have an equivalence
\begin{align}\label{eq:Cohinfty}
\lim_{\genfrac{}{}{0pt}{}{\longrightarrow}{n}}\, \Coh(X_n)\to \Coh(X_\infty)\ ,
\end{align} 
where the direct limit on the left-hand-side is obtained via the pullbacks $\pi_{m,n}^\ast\colon \Coh(X_n)\to \Coh(X_m)$ with $n\vert m$. In particular, any coherent sheaf on $X_\infty$ comes ''{uniquely}'' from a coherent sheaf on $X_n$ for some $n$. Here, ``uniquely" means that if there are two such coherent sheaves on $X_n$ and $X_m$ giving the same coherent sheaf on $X_\infty$, they are isomorphic on $X_k$ for some integer $k$ such that $m\vert k$ and $n\vert k$. A similar result holds for morphisms between coherent sheaves on $X_\infty$. 

Define $\Tor_{p_\infty}(X_\infty)$ as the direct limit of all the $\Tor_{p_n}(X_n)$'s. Recall that for any $n\geq 2$ the abelian category $\Tor_{p_n}(X_n)$ is equivalent to the abelian category $\Rep_k^{\mathsf{nil}}\big(A_{n-1}^{(1)}\big)$ of nilpotent representations of the cyclic quiver $A_{n-1}^{(1)}$ of with $n$ vertices over $k$, while for $n=1$, the abelian category $\Tor_{p_1}(X_1)\coloneqq\Tor_{p}(X)$ is equivalent to the abelian category $\Rep_k^{\mathsf{nil}}\big(A_0^{(1)}\big)$ of nilpotent representations of the 1-loop quiver $A_0^{(1)}$ over $k$, hence $\Tor_{p_\infty}(X_\infty) = \displaystyle\lim_{\genfrac{}{}{0pt}{}{\longrightarrow}{n}}\, \Rep_k^{\mathsf{nil}}(A_{n-1}^{(1)})$. For any half-open interval $J=[a,b[\subsetneq S^1_\Q$, with $a<b$, define
\begin{align}\label{eq:SJ}
\Scal_{J}\coloneqq\pi_{\infty, n}^\ast \big({}_n\Scal_{i}^{(k)}\big)\in \Tor_{p_\infty}(X_\infty)\ ,
\end{align}
where $a=h/n$, $b=i/n$ and $k=n(b-a)$.  For simplicity, we will call the half-open intervals as above \textit{strictly rational intervals in $S^1_\Q$}. we will drop the adjective \textit{strict} to mean that we allow $J$ to be equal to $S^1_\Q$.

For any $x\in \Q\cap [0, 1[$, define $\Lcal^{\otimes\, x}\coloneqq\pi_{\infty, n}^\ast\big(\Lcal_n^{\otimes\, i}\big)\in \Coh(X_\infty)$, if $x=i/n$ with $\mathsf{g.c.d.}(i,n)=1$.

The above description of $\Coh(X_\infty)$ implies that
\begin{align} 
\lim_{\genfrac{}{}{0pt}{}{\longrightarrow}{n}}\, \Ksf(X_n)\simeq \Ksf(X_\infty)\ .
\end{align} 
In the following, we shall give two equivalent definitions of the numerical $\Ksf$-theory of $X_\infty$. First, note that Formula \eqref{eq:rank-independent} shows that the rank is a well defined quantity for any coherent sheaf on $X_\infty$. Define
\begin{align}
\Ksfnum(X_\infty)\coloneqq\lim_{\genfrac{}{}{0pt}{}{\to}{n}} \Ksfnum(X_n)=\lim_{\genfrac{}{}{0pt}{}{\to}{n}} \big(\Z\oplus\Z^n\big) \ ,
\end{align}
where the directed set is $\Z_{>0}$ with the ordering given by divisibility and the maps $\Z\oplus\Z^n \to \Z\oplus \Z^m$ for $n \vert m$ are given by $\big(r, \mathbf d\big)\mapsto \big(r, \Phi_{m,n}(\mathbf d)\big)$, where
\begin{align}
\Phi_{m,n}(\mathbf d)\coloneqq d_1\, \sum_{j=1}^{m/n}\, \mathbf e_{m, j} + d_2\, \sum_{j=m/n+1}^{2m/n}\, \mathbf e_{m, j}+\cdots + d_n\, \sum_{j=(n-1)m/n+1}^{m}\, \mathbf e_{m, j}\ .
\end{align}
Thus we have a map $\Ksf(X_\infty) \to \Ksfnum(X_\infty)$ given at the level of classes of coherent sheaves by
\begin{align}
[\Fcal]  \mapsto \overline \Fcal=\big(\rk(\Fcal), \Phi_{m,n}(\mathrm{\mathbf{deg}}\, (\Fcal_n))\big)_{n\vert m} \in \lim_{\genfrac{}{}{0pt}{}{\to}{m}} \big(\Z\oplus\Z^m\big)
\end{align}
for any $\Fcal\simeq \pi_{\infty,n}^\ast(\Fcal_n)$ for some integer $n$ and some coherent sheaf $\Fcal_n \in \Coh(X_n)$. Moreover, the Euler form is defined on $X_\infty$ as
\begin{align}
\langle \Fcal, \Ecal\rangle = \langle \pi_{m,n}^\ast\Fcal_n, \pi_{m,h}^\ast\Ecal_h\rangle = rs(1-g_X)+r e_1-d_n s+\langle \Phi_{m,n}(\mathbf d), \Phi_{m,h}(\mathbf e) \rangle\ ,
\end{align}
for any pair $\Fcal\simeq\pi_{\infty,n}^\ast(\Fcal_n), \Ecal\simeq \pi_{\infty,h}^\ast(\Ecal_h)$ and any positive integer $m$ such that $n \vert m, h \vert m$, where $\overline \Fcal_n= (r, \mathbf d)$ and $\overline \Ecal_h=(s, \mathbf e)$. Here we used Formula \eqref{eq:Eulerform}.

Let us give an equivalent definition of $\Ksfnum(X_\infty)$. Let $\Z_0^{S^1_\Q}\subset \Z^{S^1_\Q}$ be the group consisting of functions $f\colon S^1_\Q \to \Z$ for which there exists a positive integer $n$ such that $f$ is constant on all intervals of the form $[\frac{i-1}{n},\frac{i}{n}[$ with $i\in\Z_{\geq 0}$. We define $\N_0^{S^1_\Q}$ in the same fashion. For any rational interval $J$ in $S^1_\Q$, we denote by $\chi_J$ its characteristic function. Denote by $\delta$ the characteristic function $\chi_{S^1_\Q}$ of $S^1_\Q$. 

We claim that
\begin{align}
\lim_{\genfrac{}{}{0pt}{}{\to}{n}}\big(\Z\oplus\Z^n\big)\simeq \Z\oplus \Z_0^{S^1_\Q}\ .
\end{align}
Indeed, the map $\Z \oplus \Z^{\oplus n} \to \Z \oplus \Z^{S^1_\Q}_0$ defined by $(r,\mathbf{d}) \mapsto (r,f_{\mathbf{d}})$, where
\begin{align}\label{eq:fd}
f_{\mathbf d}\coloneqq\sum_{i=1}^{n}\, d_i\chi_{[\frac{i-1}{n}, \frac{i}{n}[}\ 
\end{align}
gives rise to a homomorphism $\displaystyle \lim_{\genfrac{}{}{0pt}{}{\to}{n}} \big(\Z\oplus\Z^n\big) \to \Z \oplus \Z_0^{S^1_\Q}$, whose inverse is the map $(r,f) \mapsto (r,\mathbf{d}_f)$ where if $f$ is constant on all intervals of the form $[\frac{i-1}{n},\frac{i}{n}[$ with $i\in\Z_{\geq 0}$ for some positive integer $n\geq 2$ then
\begin{align}\label{eq:df}
\mathbf d_f\coloneqq(f(0), f(1/n), \ldots, f((n-1)/n))\ .
\end{align}

In the latter description of $\Ksfnum(X_\infty)$, by \eqref{eq:Eulerform}, the Euler form is
\begin{align}
\langle (r, f), (s, g)\rangle =rs(1-g_X)+rg_{+}(0)-sf_{-}(0)+\sum_x\, f_{-}(x)(g_{-}(x)-g_{+}(x))\ ,
\end{align}
where we have set $h_{\pm}(x)\coloneqq\lim_{t \to 0, t >0} h(x\pm t)$. As usual, $( (r, f), (s, g))=\langle (r, f), (s, g)\rangle+\langle (s, g), (r, f)\rangle$. In addition, we set $\langle f, g\rangle \coloneqq\langle (0, f), (0, g)\rangle $ and $(f, g)\coloneqq\langle f, g\rangle+\langle g, f\rangle$.
\begin{remark}
Let $\underline{\dim}\colon \Ksf(\Tor_{p_\infty}(X_\infty))\to \Z_0^{S^1_\Q}$ be the restriction of the map $\Ksf(X_\infty)\to \Ksfnum(X_\infty)$. The map $\underline{\dim}$ is an isomorphism. Indeed, for any $n\in \Z_{\geq 2}$, the map $\underline{\dim}_n$ introduced in Formula \eqref{eq:dim} is an isomorphism; moreover, the exact functor $\pi_{\infty, n}^\ast\colon \Tor_{p_n}(X_n)\to \Tor_{p_\infty}(X_\infty)$ induces a corresponding morphism $\pi_{\infty, n}^\ast$ in $\Ksf$-theory and the composition
\begin{align}
\Z^n\simeq \Z^{\Z/n\Z}\simeq \Ksf(\Tor_{p_n}(X_n))\xrightarrow{\pi_{\infty, n}^\ast} \Ksf(X_\infty)\xrightarrow{\underline{\dim}} \Z^{S^1_\Q}
\end{align}
sends $\mathbf d$ to $f_{\mathbf d}$ (the latter is defined in \eqref{eq:fd}). It is clear that if we take the direct limit, we obtain an isomorphism, which coincides with $\underline{\dim}$.

Let $n\in \Z_{\geq 2}$ and $1\leq i\leq n$. Then $\underline{\dim}\big(\Scal_{[(i-1)/n,\, i/n[}\big)=\chi_{[(i-1)/n,\, i/n[}$.
\end{remark}

Finally, by Formula \eqref{eq:degree-independent} we can define the \textit{$\infty$-degree} of a coherent sheaf $\Ecal$ on $X_\infty$ as
\begin{align}
\deg(\Ecal)\coloneqq\deg_n(\Ecal_n)\in \Q \ ,
\end{align}
where $\Ecal_n\in \Coh(X_n)$ is a \textit{representative} of $\Ecal$, i.e., $\pi_{\infty, n}^\ast \Ecal_n\simeq \Ecal$.
\begin{remark}
By Remark~\ref{rem:Ktheorylinebundle}, fixing the class of a line bundle on $X_\infty$ is equivalent to fixing its degree. 
The class of a line bundle $\mathcal{L}$ of $\infty$-degree $x \in \Q$ is equal to 
\begin{align}
\overline \Lcal=\big(1,\lfloor x \rfloor \delta + \chi_{[-\ell/n,1[}\big)
\end{align}
if $\{x\}=\ell/n$ and $\mathsf{gcd}(\ell, n)=1$. 
\end{remark}

For a coherent sheaf $\Ecal$ on $X_\infty$ we can define its \textit{$\infty$-slope} as
\begin{align}
\mu(\Ecal)\coloneqq\frac{\deg(\Ecal)}{\rk(\Ecal)}
\end{align}
if $\rk(\Ecal)\neq 0$, otherwise $\mu(\Ecal)\coloneqq\infty$. We can introduce the notion of (semi)stability for coherent sheaves on $X_\infty$ as in Definition~\ref{def:semistability} and therefore the subcategory $\Coh^{\mathsf{ss}, \nu}(X_\infty)$ of $\Coh(X_\infty)$ consisting of semistable sheaves on $X_\infty$ of slope $\nu$. By \cite[Proposition~4.3.3]{art:talpo2017}, a coherent sheaf $\Ecal$ on $X_\infty$ is semistable if and only if $\Ecal_n$ is semistable, where $\Ecal_n\in \Coh(X_n)$ is a representative of $\Ecal$. Moreover, by \cite[Proposition~4.2.11]{art:talpo2017}, if $\Fcal$ is a semistable sheaf on $X_n$, $\pi_{kn,n}^\ast \Fcal$ is semistable as well for any positive integer $k$. Thus we have
\begin{align}
\Coh^{\mathsf{ss}, \nu}(X_\infty) \simeq \lim_{\genfrac{}{}{0pt}{}{\to}{n}}\, \Coh^{\mathsf{ss}, \nu}(X_n)\ .
\end{align}
\begin{remark}\label{rem:nonfinitetype}
A proposition analogous to Proposition~\ref{prop:cohss} can be stated for $X_\infty$, as we can introduce the subcategories $\Coh^{[\nu, \nu']}(X_\infty)$, $\Coh^{\geq \nu}(X_\infty)$, $\Coh^{>\nu}(X_\infty)$, $\Coh^{\leq \nu}(X_\infty)$, and $\Coh^{<\nu}(X_\infty)$ for $\nu, \nu'\in \Q\cup\{\infty\}$, with $\nu\leq \nu'$. 
\end{remark}

\begin{lemma}\label{lem:finiteness-infty-one}
Let $(r,f)\in \Ksfnum(X_\infty)$ and $\Fcal\in \Coh(X_\infty)$. There exists $m$ (depending on $\Fcal$ and $f$) such that any subsheaf $\Gcal\subset \Fcal$ of class $(r,f)$ has a representative in $\Coh(X_m)$.
\end{lemma}
\begin{proof}
To see this, we will use the language of parabolic sheaves. Fix $n$ such that $\Fcal \in \Coh(X_n)$.  For any multiple $k$ of $n$, the sheaf $\pi_{k,n}^\ast(\Fcal)\in \Coh(X_k)$ admits a presentation as a parabolic sheaf
\begin{align}
\begin{aligned}
  \begin{tikzpicture}[xscale=1.6,yscale=-1.2]
    \node (A0_1) at (1, 0) {$0$};
    \node (A0_2) at (2, 0) {$\frac{1}{k}$};
    \node (A0_3) at (3, 0) {$\frac{2}{k}$};
    \node (A0_4) at (3.5, 0) {$\cdots$};
    \node (A0_5) at (4, 0) {$\frac{1}{n}-\frac{1}{k}$};
    \node (A0_6) at (5, 0) {$\frac{1}{n}$};
    \node (A0_7) at (6, 0) {$\frac{1}{n}+\frac{1}{k}$};
    \node (A0_8) at (6.5, 0) {$\cdots$};
    \node (A1_1) at (1, 1) {$F_0$};
    \node (A1_2) at (2, 1) {$F_{\frac{1}{k}}$};
    \node (A1_3) at (3, 1) {$F_{\frac{2}{k}}$};
    \node (A1_4) at (3.5, 1) {$\cdots$};
   \node (A1_5) at (4, 1) {$F_{\frac{k-n}{kn}}$};
    \node (A1_6) at (5, 1) {$F_\frac{1}{n}$};
    \node (A1_7) at (6, 1) {$F_\frac{k+n}{kn}$};
    \node (A1_8) at (6.5, 1) {$\cdots$};
  \path (A1_1) edge [->]node [auto] {$\scriptstyle{\mathsf{id}}$} (A1_2);
 \path (A1_2) edge [->]node [auto] {$\scriptstyle{\mathsf{id}}$} (A1_3);
 \path (A1_2) edge [->]node [auto] {$\scriptstyle{\mathsf{id}}$} (A1_3);
  \path (A1_5) edge [->]node [auto] {$\scriptstyle{}$} (A1_6);
   \path (A1_6) edge [->]node [auto] {$\scriptstyle{\mathsf{id}}$} (A1_7);
  \end{tikzpicture} 
  \end{aligned}
\end{align}
in which all the maps arriving to a label $x \not\in \frac{1}{n}\mathbb{Z}$ are isomorphisms. There exists $m$ such that $n | m$ and the function $f$ is locally constant on any interval of the form $[\frac{i}{m}, \frac{i+1}{m}[$. For any multiple $k$ of $m$, any $\Gcal \in \Coh(X_k)$ of class $(r,f)$ admits a presentation as a parabolic sheaf $(\Gcal_{x})_{x \in \frac{1}{k}\Z}$ such that $\overline{\Gcal_{x-1/k}}=\overline{\Gcal_{x}}$ if $x \notin \frac{1}{m}\Z$. Hence if $\Gcal \subset \Fcal$ the morphism $\Gcal_{x-1/k} \to \Gcal_{x}$ is an isomorphism as soon as $x \notin \frac{1}{m}\Z$. This means that there exists $\Gcal' \in \Coh(X_m)$ such that $\Gcal \simeq \pi_{k,m}^\ast(\Gcal')$.
\end{proof}

\begin{corollary}\label{cor:finiteness-infty-one}
Let $(r,f)\in \Ksfnum(X_\infty)$ and $\Fcal\in \Coh(X_\infty)$. The number of isomorphism classes of subsheaves $\Gcal\subset \Fcal$ of class $(r,f)$ is finite.
\end{corollary}
\begin{proof} By the above Lemma,  any subsheaf of $\Fcal$ of class $(r,f)$ comes from $\Coh(X_m)$ hence by Lemma~\ref{lem:finiteness} there are only finitely many of them.
\end{proof}

\begin{corollary}\label{cor:finiteness-infty-two}
Let $(r,f)\in \Ksfnum(X_\infty)$ and  $\Fcal, \Gcal\in \Coh(X_\infty)$. The number of pairs of subsheaves $(\Hcal\subset \Fcal, \Kcal\subset \Gcal)$ (resp.\ pairs of quotient sheaves $(\Fcal\twoheadrightarrow \Hcal, \Gcal\twoheadrightarrow \Kcal)$) satisfying $\overline \Hcal+\overline \Kcal=(r,f)$ is finite.
\end{corollary}
\begin{proof}
This follows directly from the previous Corollary applied to the sheaf
$\Fcal \oplus \Gcal$ together with the fact that any sheaf $\Lcal \in \Coh(X_m)$ can admit only finitely decompositions (up to isomorphism) as a direct sum $\Lcal=\Hcal \oplus \Kcal$.
\end{proof}

We finish this section, with some results which will be useful later.
\begin{proposition}\label{prop:representative-infty}
Let $(r,f)\in \Ksfnum(X_\infty)$. Let $n$ be the minimum integer for which $f$ is constant on intervals of the form $[i/n, (i+1)/n[$ for $i=0, \ldots, n-1$. Then any locally free sheaf of class $(r,f)$ has a representative in $X_n$.
\end{proposition} 
\begin{proof}
Let $\Fcal\in\Coh(X_\infty)$ be a locally free sheaf on $X_\infty$ with $\overline \Fcal=(r, f)$. Assume that $\Fcal\simeq \pi_{\infty, \ell n}^\ast(\widetilde \Fcal)$ for some $\widetilde \Fcal\in \Coh(X_{\ell n})$, with $\ell\geq 1$. Note that
\begin{align}
\mathrm{\mathbf{deg}}\,(\widetilde \Fcal)={}_{\ell n}\mathbf d_f=\Phi_{\ell n, n}({}_n \mathbf d_f)\ .
\end{align}
Hence
\begin{align}
\deg\big({\pi_{\ell n}}_\ast\big(\widetilde \Fcal\otimes \Lcal_{\ell n}^{\otimes \, s\ell +1}\,\big)\big)=\deg\big({\pi_{\ell n}}_\ast\big(\widetilde \Fcal\otimes \Lcal_{\ell n}^{\otimes \, s\ell +2}\,\big)\big)=\cdots = \deg\big({\pi_{\ell n}}_\ast\big(\widetilde \Fcal\otimes \Lcal_{\ell n}^{\otimes \, (s+1)\ell}\,\big)\big)
\end{align}
for $s=0, 1, \ldots, n-1$. By Proposition~\ref{prop:injectivity}, the maps
\begin{align}
{\pi_{\ell n}}_\ast\big(\widetilde \Fcal\otimes \Lcal_{\ell n}^{\otimes \, s\ell +1}\,\big)\to {\pi_{\ell n}}_\ast\big(\widetilde \Fcal\otimes \Lcal_{\ell n}^{\otimes \, s\ell +2}\,\big) \to \cdots \to {\pi_{\ell n}}_\ast\big(\widetilde \Fcal\otimes \Lcal_{\ell n}^{\otimes \, (s+1)\ell}\,\big)
\end{align}
are injective, hence they are isomorphisms. Therefore, there exists $\widetilde \Fcal'\in \Coh(X_n)$ such that $\widetilde \Fcal\simeq \pi_{\ell n, n}^\ast \big( \widetilde \Fcal'\big)$, thus $\Fcal\simeq \pi_{\infty, n}^\ast \big(\widetilde \Fcal'\big)$.
\end{proof}
\begin{corollary}\label{cor:completion-I}
Let $(r,f), (r_1, f_1), (r_2, f_2)\in \Ksfnum(X_\infty)$ be such that $(r, f)=(r_1, f_1)+(r_2, f_2)$. Let $n$ be the minimum integer for which $f_1, f_2$ are constant on intervals of the form $[i/n, (i+1)/n[$ for $i=0, \ldots, n-1$. Then for any short exact sequence
\begin{align}
0\to \Fcal_1 \to \Fcal \to \Fcal_2 \to 0
\end{align}
of coherent sheaves on $X_\infty$ such that $\Fcal_i$ is locally free and $\overline{\Fcal_i}=(r_i, f_i)$ for $i=1,2$, there exist $\widetilde{\Fcal_i}\in \Coh(X_{n})$ such that $\Fcal_i\simeq \pi_{\infty, n}^\ast\big(\widetilde{\Fcal_i}\big)$ for $i=1, 2$.
\end{corollary}
By using the previous proposition and recalling that the functor $\pi_{\infty, m}^\ast$ is fully faithful for any $m$, one can prove the following.
\begin{corollary}\label{cor:completion-II}
Let $(r,f), (r_1, f_1), (r_2, f_2)\in \Ksfnum(X_\infty)$ be such that $(r, f)=(r_1, f_1)+(r_2, f_2)$. Let $n$ be the minimum integer for which $f_1,f_2$ are constant on intervals of the form $[i/n, (i+1)/n[$ for $i=0, \ldots, n-1$. Then for any short exact sequence
\begin{align}
0\to \Fcal_1 \to \Fcal \to \Tcal_2 \to 0
\end{align}
of coherent sheaves on $X_\infty$ such that $\Fcal_1$ is locally free and $\overline{\Fcal_1}=(r, f_1)$, $\overline{\Tcal_2}=(0, f_2)$ there exist  $\widetilde{\Fcal_1},\widetilde{\Tcal_2} \in \Coh(X_{n})$ such that $\Fcal_1\simeq \pi_{\infty, n}^\ast\big(\widetilde{\Fcal_1}\big)$ and $\Tcal_2\simeq \pi_{\infty,n}^\ast\big(\widetilde{\Tcal_2}\big)$.
\end{corollary}

\bigskip\section{Hall algebra of a root stack over a curve}\label{sec:hallalgebran}

\subsection{Preliminaries on Hall algebras} Let $k=\F_q$, let $X$ be a smooth geometrically connected projective curve over $k$, and for $n\in\Z_{>0}$ let $X_n$ be the $n$-th root stack of $X$ over a rational closed point $p\in X$ as in Section~\ref{sec:preliminariesorbcurves-nrootstack}. In this section we recall the definition of the Hall algebra of $\Coh(X_n)$ and state some of its important properties. Most results in this section can be found in \cite{art:lin2014}, but we briefly derive them again since our notations differ from Lin's. 

Let $\alpha\in \Ksfnum(X_n)$ and denote by $\Mcal_{\alpha,n}$ the set of isomorphism classes of coherent sheaves on $X_n$ of class $\alpha$. The Hall algebra of $X_n$ is, as vector space,
\begin{align}
\Hbf_n\coloneqq\bigoplus_{\alpha\in \Ksfnum(X_n)}\, \Hbf_n[\alpha] \quad\text{where}\quad \Hbf_n[\alpha]\coloneqq\{f\colon \Mcal_{\alpha,n}\to \widetilde \Q\, \vert\, \text{$\mathsf{supp}(f)$ is finite}\}=\bigoplus_{\Fcal\in  \Mcal_{\alpha,n}}\,  \widetilde \Q \, 1_\Fcal\ ,
\end{align}
where $1_{\Fcal}$ denotes the characteristic function of $\Fcal\in \Mcal_{\alpha,n}$. The multiplication is defined as
\begin{align}\label{eq:Hallprod}
(f\cdot g)(\Rcal)\coloneqq\sum_{\Ncal\subseteq \Rcal}\, \upsilon^{\langle \Rcal/\Ncal,\, \Ncal\rangle}\, f(\Rcal/\Ncal)g(\Ncal) \ ,
\end{align}
and the comultiplication is
\begin{align}\label{eq:Hallcoprod}
\Delta(f)(\Mcal, \Ncal)\coloneqq\frac{\upsilon^{-\langle \Mcal,\, \Ncal\rangle}}{\vert \Ext^1(\Mcal, \Ncal)\vert}\, \sum_{\xi\in \Ext^1(\Mcal, \Ncal)}\, f(\Xcal_\xi)
\end{align}
where $\Xcal_\xi$ is the extension of $\Ncal$ by $\Mcal$ corresponding to $\xi$. The coproduct takes values in the completion  $\Hbf_n\widehat{\otimes}\Hbf_n$ of the tensor product $\Hbf_n \otimes\Hbf_n$ that we will now define. For $\alpha\in \Ksfnum(X_n)$, define
\begin{align}\label{eq:Hallcompletion-1}
\widehat \Hbf_n[\alpha]\coloneqq\{f\in \Mcal_{\alpha,n}\to \widetilde \Q\}\ .
\end{align}
We shall identify elements in $\widehat \Hbf_n[\alpha]$ with (possibly infinite) series
\begin{align}
\sum_{\overline \Fcal=\alpha}\, a_\Fcal\, 1_\Fcal\, \text{ with }\, a_\Fcal\in \widetilde \Q \ .
\end{align}
We put 
\begin{align}
\widehat \Hbf_n\coloneqq\bigoplus_\alpha\, \widehat \Hbf_n[\alpha] \ .
\end{align}
Similarly, for $\alpha, \beta\in \Ksfnum(X_n)$, define
\begin{align}\label{eq:Hallcompletion-2}
\Hbf_n[\alpha]\widehat{\otimes}\Hbf_n[\beta]\coloneqq\{f\colon \Mcal_{\alpha,n}\times\Mcal_{\beta,n}\to \widetilde \Q\}\ .
\end{align}
Again, we identify it with the set of infinite series $\sum_{\overline \Fcal=\alpha, \overline \Ecal=\beta}\, a_{\Fcal, \Ecal}\, 1_\Fcal\otimes 1_\Ecal$, with $a_{\Fcal, \Ecal}\in \widetilde \Q$. Set
\begin{align}\label{eq:Hallcompletion-3}
\Hbf_n\widehat{\otimes}\Hbf_n\coloneqq\bigoplus_{\gamma\in \Ksfnum(X_n)}\, \big(\Hbf_n\widehat{\otimes}\Hbf_n\big)[\gamma]  \ \text{ where }\  \big(\Hbf_n\widehat{\otimes}\Hbf_n\big)[\gamma]\coloneqq\prod_{\gamma=\alpha+\beta}\, \Hbf_n[\alpha]\widehat{\otimes}\Hbf_n[\beta] \ . 
\end{align}
As a consequence of Lemma~\ref{lem:finiteness} (cf.\ the arguments in \cite[Proposition~2.2]{art:burbanschiffmann2012}), we have the following.
\begin{proposition}
The following properties hold:
\begin{itemize}\itemsep0.1cm
\item $\widehat \Hbf_n$ and $\Hbf_n\widehat{\otimes}\Hbf_n$ are naturally equipped with the structure of associative algebras.
\item Let $\Delta_{\alpha,\beta}\colon \Hbf_n[\alpha+\beta]\to \Hbf_n[\alpha]\widehat{\otimes}\Hbf_n[\beta]$ be the graded component of the comultiplication $\Delta$ of degree $(\alpha,\beta)$. Then $\Delta_{\alpha,\beta}\big(\Hbf_n[\alpha+\beta]\big)\subset \Hbf_n[\alpha]\otimes\Hbf_n[\beta]$.
\item The comultiplication $\Delta$ takes values in $\Hbf_n\widehat{\otimes}\Hbf_n$; it extends to a coassociative coproduct $\Delta\colon \widehat \Hbf_n\to \Hbf_n\widehat{\otimes}\Hbf_n$.
\end{itemize}
\end{proposition}

The triple $(\Hbf_\Acal, \cdot, \Delta)$ is not a topological bialgebra. For this reason, let us introduce the extended Hall algebra in the following way. Let $\Kbf_n\coloneqq\widetilde \Q[\Ksfnum(X_n)]$ be the group algebra of the numerical Grothendieck group of $X_n$ and denote by $\kbf_\alpha$ the element of $\Kbf_n$ corresponding to $\alpha\in \Ksfnum(X_n)$. Then the extended Hall algebra is
\begin{align}
\Hbf_n^{\mathsf{tw}}\coloneqq\Hbf_n\otimes_{\widetilde \Q} \Kbf_n \ ,
\end{align}
with the relations
\begin{align}\label{eq:Hallnewrelations}
\kbf_\alpha\kbf_\beta= \kbf_{\alpha+\beta}\ , \kbf_0=1\ , \kbf_\alpha 1_{\Fcal} \kbf_\alpha^{-1}= \upsilon^{(\alpha, \overline \Fcal)}\,1_\Fcal\ .
\end{align}
The new coproduct is
\begin{align}\label{eq:Hallnewcoprod}
\tilde \Delta(\kbf_\alpha)\coloneqq\kbf_\alpha\otimes \kbf_\alpha \quad\text{and}\quad \tilde \Delta(f)\coloneqq\sum_{\Mcal,\, \Ncal}\, \Delta(f)(\Mcal, \Ncal)\, 1_{\Mcal}\kbf_{\overline \Ncal}\otimes 1_\Ncal\ . 
\end{align}
Then $(\Hbf_n^{\mathsf{tw}}, \cdot, \tilde \Delta)$ is a topological bialgebra. Also $\Hbf_n^{\mathsf{tw}}$ is $\Ksfnum(X_n)$-graded: we denote by $\tilde \Delta_{\alpha,\beta}$ the graded component of the coproduct of degree $(\alpha,\beta)$.

Similarly, we can define $\widehat \Hbf_n^{\mathsf{tw}}$. The coproduct $\tilde \Delta$ extends to an algebra homomorphism $\Delta\colon \widehat \Hbf_n^{\mathsf{tw}}\to \Hbf_n^{\mathsf{tw}}\widehat{\otimes}\Hbf_n^{\mathsf{tw}}$. Thus $(\widehat \Hbf_n^{\mathsf{tw}}, \cdot, \tilde \Delta)$ is a topological bialgebra. 

Finally, let $(\cdot, \cdot)_G\colon \Hbf_n^{\mathsf{tw}}\otimes \Hbf_n^{\mathsf{tw}}\to \C$ denote Green's Hermitian scalar product defined by
\begin{align}\label{eq:HallGreen}
(1_\Mcal\kbf_\alpha, 1_\Ncal \kbf_\beta)_G\coloneqq\frac{\delta_{\Mcal,\, \Ncal}}{\vert \mathsf{Aut}(\Mcal)\vert}\, v^{(\alpha, \beta)}\ .
\end{align}
This scalar product is a \textit{Hopf pairing}, i.e.
\begin{align}
(ab, c)_G=(a\otimes b, \tilde \Delta(c))_G
\end{align}
for any $a,b,c\in \Hbf_n^{\mathsf{tw}}$. The restriction of $(\cdot, \cdot)_G$ to $\Hbf_n$ is nondegenerate.

Let's recall the construction of the \textit{reduced} Drinfeld double of the Hopf algebra $\Hbf_n^{\mathsf{tw}}$. First consider the pair of algebras $\Hbf_n^{\mathsf{tw}, \pm}$:
\begin{align}
\Hbf_n^{\mathsf{tw}, +}\coloneqq\bigoplus_{\Fcal\in \bigsqcup\, \Mcal_{\alpha,n}}\, \widetilde \Q\, 1_{\Fcal}^+\otimes_{\widetilde \Q} \Kbf_n\quad\text{and}\quad \Hbf_n^{\mathsf{tw}, -}\coloneqq\bigoplus_{\Fcal\in \bigsqcup\, \Mcal_{\alpha,n}}\, \widetilde \Q\, 1_{\Fcal}^-\otimes_{\widetilde \Q} \Kbf_n\ .
\end{align}
In these notations $\Hbf_n^{\mathsf{tw}, \pm}=\Hbf_n^{\mathsf{tw}}$ viewed as $\widetilde \Q$-algebras. The \textit{negative} extended Hall algebra $\Hbf_n^{\mathsf{tw}, -}$ is a topological $\widetilde \Q$-algebra with the opposite coproduct (cf.\ \cite[Section~2.4]{art:doujiangxiao2012}).

Let's use the Sweedler’s notations $\tilde \Delta\big(a^{\pm}\big)=\sum_i\, a_i^{(1)\, \pm}\otimes a_i^{(2)\, \pm}$. Then the \textit{Drinfeld double} of $\Hbf_n^{\mathsf{tw}}$ with respect to Green's pairing $(\cdot, \cdot)_G$ is the associative algebra $\tildeDHbf_n$, defined as the free product of algebras $\Hbf_n^{\mathsf{tw}, +}$ and $\Hbf_n^{\mathsf{tw}, -}$ subject to the following relations for all $a, b \in\Hbf_n^{\mathsf{tw}}$:
\begin{align}
\sum_{i,j}\, a_i^{(1)\, -}b_j^{(2)\, +} \big(a_i^{(2)}, b_j^{(1)}\big)_G=\sum_{i,j}\, b_j^{(1)\, +} a_i^{(2)\, -} \big(a_i^{(1)}, b_j^{(2)}\big)_G\ .
\end{align}
The \textit{reduced} Drinfeld double $\DHbf_n$ is the quotient of $\tildeDHbf_n$ by the two-sided ideal $\langle \kbf_\alpha^+\otimes \kbf_{-\alpha}^{-}-1^{+}\otimes 1^{-}\, \vert\, \alpha\in\Ksfnum(X_n)\rangle$: this is a Hopf ideal and the reduced Drinfeld double $\DHbf_n$ is again a Hopf algebra. We have an isomorphism of $\widetilde \Q$-vector spaces
\begin{align}
\Hbf_n^{+}\otimes_{\widetilde \Q}\Kbf_n\otimes_{\widetilde \Q} \Hbf_n^{-}\xrightarrow{\mathsf{mult}} \DHbf_n\ ,
\end{align}
called the \textit{triangular decomposition} of $\DHbf_n$.

{\bf Notation.} We shall denote by $\Hbf$, $\Hbf^{\mathsf{tw}}$, $\tildeDHbf$ and $\DHbf$ respectively $\Hbf_1$, $\Hbf_1^{\mathsf{tw}}$, $\tildeDHbf_1$ and $\DHbf_1$ i.e., the Hall algebra, the extended Hall algebra, the corresponding Drinfeld double and its reduced version associated with the curve $X$. \triend 

The exact functor $\pi_n^\ast$ induces injective homomorphism $\mathcal{M}_{\alpha,1}\hookrightarrow \mathcal{M}_{\alpha,n}$ which we still denote by $\pi_{n}^\ast$. This gives rise to maps
\begin{align}
\Omega_n\colon \Hbf\to \Hbf_n  \ , \quad \Omega_n\colon \Hbf^{\mathsf{tw}}\to \Hbf_n^{\mathsf{tw}} 
\end{align}
defined by
\begin{align}
\Omega_n(f)\coloneqq(\pi_n^\ast)_!(f)
\end{align}
sending $1_E$ to $1_{\pi_n^\ast E}$ for $E\in \Mcal_{\alpha, 1}$, $\alpha\in \Ksfnum(X)$, and $\kbf_{(r,d)}$ to $\kbf_{(r, \, d\, \delta_n)}$ for $(r,d)\in \Ksfnum(X)$. It is easy to see that $\Omega_n$ is a morphism of algebras.
Let $\Csf$ be the Serre subcategory of $\Coh(X_n)$ generated by the simple objects ${}_n\Scal_1, \ldots, {}_n\Scal_{n-1}$. Then the \textit{perpendicular category}\footnote{See \cite[Definition~3.1]{art:burbanschiffmann2013} for the definition.} $\Csf^\perp$ is equivalent to the category $\Coh(X)$. Therefore by \cite[Theorem~3.3]{art:burbanschiffmann2013}, we have an injective algebra homomorphism $\Omega_n\colon \DHbf\hookrightarrow \DHbf_n$.

Tensor multiplication by $\Lcal_n$ induces isomorphism of algebras
\begin{align}
T_n\colon \Hbf_n\to \Hbf_n \ , \quad T_n\colon \Hbf_n^{\mathsf{tw}}\to \Hbf_n^{\mathsf{tw}} \ .
\end{align}
Note that $T_n^{\, n}$ corresponds to the isomorphism induces by the functor $\_\otimes \pi_n^\ast \Ocal_X(p)$. There is an induced algebra isomorphim $T_n\colon \DHbf_n\xrightarrow{\sim} \DHbf_n$.

\subsection{Hecke algebra and $\Ubf_\upsilon(\glfrakhat(n))$}

Let $\Tor(X_n)$ the category of zero-dimensional coherent shea\-ves on $X_n$ and
\begin{align}
\Hbf^{\tor}_n\coloneqq\bigoplus_{\Tcal\in \Tor(X_n)}\, \widetilde \Q 1_\Tcal \quad\text{and}\quad \Hbf^{\mathsf{tw}, \tor}_n\coloneqq\Hbf^{\tor}_n\otimes_{\widetilde \Q} \Kbf_n\ .
\end{align}
Define for any closed point $x\in X_n$
\begin{align}
\Hbf_{n, x}^{\tor}\coloneqq\bigoplus_{\Tcal\in \Tor_{x}(X_n)}\, \widetilde\Q\, 1_\Tcal \quad\text{and}\quad \Hbf_{n, x}^{\mathsf{tw}, \tor}\coloneqq\Hbf_{n, x}^{\tor}\otimes_{\widetilde \Q} \Kbf_n\ .
\end{align}
The decomposition of $\Tor(X_n)$ over the closed points of $X_n$ gives a decomposition at the level of Hall algebras
\begin{align}
\Hbf^{\tor}_n=\bigotimes_{x\in X_n}\, \Hbf_{n, x}^{\tor} \quad\text{and}\quad \Hbf^{\mathsf{tw}, \tor}_n=\bigotimes_{x\in X_n}\,\Hbf_{n, x}^{\mathsf{tw}, \tor} \ .
\end{align}

\subsubsection{Hecke algebra at the stacky point, $\Ubf_\upsilon(\slfrakhat(n))$ and  $\Ubf_\upsilon(\glfrakhat(n))$}\label{sec:Heckealgebrastacky}

We briefly recall the description of the (extended) Hall algebra associated with the category $\Tor_{p_n}(X_n)$ by following \cite[Section~5.1]{art:burbanschiffmann2013} and references therein, for $n\geq 2$.

We have a decomposition of $\Hbf_{n, p_n}^{\tor}$ and $\Hbf_{n, p_n}^{\mathsf{tw}, \tor}$ according to their dimension vectors (see  \eqref{eq:dim})
\begin{align}
\Hbf_{n, p_n}^{\tor}=\bigoplus_{\mathbf d\in \Z^n}\, \Hbf_{n, p_n}^{\tor}[\mathbf d]\quad\text{and}\quad \Hbf_{n, p_n}^{\mathsf{tw}, \tor}=\bigoplus_{\mathbf d\in \Z^n}\, \Hbf_{n, p_n}^{\mathsf{tw}, \tor}[\mathbf d]\ .
\end{align}

Recall that $\Tor_{p_n}(X_n)$ is equivalent to the category $\Rep_k^{\mathsf{nil}}\big(A_{n-1}^{(1)}\big)$ of nilpotent representations of the cyclic quiver $A_{n-1}^{(1)}$ with $n$ vertices over $k$. 

As is customary, we call the \textit{composition subalgebra} $\Cbf_n$ of $\Hbf_{n, p_n}^{\mathsf{tw}, \tor}$ the topological $\widetilde \Q$-Hopf algebra generated by $1_{{}_n\Scal_i}$ and $\kbf_{\pm\,\overline{{}_n\Scal_i}}$ for $i=1, \ldots, n$. By \cite{art:ringel1993}, we have the following result.
\begin{proposition}
The assignment
\begin{align}\label{eq:F-K}
\begin{aligned}
E_i& \mapsto \upsilon^{1/2}\, 1_{{}_n \Scal_i}\ , \quad  K_i^\pm\mapsto \kbf_{(0, \pm \mathbf e_{n, i})}\quad\text{ for $i\neq 0$}\ ,\\[2pt]
E_0 &\mapsto \upsilon^{1/2}\, 1_{{}_n \Scal_n}\ , \quad  K_0^\pm\mapsto \kbf_{(0, \pm \mathbf e_{n, n})}\ , 
\end{aligned}
\end{align}
defines an isomorphism between $\Cbf_n$ and $\Ubf_\upsilon^+(\slfrakhat(n))\otimes_{\widetilde \Q}\widetilde \Q[\delta_n]$.
\end{proposition}
Such an isomorphism extends to the whole $\Ubf_\upsilon\big(\slfrakhat(n)\big)$ by setting
\begin{align}\label{eq:E}
F_i \mapsto -\upsilon^{1/2}\, 1_{{}_n \Scal_i}^-\quad\text{for $i\neq 0$}\ , \quad   F_0 \mapsto -\upsilon^{1/2}\, 1_{{}_n \Scal_n}^-\ ,
\end{align}
yielding:
\begin{corollary}[{\cite{art:xiao1997}}]\label{cor:DC_n}
The assignments \eqref{eq:F-K} and \eqref{eq:E} give an isomorphism between $\Ubf_\upsilon\big(\slfrakhat(n)\big)$ and the reduced Drinfeld double $\mathbf{D}\Cbf_n$ of $\Cbf_n$.
\end{corollary}

Let $r\in \Z_{>0}$. Following Hubery \cite{art:hubery2005}, define the element
\begin{align}\label{eq:nc_r}
{}_n c_r\coloneqq \sum_{\genfrac{}{}{0pt}{}{\Fcal\in \Tor_{p_n}(X_n)\colon \underline{\dim}([\Fcal])=r\delta_n}{\mathsf{socle}(\Fcal) \,\text{\tiny square free}}}\, (-1)^{\dim_k \End(\Fcal)}\, \vert \Aut(\Fcal)\vert\, 1_\Fcal\;\in \;\Hbf_{n, p_n}^{\mathsf{tw}, \tor}[r\delta_n]\ .
\end{align}
Recall that the \textit{socle} of $\Fcal$ is the maximal semisimple subsheaf of $\Fcal$. A semisimple sheaf of $\Tor_{p_n}(X_n)\simeq \mathsf{Rep}^{\mathsf{nil}}\big(A_{n-1}^{(1)}\big)$ is \textit{square free} if it is isomorphic to $\oplus_{i=1}^n\, {}_n\Scal_i^{\oplus n_i}$ with $n_j\leq 1$ for $j=1, \ldots, n$.

Moreover, define the element ${}_n z_r\in \Hbf_{n, p_n}^{\mathsf{tw}, \tor}[r\delta_n]$ by the recursion formula
\begin{align}\label{eq:nz_r}
{}_n z_r\coloneqq r \, {}_n c_r-\sum_{\ell=1}^{r-1}\, {}_n z_\ell\, {}_n c_{r-\ell}\ .
\end{align}
${}_n z_r$ is a \textit{primitive element} for any $r\in \Z_{>0}$, i.e., $\tilde \Delta({}_n z_r)={}_n z_r\otimes 1+\kbf_{(0, \, r\delta_n)}\otimes {}_n z_r$. Denote by $\Zbf_n$ the center of $\Hbf_{n, p_n}^{\mathsf{tw}, \tor}$. As proved in \textit{loc.cit.} (see also \cite{art:schiffmann2002} and \cite[Theorems~5.3 and 5.4]{art:burbanschiffmann2013}), $\Zbf_n\simeq \widetilde \Q[{}_n z_1, {}_n z_2, \ldots, {}_n z_r, \ldots]$ and the canonical morphism 
\begin{align}
\Zbf_n\otimes_{\widetilde \Q} \Cbf_n \xrightarrow{\mathsf{mult}} \Hbf_{n, p_n}^{\mathsf{tw}, \tor}
\end{align}
is an isomorphism of vector spaces over $\widetilde \Q$.
\begin{remark}
Note that our ${}_n c_r$'s and ${}_n z_r$'s differ from those introduced in \cite[Theorem~5.4]{art:burbanschiffmann2013} by the factor $(-1)^r q^{-rn}$.
\end{remark}
Summarizing, we have the following result:
\begin{corollary}\label{cor:heckealgebrastackypoint-n}
We have a decomposition
\begin{align}
\Hbf_{n, p_n}^{\mathsf{tw}, \tor}\simeq \widetilde \Q[{}_n z_1, {}_n z_2, \ldots, {}_n z_r, \ldots]\otimes_{\widetilde \Q} \Ubf_\upsilon^+(\slfrakhat(n))\otimes_{\widetilde \Q}\widetilde \Q[\delta_n]\ .
\end{align}
It extends to the reduced Drinfeld double:
\begin{align}
\mathbf{D} \Hbf_{n, p_n}^{\mathsf{tw}, \tor}\simeq \Hcal_n\otimes_{\Acal_n} \Ubf_\upsilon(\slfrakhat(n))\eqqcolon \Ubf_\upsilon(\glfrakhat(n))\ .
\end{align}
Here $\Acal_n\coloneqq\widetilde \Q[C_n^\pm]$ is the ring of Laurent polynomials in the variable $C_n\coloneqq\kbf_{\delta_n}$, $\Hcal_n\coloneqq\widetilde \Q\langle {}_n  z_r^\pm\, \vert\, r\in \Z_{>0}\rangle\otimes_{\widetilde \Q}\Acal_n$, subject to the relations
\begin{align}
[{}_n z_r, {}_n z_t]\coloneqq\delta_{r+t, 0}\, ({}_n z_r, {}_n z_r)_G\, (C_n^{-r}-C_n^r) \quad\text{and}\quad [{}_n z_r, C_n^\pm]=0 \text{ for $r,t\in \Z\setminus\{0\}$} \ .
\end{align}
Here we set ${}_n z_{\pm r}\coloneqq{}_n z_r^{\pm}$ for $r\in\Z_{>0}$.
\end{corollary}

\subsubsection{Spherical Hecke algebra of the curve}\label{sec:sphericalHeckealgebracurve}

We shall now define the \textit{composition subalgebra} of $\Hbf^{\mathsf{tw}, \mathsf{tor}}$. For any positive integer $d$ and any closed point $x \in X$ define
\begin{align}
\onebf_{0,\, d; \, x}\coloneqq\sum_{\genfrac{}{}{0pt}{}{\Tcal\in \Tor_x(X)}{\deg(\Tcal)=d}}\, 1_\Tcal
\end{align}
and the element $T_{0, \, d;\, x}$ by the relation
\begin{align}
1+\sum_{d\geq 1}\, \onebf_{0,\, d; \, x}\, z^d=\exp\Big(\sum_{d\geq 1}\, \frac{T_{0, \, d; \, x}}{[d]_\upsilon}\, z^d\Big) \ .
\end{align}
Set $\onebf_{0,\, 0; \, x}=T_{0, \, 0;\, x}=1$. Now put
\begin{align}
\onebf_{0,\, d}\coloneqq\sum_{\genfrac{}{}{0pt}{}{\Tcal\in \Tor(X)}{\deg(\Tcal)=d}}\, 1_\Tcal\ .
\end{align}
and the elements $T_{0, \, d}, \theta_{0, \, d}$ in $\Hbf^{\mathsf{tw}, \tor}$ via the relations
\begin{align}\label{eq:torsionrelations}
1+\sum_{d\geq 1}\, \onebf_{0,\, d}\, z^d=\exp\Big(\sum_{d\geq 1}\, \frac{T_{0, \, d}}{[d]}\, z^d\Big) \,\text{ and }\, 1+\sum_{d\geq 1}\, \theta_{0,\, d}\, z^d=\exp\Big( (\upsilon-\upsilon^{-1})\, \sum_{d\geq 1}\, T_{0, \, d}\, z^d\Big)\ .
\end{align}
We set also $\onebf_{0,0}=T_{0,0}=\theta_{0,0}=1$. One has
\begin{align}
1+\sum_{d\geq 1}\, \onebf_{0,\, d}\, z^d&=\prod_{x\in X}\, \big(1+\sum_{d\geq 1}\, \onebf_{0,\, d; \, x}\, z^d\big)\ ,\\
\exp\Big(\sum_{d\geq 1}\, \frac{T_{0, \, d}}{[d]}\, z^d\Big)&=\prod_{x\in X}\, \Big(\exp\Big(\sum_{d\geq 1}\, \frac{T_{0, \, d; \, x}}{[d]}\, z^d\Big)\Big)\ .
\end{align}

\begin{lemma}[{\cite[Example~4.12, Lemmas~4.50 and 4.51]{book:schiffmann2012}}]
Let $d$ be a positive integer. Then
\begin{itemize}\itemsep0.2em
\item $\tilde \Delta(T_{0,\, d})=T_{0, \, d}\otimes 1+\kbf_{(0, \, d)}\otimes T_{0, \, d}$.
\item $\tilde \Delta(\theta_{0, \, d})=\sum_{s=0}^d\, \theta_{0,\, s}\kbf_{(0,\, d-s)}\otimes\theta_{0,\, d-s}$.
\item $(T_{0, \, d}, T_{0, \, d})_G=\big(\upsilon^{1-d}\, \#X(\F_{q^d})[d]\big)/d(q-1)$.
\end{itemize}
\end{lemma}
The sets $\{\onebf_{0,\, d}\,\vert\, d\geq 1\}$, $\{T_{0,\, d}\,\vert\, d\geq 1\}$ and $\{\theta_{0,\, d}\,\vert\, d\geq 1\}$ all generate the same subalgebra $\Cbf=\Cbf_1$ of $\Hbf^{\mathsf{tw}, \tor}$. It is known that $\Cbf=\C[\onebf_{0, 1}, \onebf_{0, 2}, \ldots]$, i.e., the commuting elements $\onebf_{0,\, d}$ for $d\geq 1$ are algebraically independent. Clearly, the same holds for the $T_{0,\, d}$'s and the $\theta_{0,\, d}$'s. 

\subsubsection{Spherical Hecke algebra of the root stack over a curve and  and $\Ubf_\upsilon(\glfrakhat(n))$}

Let $\Ubf_n^0\subset \Hbf_n^{\mathsf{tw}}$ be the algebra generated by $\Omega_n(\Cbf)$ and $\Cbf_n$. We call it the \textit{spherical Hecke algebra} of $X_n$.

For $r\in\Z_>0$, define
\begin{align}
{}_n Z_r\coloneqq\sum_{d\vert r}\, \sum_{\substack{x\in X\\ \deg(x)=d}}\, {}_n Z_{r; \, x}\ ,
\end{align}
where
\begin{align}
{}_n Z_{r;\, p}\coloneqq\frac{[r]}{r}\, {}_n z_r\quad\text{and}\quad {}_n Z_{r;\, x}\coloneqq\Omega_n(T_{r; \, x}) \text{ for $x\in X\setminus\{p\}$}\ .
\end{align}
We have
\begin{align}
\tilde \Delta({}_n Z_r)={}_n Z_r\otimes 1+\kbf_{(0, r\delta_n)}\otimes {}_n Z_r \ .
\end{align}
\begin{remark}
In addition, the elements ${}_nZ_r$ commute and are algebraically independent.
Note that our definition of ${}_n Z_r$ differs from the one in \cite[Section~6]{art:burbanschiffmann2013} by a factor $(1-q^{rn})^{-1}$.
\end{remark}
Set $\Lambda_n=\widetilde{\Q}[{}_nZ_1, {}_nZ_2, \ldots]$ and $\mathbf{K}_n^0=\widetilde{\Q}[\Ksf^{\tor}(X_n)]$, where $\Ksf^{\tor}(X_n)$ is the subgroup of $\Ksfnum(X_n)$ of classes of torsion sheaves on $X_n$.
\begin{proposition}[{cf.\ \cite[Proposition~6.3-(1)]{art:burbanschiffmann2013}}]
We have a decomposition
\begin{align}
\Ubf_n^0\simeq \Lambda_n\otimes_{\widetilde \Q} \Ubf_\upsilon^+(\slfrakhat(n))\otimes_{\widetilde \Q}\mathbf{K}_n^0\ .
\end{align}
\end{proposition}
We have also a characterization of the reduced Drinfeld double of $\Ubf_n^0$.
\begin{corollary}[{cf.\ \cite[Corollary~6.4]{art:burbanschiffmann2013}}]
We have a decomposition
\begin{align}\label{eq:Uq(gl)}
\DUbf_n^0\simeq \Hcal_n\otimes_{\Acal_n} \Ubf_\upsilon(\slfrakhat(n))\eqqcolon \Ubf_\upsilon(\glfrakhat(n))\ .
\end{align}
Here $\Acal_n\coloneqq\widetilde \Q[C_n^\pm]$ is the ring of Laurent polynomials in the variable $C_n\coloneqq\kbf_{\delta_n}$, $\Hcal_n\coloneqq\widetilde \Q\langle {}_n  Z_r^\pm\, \vert\, r\in \Z_{>0}\rangle\otimes_{\widetilde \Q}\Acal_n$, subject to the relations
\begin{align}
[{}_n Z_r, {}_n Z_t]\coloneqq\delta_{r+t, 0}\, ({}_n Z_r, {}_n Z_r)_G\, (C_n^{-r}-C_n^r) \quad\text{and}\quad [{}_n Z_r, C_n^\pm]=0 \text{ for $r,t\in \Z\setminus\{0\}$} \ .
\end{align}
Here we set ${}_n Z_{\pm r}\coloneqq{}_n Z_r^{\pm}$ for $r\in\Z_{>0}$.
\end{corollary}
\begin{remark}
Let us fix a positive integer $n\geq 2$. Note that $\Ubf_\upsilon(\slfrakhat(n))$ can be geometrically realized as the reduced Drinfeld double of $\Cbf_n$, while $\Ubf_\upsilon(\glfrakhat(n))$ can be  obtained as reduced Drinfeld double of either the entire Hall algebra $\Hbf_{n, p_n}^{\mathsf{tw}, \tor}$ of $\Tor_{p_n}(X_n)$ or the spherical Hecke algebra $\Ubf_n^0$.
\end{remark}

\subsection{Spherical Hall algebra}

For $d\in \Z$, define
\begin{align}
{}_n\onebf^{\mathbf{ss}}_{1,\, d}\coloneqq\sum_{\genfrac{}{}{0pt}{}{M\in \Pic(X)}{\deg(M)=\nintpart{d}}}\, 1_{\pi_n^\ast M\otimes \Lcal_n^{\otimes\, \nfrapart{d}}}\ .
\end{align}
Denote by $\Ubf_n^>$ the subalgebra of $\Hbf_n^{\mathsf{tw}}$ generated by $\{\onebf^{\mathbf{ss}}_{1,\, d}\, \vert\, d\in \Z\}$. We call it the \textit{(positive) spherical Hall algebra} of $\Hbf_n^{\mathsf{tw}}$. We denote by $\Ubf^>$ the positive spherical subalgebra of $\Hbf^{\mathsf{tw}}=\Hbf_1^{\mathsf{tw}}$.

We define the \textit{spherical Hall algebra} $\Ubf_n$ as the subalgebra of $\Hbf_n^{\mathsf{tw}}$ generated by $\Ubf_n^0$ and $\Ubf_n^>$. First, note that the canonical multiplication map
\begin{align}
\Ubf_n^>\otimes_{\widetilde \Q} \Ubf_n^0\xrightarrow{\mathsf{mult}} \Ubf_n
\end{align}
is an isomorphism of $\widetilde \Q$-vector spaces (cf.\ \cite[Proposition~6.5]{art:burbanschiffmann2013}).

Since $\Ubf_n^0$ is a Hopf algebra (cf.\ \cite[Section~5.4]{art:burbanschiffmann2013}) and thanks to Proposition~\ref{prop:coproduct}, $\Ubf_n$ is a topological sub-bialgebra of $\Hbf_n^{\mathsf{tw}}$. As a consequence, we have the following (cf.\ \cite[Theorems~5.22 and 5.25]{art:burbanschiffmann2013}):
\begin{proposition}
The canonical multiplication map
\begin{align}
\overline{\Ubf}_n^+\otimes_{\widetilde \Q}\Kbf_n\otimes_{\widetilde \Q} \overline{\Ubf}_n^- \xrightarrow{\mathsf{mult}}\DUbf_n
\end{align}
is an isomorphism of $\widetilde \Q$-vector spaces. Here $\overline{\Ubf}_n^\pm$ is the subalgebra of the reduced Drinfeld double $\DUbf_n$ generated by the elements $\onebf_{1, \, d}^\pm$, $\Omega_n(T_{0, \, r})^\pm$ and $1_{{}_n\Scal_i}^\pm$ for $d, r, i, j\in\Z$, $r>0$, and $1\leq i\leq n$.

Moreover, there exists an injective algebra homomorphism $\DUbf\hookrightarrow \DUbf_n$.
\end{proposition}

\subsection{Shuffle algebra presentation of $\Ubf_n^>$}\label{sec:shuffle}

In this section, we shall provide a \textit{shuffle presentation} of $\Ubf_n^>$. To achieve this, we use the so-called constant term map, which is (a component of) the maximal iterated coproduct .

\begin{remark}
In the following, In order to avoid cumbersome notation (cf.\ Remark~\ref{rem:Ktheorylinebundle}), we denote the element in $\Kbf_n$ correspondly to a class $\overline{\pi_n^\ast M\otimes \Lcal_n^{\otimes \, \ell}}$ simply by $\kbf_{(1,\, d)}$, where $d=\deg(M) n+\ell$. 
\end{remark}
Let $\Hbf^{\bun}_n$ be the Hall algebra associated with the exact subcategory consisting of locally free sheaves on $X_n$. Then $\Hbf^{\bun}_n$ is a subalgebra of $\Hbf_n$, which is not stable under the coproduct. The decomposition of a coherent sheaf into a locally free part and a torsion part gives rise to isomorphism 
\begin{align}
\Hbf^{\bun}_n\otimes \Hbf^{\tor}_n \to \Hbf_n \quad\text{and}\quad \Hbf^{\bun}_n\otimes \Hbf^{\tor}_n\otimes \Kbf_n\to \Hbf^{\mathsf{tw}}_n
\end{align}
defined by the multiplication map; the comultiplication provides an inverse of this map. There is a natural projection $\omega_n\colon \Hbf^{\mathsf{tw}}_n\to \Hbf^{\bun}_n$ defined as follows: for any $\mathbf u_v\in \Hbf_n^{\bun}, \mathbf u_t\in \Hbf_n^{\tor}, (r,\mathbf d)\in \Z\oplus\Z^n$ we set
\begin{align}\label{eq:omega}
\omega_n\big(\mathbf u_v\mathbf u_t\kbf_{(r, \, \mathbf d)}\big)=
\begin{cases}
\mathbf u_v & \text{if } \mathbf{u}_t=1\ ,\\
0 & \text{otherwise} \ .
\end{cases}
\end{align}
In terms of functions, this map is simply the restriction of functions on $\Mcal_\alpha$ to its open substack $\Mcal_\alpha^{\bun}$ parametrizing vector bundles.

We may now define the \textit{constant term map}. For $r\geq 1$, define
\begin{align}
\Hbf_n^{\mathsf{tw}}[r]\coloneqq\bigoplus_{\genfrac{}{}{0pt}{}{\alpha\in \Ksfnum(X_n)}{\rk(\alpha)=r}}\, \Hbf_n^{\mathsf{tw}}[\alpha] \quad\text{and}\quad \Ubf_n^>[r]\coloneqq\Ubf_n^>\cap \Hbf_n^{\mathsf{tw}}[r]\ 
\end{align}
and denote by
\begin{align}
\tilde{\Delta}_{1, \ldots, 1}\colon \Hbf_n^{\mathsf{tw}}[r] \to \Hbf_n^{\mathsf{tw}}[1]^{\widehat \otimes\, r}
\end{align}
the corresponding component of the $r$fold iterated coproduct.
For $r\geq 1$ we set
\begin{align}
J_r\colon \Ubf_n^>[r] \to \left( \Ubf_n^>[1]\right)^{\widehat{\otimes} \, r}\ , \quad u \mapsto \omega_n^{\otimes\, r}\, \tilde \Delta_{1, \ldots, 1}(u)\ ,
\end{align}
and denote by $J\colon \Ubf_n^>\to \bigoplus_r \left( \Ubf_n^>[1]\right)^{\widehat{\otimes} \, r}$ the sum of all the maps $J_r$. We call it the \textit{constant term map}.
\begin{lemma}[see e.g. {\cite[Lemma~1.5]{art:schiffmannvasserot2012}}]
The constant term map $J$ is injective.
\end{lemma}

Recall that a \textit{braided structure} on a vector space $W$ is an invertible linear map $\sigma\colon W\otimes W\to W\otimes W$ satisfying the braid relation
\begin{align}\label{eq:braided}
(\sigma\otimes \mathsf{id})\, (\mathsf{id}\otimes \sigma)\, (\sigma\otimes \mathsf{id})=(\mathsf{id}\otimes \sigma)\, (\sigma\otimes \mathsf{id})\, (\mathsf{id}\otimes \sigma)\ .
\end{align}
Now, let $V_n$ be the vector space $\C^n$ with basis vectors $\vec v_0, \ldots, \vec v_{n-1}$. In the following, we shall identify $\otimes_{i=1}^k\, \big(\C[x_i^{\pm 1}]\otimes V_n\big)$ with $\C[x_1^{\pm 1}, \ldots, x_k^{\pm k}]\otimes V_n^{\otimes\, k}$. Let $h(z)\in\C(z)$ be a rational function. Define the following map:
\begin{align}
\varpi^h \colon \C[x_1^{\pm 1}, x_2^{\pm 1}]\otimes V_n^{\otimes \, 2}\to \C[x_1^{\pm 1}, x_2^{\pm 1}][[x_1/x_2]]\otimes V_n^{\otimes \, 2}
\end{align}
where the image under $\varpi^h$ of $x_1^{d_1}x_2^{d_2}\, \vec v_{ i_1}\otimes \vec v_{i_2}$ is
\begin{align}
\begin{cases}
 h(x_1x_2^{-1})\, \upsilon^{-1}\, \frac{1-x_1 x_2^{-1}}{1-\upsilon^{-2}\, x_1 x_2^{-1}}\,x_1^{d_2}x_2^{d_1}\, \vec v_{i_2}\otimes \vec v_{i_1} 
 +h(x_1x_2^{-1})\, \frac{1-\upsilon^{-2}}{1-\upsilon^{-2}x_1x_2^{-1}} \, x_1^{d_2}x_2^{d_1}\, \vec v_{i_1}\otimes \vec v_{i_2} & \text{if $i_1>i_2$}\ , \\[15pt]
 h(x_1 x_2^{-1})\,x_1^{d_2}x_2^{d_1}\, \vec v_{i_1}\otimes \vec v_{i_2} & \text{if $i_1=i_2$} \ ,\\[15pt]
 h(x_1x_2^{-1})\, \upsilon^{-1}\, \frac{1-x_1 x_2^{-1}}{1-\upsilon^{-2}\, x_1 x_2^{-1}}\,x_1^{d_2}x_2^{d_1}\, \vec v_{i_2}\otimes \vec v_{i_1}   +h(x_1x_2^{-1})\, \frac{1-\upsilon^{-2}}{1-\upsilon^{-2}x_1x_2^{-1}} \, x_1^{d_2+1}x_2^{d_1-1}\, \vec v_{i_1}\otimes \vec v_{i_2}  & \text{if $i_1<i_2$}\ .
\end{cases}
\end{align}
Here, the rational functions on the right-hand-side are developed as a Laurent series in $x_1/x_2$. 
\begin{example}
Let $h(z)=(q-z)/(1-qz)$. Then the previous formula simplifies as
\begin{align}
\begin{cases}
 \upsilon\, \frac{x_1-x_2}{qx_1-x_2}\,x_1^{d_2}x_2^{d_1}\, \vec v_{i_2}\otimes \vec v_{i_1} 
 +\frac{1-q}{qx_1-x_2} \, x_1^{d_2}x_2^{d_1}\, \vec v_{i_1}\otimes \vec v_{i_2} & \text{if $i_1>i_2$}\ , \\[15pt]
 \frac{x_1-qx_2}{qx_1-x_2}\,x_1^{d_2}x_2^{d_1}\, \vec v_{i_1}\otimes \vec v_{i_2} & \text{if $i_1=i_2$} \ ,\\[15pt]
  \upsilon\, \frac{x_1-x_2}{qx_1-x_2}\, x_1^{d_2}x_2^{d_1}\, \vec v_{i_2}\otimes \vec v_{i_1} 
 +\frac{1-q}{qx_1-x_2}  \, x_1^{d_2+1}x_2^{d_1-1}\, \vec v_{i_1}\otimes \vec v_{i_2}  & \text{if $i_1<i_2$}\ .
\end{cases}
\end{align}
\end{example}

For any $k\in \Z_{>0}$ and $1\leq i\leq k-1$, we denote by $\varpi_i^h$ the operator
\begin{align}
\mathsf{id}^{\otimes\, i-1}\otimes \varpi^h\otimes \mathsf{id}^{\otimes k-i-1}\colon \C[x_1^{\pm 1}, \ldots, x_k^{\pm k}]\otimes V_n^{\otimes\, k}\to \C[x_1^{\pm 1}, \ldots, x_k^{\pm k}][[x_1/x_2, \ldots, x_{k-1}/x_k]]\otimes V_n^{\otimes\, k}\ .
\end{align}

Finally, set $\varpi^h_0\coloneqq \mathsf{id}$.  Let $\Sfrak_k$ be the group of permutations on $k$ letters. For $\sigma\in \Sfrak_k$ with reduced expression $\sigma=s_{i_1}\cdots s_{i_\ell}$, where $s_i=(i,\, i+1)$, we set
\begin{align}\label{eq:varpi}
\varpi_\sigma^h\coloneqq\varpi_{i_1}^h\circ \varpi_{i_2}^h\circ \cdots \circ \varpi_{i_\ell}^h\ .
\end{align}

Consider the following assignment: for any $d\in \Z$
\begin{align}\label{eq:assignment}
\onebf_{1, d}^{\mathbf{ss}}\mapsto x^{\nintpart{d}}\, \vec v_{\nfrapart{d}}\ .
\end{align}
Thus we have
\begin{align}
\left( \Ubf_n^>[1]\right)^{\otimes \, r}& \simeq \C[x_1^{\pm 1}, \ldots, x_r^{\pm 1}]\otimes V_n^{\otimes\, r}\ , \\[4pt]
\left( \Ubf_n^>[1]\right)^{\widehat{\otimes} \, r} & \simeq \C[x_1^{\pm 1}, \ldots, x_r^{\pm 1}][[x_1/x_2,\ldots, x_{r-1}/x_r]]\otimes V_n^{\otimes\, r}\ .
\end{align}
Let
\begin{align}
\mathsf{Sh}_{r,s}\coloneqq\big\{\sigma \in \Sfrak_{r+s}\,\big\vert\, \sigma(i)<\sigma(j)\text{ if } 1\leq i<j\leq r \text{ or } r<i<j \leq r+s\big\}
\end{align}
be the set of $(r,s)$-shuffle. 

\begin{theorem}\label{prop:shuffle-vec}
Set
\begin{align}
h_X(z)\coloneqq\frac{\upsilon^{2(1-g_X)}\zeta_X(z)}{\zeta_X(\upsilon^{-2} z)}\ .
\end{align}
Then 
\begin{itemize}[leftmargin=0.2cm]
\item $\varpi^{h_X}$ satisfies the relation \eqref{eq:braided}. 
\item the vector space
\begin{align}
\mathbf{Sh}_{n, \, h(z)}^{\mathsf{vec}}\coloneqq\C\oplus \bigoplus_{r\geq 1}\, \C[x_1^{\pm 1}, \ldots, x_r^{\pm 1}][[x_1/x_2,\ldots, x_{r-1}/x_r]]\otimes V_n^{\otimes\, r}\ ,
\end{align}
with the multiplication given by 
\begin{multline}
P(x_1, \ldots, x_r)\, \bigotimes_{i=1}^r\, \vec v_{\ell_i}\star Q(x_1, \ldots, x_s)\, \bigotimes_{k=1}^s\, \vec v_{j_k} \coloneqq\\
\sum_{\sigma\in\mathrm{Sh}_{r,s}}\, \varpi^{h_X}_\sigma\big(P(x_1, \ldots, x_r)\, Q(x_{r+1}, \ldots, x_{r+s})\, \otimes\bigotimes_{i=1}^{r+s}\, \vec v_{c_i}\big)
\end{multline}
where $c_i=\ell_i$ for $i=1, \ldots, r$ and $c_{r+k}=j_k$ for $k=1, \ldots, s$, is an associative algebra. Moreover, $\mathbf{Sh}_{n, \, h(z)}^{\mathsf{vec}}$ is equipped with a coproduct $\Delta\colon \mathbf{Sh}_{n, h(z)}^{\mathsf{vec}}\to \mathbf{Sh}_{n, h(z)}^{\mathsf{vec}}\widehat{\otimes} \mathbf{Sh}_{n, h(z)}^{\mathsf{vec}}$ given by
\begin{multline}
\Delta_{u,w}\big(x_1^{d_1}\cdots x_r^{d_r}\, \vec v_{\ell_1}\otimes \cdots \vec v_{\ell_r}\big)\coloneqq\\
(x_1^{d_1}\cdots x_u^{d_u}\, \vec v_{\ell_1}\otimes \cdots \vec v_{\ell_u})\otimes (x_1^{d_{u+1}}\cdots x_w^{d_{r}}\, \vec v_{\ell_{u+1}}\otimes \cdots \vec v_{\ell_r})\quad\text{and}\quad \Delta\coloneqq\bigoplus_{r=u+w}\, \Delta_{u,w}\ .
\end{multline}
\item the constant term map $J\colon \Ubf_n^>\to \mathbf{Sh}_{n, h_X(z)}^{\mathsf{vec}}$ is an algebra morphism such that
\begin{align}
(J_u\otimes J_w)\circ (\omega_n\otimes\omega_n)\circ \Delta_{u,w}=\Delta_{u,w}\circ J_{u+w}\ .
\end{align}
\end{itemize}
\end{theorem} 
\begin{proof}
Let's start by computing $\tilde \Delta_{1,1}\big({}_n\onebf^{\mathbf{ss}}_{1,\, d_1}\cdot {}_n\onebf^{\mathbf{ss}}_{1,\, d_2}\big)$:
\begin{align}
\omega_n^{\otimes\, 2}\tilde \Delta_{1,1}\big({}_n\onebf^{\mathbf{ss}}_{1,\, d_1}\cdot {}_n\onebf^{\mathbf{ss}}_{1,\, d_2}\big)={}_n\onebf^{\mathbf{ss}}_{1,\, d_1}\otimes {}_n\onebf^{\mathbf{ss}}_{1,\, d_2} + \sum_{\genfrac{}{}{0pt}{}{e\in\Z}{e\geq 0}}\, \omega_n^{\otimes\, 2 }\big(T_n^{\, \nfrapart{d_1}}\big({}_n\theta_{0,\, e}\big)\,\kbf_{(1,\, d_1-e)}\, {}_n\onebf^{\mathbf{ss}}_{1,\, d_2}\otimes {}_n\onebf^{\mathbf{ss}}_{1,\, d_1-e}\big)
\end{align}
This can be seen as a ``standard factor" appearing in $J_2$. Now, we can apply Proposition~\ref{prop:omega}
\begin{multline}
\sum_{e\in\Z_{\geq 0}}\,\omega_n^{\otimes\, 2 }\big(T_n^{\, \nfrapart{d_1}}\big({}_n\theta_{0,\, e}\big)\,\kbf_{(1,\, d_1-e)}\, {}_n\onebf^{\mathbf{ss}}_{1,\, d_2}\otimes {}_n\onebf^{\mathbf{ss}}_{1,\, d_1-e}\big)\\
\shoveright{=\sum_{e\in\Z_{\geq 0}}\,\upsilon^{\left((1, \, d_1-e), (1, \, d_2) \right)}\omega_n^{\otimes\, 2 }\big(T_n^{\, \nfrapart{d_1}}\big({}_n\theta_{0,\, e}\big)\, {}_n\onebf^{\mathbf{ss}}_{1,\, d_2}\,\kbf_{(1,\, d_1-e)}\otimes {}_n\onebf^{\mathbf{ss}}_{1,\, d_1-e}\big)}\\
\shoveright{=\sum_{e\in\Z_{\geq 0}}\,\upsilon^{\left((1, \, d_1-e), (1, \, d_2) \right)}\omega_n^{\otimes\, 2 }\big(T_n^{\, \nfrapart{d_1}}\big({}_n\theta_{0,\, e}\, {}_n\onebf^{\mathbf{ss}}_{1,\, d_2-\nfrapart{d_1}}\big)\,\kbf_{(1,\, d_1-e)}\otimes {}_n\onebf^{\mathbf{ss}}_{1,\, d_1-e}\big)}\\[3pt]
=\sum_{e\in \Z_{\geq 0}}\, \upsilon^{\left((1, \, d_1-e), (1, \, d_2) \right)}\, \xi_e^{(d_2-d_1)}\,  {}_n\onebf^{\mathbf{ss}}_{1,\, d_2+e}\otimes {}_n\onebf^{\mathbf{ss}}_{1,\, d_1-e}\ .
\end{multline}
Thus
\begin{multline}
\sum_{e\in\Z_{\geq 0}}\, \omega_n^{\otimes\, 2 }\big(T_n^{\, \nfrapart{d_1}}\big({}_n\theta_{0,\, e}\big)\,\kbf_{(1,\, d_1-e)}\, {}_n\onebf^{\mathbf{ss}}_{1,\, d_2}\otimes {}_n\onebf^{\mathbf{ss}}_{1,\, d_1-e}\big)=\\[5pt]
\begin{cases}
\displaystyle\upsilon^{2-2g_X}\, \sum_{s\in\Z_{\geq 0}}\, \xi_s\, {}_n\onebf^{\mathbf{ss}}_{1,\, d_2+s\,n}\otimes {}_n\onebf^{\mathbf{ss}}_{1,\, d_1-s\,n} & \text{if $\nfrapart{d_2-d_1}=0$}\ ,\\[20pt]
\displaystyle\upsilon^{2-2g_X}\, \sum_{s\in\Z_{\geq0}}\, \upsilon^{-1} \xi_s^\circ \, {}_n\onebf^{\mathbf{ss}}_{1,\, d_2+s\,n}\otimes {}_n\onebf^{\mathbf{ss}}_{1,\, d_1-s\,n}& \\[15pt]
\displaystyle+\upsilon^{2-2g_X}\, \sum_{s\in\Z_{\geq0}}(\xi_s-q^{-1} \xi_s^\circ)\,{}_n\onebf^{\mathbf{ss}}_{1,\, d_2+s\,n+\nfrapart{d_1-d_2}}\otimes {}_n\onebf^{\mathbf{ss}}_{1,\, d_1-s\,n-\nfrapart{d_1-d_2}} & \text{if $\nfrapart{d_2-d_1}\neq 0$}\ .
\end{cases}
\end{multline}
Let us introduce the automorphism $\gamma^{\pm m}$ of $\Ubf_n^>$ defined as
\begin{align}
\gamma^{\pm m}\big({}_n\onebf_{1, \, d}^{\mathbf{ss}}\big)\coloneqq{}_n\onebf_{1, \, d\pm m}^{\mathbf{ss}}
\end{align}
for $m\in \Z$. Denote by $\gamma_i^\bullet$ the operator $\gamma^\bullet$ acting on the $i$-th component of the tensor product. Then by using Corollaries~\ref{cor:xie} and \ref{cor:xiecirc} we get
\begin{multline}
\omega_n^{\otimes\, 2}\tilde \Delta_{1,1}\big({}_n\onebf^{\mathbf{ss}}_{1,\, d_1}\cdot {}_n\onebf^{\mathbf{ss}}_{1,\, d_2}\big)=\\[10pt]
\begin{cases}
\displaystyle {}_n\onebf^{\mathbf{ss}}_{1,\, d_1}\otimes {}_n\onebf^{\mathbf{ss}}_{1,\, d_2}+ \upsilon^{2-2g_X}\, \sum_{s\in\Z_{\geq 0}}\, \xi_s\, \big(\gamma_1^{s\, n}\gamma_2^{-s\, n}\big)\, {}_n\onebf^{\mathbf{ss}}_{1,\, d_2}\otimes {}_n\onebf^{\mathbf{ss}}_{1,\, d_1}& \text{if $\nfrapart{d_2-d_1}=0$}\ ,\\[20pt]
\displaystyle {}_n\onebf^{\mathbf{ss}}_{1,\, d_1}\otimes {}_n\onebf^{\mathbf{ss}}_{1,\, d_2}+\upsilon^{2-2g_X}\, \sum_{s\in\Z_{\geq0}} \,\upsilon^{-1}\, \xi_s^\circ\,\big(\gamma_1^{s\, n}\gamma_2^{-s\, n}\big) \, {}_n\onebf^{\mathbf{ss}}_{1,\, d_2}\otimes {}_n\onebf^{\mathbf{ss}}_{1,\, d_1} & \\[15pt]
\displaystyle + \upsilon^{2-2g_X}\, \sum_{s\in\Z_{\geq0}} \, (\xi_s-q^{-1} \xi_s^\circ)\,\big(\gamma_1^{s\, n}\gamma_2^{-s\, n}\big)\,\big(\gamma_1\gamma_2^{-1}\big)^{n\left\{(d_1-d_2)/n\right\}}\, {}_n\onebf^{\mathbf{ss}}_{1,\, d_2}\otimes {}_n\onebf^{\mathbf{ss}}_{1,\, d_1} & \text{if $\nfrapart{d_2-d_1}\neq 0$}\ .
\end{cases}
\\[10pt]
=
\begin{cases}
\displaystyle {}_n\onebf^{\mathbf{ss}}_{1,\, d_1}\otimes {}_n\onebf^{\mathbf{ss}}_{1,\, d_2}+ \upsilon^{2-2g_X}\, \frac{\zeta_X(\gamma_1^n\gamma_2^{-n})}{\zeta_X(\upsilon^{-2}\gamma_1^n\gamma_2^{-n})}\, {}_n\onebf^{\mathbf{ss}}_{1,\, d_2}\otimes {}_n\onebf^{\mathbf{ss}}_{1,\, d_1}& \text{if $\nfrapart{d_2-d_1}=0$}\ ,\\[20pt]
\displaystyle {}_n\onebf^{\mathbf{ss}}_{1,\, d_1}\otimes {}_n\onebf^{\mathbf{ss}}_{1,\, d_2}+\upsilon^{2-2g_X}\, \upsilon^{-1}\, \frac{\zeta_X(\gamma_1^n\gamma_2^{-n})}{\zeta_X(\upsilon^{-2}\gamma_1^n\gamma_2^{-n})}\, \frac{1-\gamma_1^n\gamma_2^{-n}}{1-\upsilon^{-2}\gamma_1^n\gamma_2^{-n}} \, {}_n\onebf^{\mathbf{ss}}_{1,\, d_2}\otimes {}_n\onebf^{\mathbf{ss}}_{1,\, d_1} & \\[15pt]
\displaystyle + \upsilon^{2-2g_X}\, \frac{\zeta_X(\gamma_1^n\gamma_2^{-n})}{\zeta_X(\upsilon^{-2}\gamma_1^n\gamma_2^{-n})}\,\frac{1-\upsilon^{-2}}{1-\upsilon^{-2}\gamma_1^n\gamma_2^{-n}}\,\big(\gamma_1\gamma_2^{-1}\big)^{n\left\{(d_1-d_2)/n\right\}}\, {}_n\onebf^{\mathbf{ss}}_{1,\, d_2}\otimes {}_n\onebf^{\mathbf{ss}}_{1,\, d_1} & \text{if $\nfrapart{d_2-d_1}\neq 0$}\ .
\end{cases}
\end{multline}
Now, by using the assignment \eqref{eq:assignment}, the assertion follows from the same arguments of the corresponding result in the non-stacky case \cite[Proposition~1.6]{art:schiffmannvasserot2012}.
\end{proof}

The subspace $\mathbf{Sh}_{n, \, h(z)}^{\mathsf{vec}, \mathsf{rat}}$ of $\mathbf{Sh}_{n, h(z)}^{\mathsf{vec}}$ consisting of Laurent series which are expansions of rational functions form a subalgebra, which is moreover stable under the coproduct.
\begin{corollary}
$\Ubf_n^>$ is isomorphic to the subalgebra $\Sbf_{n, \, h_X(z)}^{\mathsf{vec}}\subset \mathbf{Sh}_{n, \, h_X(z)}^{\mathsf{vec},\mathsf{rat}}$ generated by the degree one component in $x$. Moreover, the algebra $\Omega_n(\Ubf^>)$ is isomorphic to the subalgebra generated by the degree one component in $x$ of the image of the morphism of algebras
\begin{align}
\Omega_n\colon  \mathbf{Sh}_{1, \, h(z)}^{\mathsf{vec}} \to \mathbf{Sh}_{n, \, h_X(z)}^{\mathsf{vec}} \ , \quad P(x_1, \ldots, x_r) \mapsto \, P(x_1, \ldots, x_r) \,\vec v_{0} \otimes \cdots \otimes \vec v_{0}\ . 
\end{align}
\end{corollary}

\subsubsection{Generic form}\label{sec:generic}

Observe that the positive spherical Hall algebra $\Ubf_n^>$ depends only on the genus $g_X$ of the curve $X$, the order $n$ of the root stack and the Weil numbers $\alpha_1, \ldots, \alpha_{2g_X}$ of $X$. Therefore it possesses a \textit{generic form} in the following sense (cf.\ \cite[Section~3.5]{art:lin2014}). Let us fix $g \geq 0$ and consider the torus
\begin{align}
T_g \coloneqq \{ ( \alpha_1, \ldots, \alpha_{2g} ) \in (\C^\ast)^{2g} \,\vert\, \alpha_{2i-1} \alpha_{2i} = \alpha_{2j-1} \alpha_{2j} , \, \forall i,j \} \ .
\end{align}
The Weyl group 
\begin{align}
W_g \coloneqq \Sfrak_g \ltimes (\Z/2\Z)^g
\end{align}
naturally acts on $T_g$ and the collection $(\alpha_1,\ldots,\alpha_{2g})$ defines a canonical element $\alpha_X$ in the quotient $T_g/\!\!\!/W_g$. Let $\R_g \coloneqq \Q[T_g]^{W_g}$ and let $\K_g$ be its localization at the multiplicative set generated by $\{ q^s -1\,\vert\, s \geq 1 \}$ where by definition $q (\alpha_1, \ldots, \alpha_{2g}) = \alpha_{2i-1} \alpha_{2i}$ for any $1 \leq i \leq g$. For any choice of a smooth geometrically connected projective curve $X$ over $k$ of genus $g$ there is a natural map $\K_g \to \C$, $f \mapsto f(\alpha_X)$. 

We can define a shuffle algebra $\mathbf{Sh}_{n, h_X(z), \, \K_g}^{\mathsf{vec}}$ over $\K_g$ and therefore $\Ubf_{n, \,\K_g}^>$. The (twisted) bialgebra structure and Green's scalar product both depend polynomially on the $\{ \alpha_1, \dots, \alpha_{2g} \}$ and hence may be defined over $\K_g$. Let $\Ubf_{n, \,\R_g}^>$ be the $\R_g$-subalgebra of $\Ubf_{n, \,\K_g}^>$ generated in degree one over $\R_g$. By construction, $\Ubf_{n, \,\R_g}^>$ is a torsion-free integral form of $\Ubf_{n,\, \K_g}^>$ in the sense that $\Ubf_{n, \,\R_g}^> \otimes_{\R_g} \K_g = \Ubf_{n, \, \K_g}^>$. Moreover, there exists a specialization map
\begin{equation*}
\Ubf_{n, \,\R_g}^> \to \Ubf_{n}^>\ , 
\end{equation*}    
to a fixed curve $X$ of genus $g$. Finally, also $\Ubf_{n}^0$ has an obvious generic form $\Ubf_{n, \, \K_g}^0$. We define $\Ubf_{n, \, \K_g}$ as the tensor product $\Ubf_{n, \, \K_g}^>\otimes\Ubf_{n, \, \K_g}^0$.

\subsection{Hecke operators on line bundles and the fundamental representation of $\Ubf_\upsilon(\glfrakhat(n))$}

Consider
\begin{align}
\DUbf_n^0\simeq \Ubf_\upsilon(\glfrakhat(n))\coloneqq \Hcal_n\otimes_{\Acal_n} \Ubf_\upsilon(\slfrakhat(n))
\end{align}
(cf.\ Formula~\eqref{eq:Uq(gl)}). We define its \textit{fundamental representation} as the $\widetilde{\Q}$-vector space
\begin{align}
\V_n=\bigoplus_{d \in \Z}\, \widetilde \Q\, \vec u_d \ ,
\end{align} 
with action given by
\begin{align}
F_i\bullet \vec u_d \coloneqq& \delta_{\nfrapart{d+i}, 0}\, \upsilon^{1/2}\, \vec u_{d+1}\ , \\[4pt]
E_i\bullet \vec u_d \coloneqq & \delta_{\nfrapart{d+i},1}\, \upsilon^{-1/2}\,\vec u_{d-1}\ , \\[4pt]
K_i^\pm\bullet \vec u_d \coloneqq& \upsilon^{\pm(\delta_{{\nfrapart{d+i}}, 0}-\delta_{\nfrapart{d+i}, 1})}\, \vec u_d\ ,\\[4pt]
{}_n Z_r\bullet \vec u_d \coloneqq & ({}_n Z_r, {}_n \theta_{0, rn})_G\, \vec u_{d+rn} \ ,
\end{align}
for $i\in \Z/n\Z$ and $r\in \Z\setminus\{0\}$. 
\begin{remark}
The action of $\Ubf_\upsilon(\slfrakhat(n))$ on $\V_n$ defined in \cite[Section~1.1]{art:kashiwaramiwastern1995} is equivalent to ours after a suitable normalization of the $F_i$'s and $E_i$'s and after exchanging the index $i$ with $-i\bmod{n}$.
\end{remark}

Consider the natural action of the Hall algebra $\Hbf_n^{\mathsf{tw}, \tor}$ on $\Hbf_n^{\bun}$ by means of \textit{Hecke operators} given by the formula
\begin{align}
\Hbf_n^{\mathsf{tw}, \tor}\otimes \Hbf_n^{\bun}\to \Hbf_n^{\bun}\ , \quad u_0\otimes u\mapsto u_0\bullet u\coloneqq\omega_n(u_0\, u)\ .
\end{align}
We have the following.
\begin{theorem}
Under the assignments \eqref{eq:F-K}, \eqref{eq:E}, and $\vec u_d\mapsto {}_n\onebf^{\mathbf{ss}}_{1,\, d}$ the action of $\DUbf_n^0$ on $\Ubf^>_n[1]$ coincides with the action of $\Ubf_\upsilon(\glfrakhat(n))$ on $\V_n$ after applying the $\C$-algebra automorphism\footnote{This automorphism can be found for example in \cite[Section~1]{art:beck1994}.}
\begin{align}
\Phi(E_i)=F_i\ , \quad \Phi(F_i)=E_i \ , \quad \Phi(K_i)=K_i\ , \Phi({}_n Z_r)={}_n Z_r \text{ and } \Phi(\upsilon)=\upsilon^{-1}\ .
\end{align} 
\end{theorem}
\begin{proof}
First, by Corollary~\ref{cor:heckeaction}, the action of $1_{{}_n \Scal_i}$ on $\Ubf_n^>[1]$ is given as
\begin{align}
1_{{}_n \Scal_i}\bullet {}_n\onebf^{\mathbf{ss}}_{1,\, d}=\delta_{\nfrapart{d+i}, 0}\, \upsilon^{-1}\, {}_n\onebf^{\mathbf{ss}}_{1,\, d+1}\ ,
\end{align}
for $d,i\in \Z$, $1\leq i\leq n$. Moreover, for any $\mathbf c=(c_1, \ldots, c_n) \in\Z^n$, the element $\kbf_{(0, \mathbf c)}\in \Kbf_n$ acts as
\begin{align}\label{eq:actionK}
\kbf_{(0, \mathbf c)}\bullet {}_n \onebf^{\mathbf{ss}}_{1, d}=\upsilon^{c_{n-\nfrapart{d}+1}-c_{n-\nfrapart{d}}}\, {}_n \onebf^{\mathbf{ss}}_{1, d}\ ,
\end{align}
where formally we set $c_{n+1}\coloneqq c_1$. Finally, for $r, d\in \Z$, $r>0$, we have
\begin{align}\label{eq:actionZr}
{}_n Z_r\bullet  {}_n \onebf^{\mathbf{ss}}_{1, d}=({}_n Z_r, {}_n \theta_{0, rn})_G\, \onebf^{\mathbf{ss}}_{1, d+rn}\ .
\end{align} 
Indeed, although by definition ${}_n Z_r\bullet  {}_n \onebf^{\mathbf{ss}}_{1, d} =\omega_n({}_n Z_r\,  {}_n \onebf^{\mathbf{ss}}_{1, d})=\omega_n([{}_n Z_r,\,  {}_n \onebf^{\mathbf{ss}}_{1, d}])$, by using arguments similar to those in the proof of \cite[Proposition~1.2]{art:schiffmannvasserot2012}, one can see that $[{}_n Z_r,\,  {}_n \onebf^{\mathbf{ss}}_{1, d}]\in \Ubf_n^>$ already. By using the Green pairing, we get \eqref{eq:actionZr}.

The action of $1_{{}_n \Scal_i}^-, \kbf_{(0, \mathbf c)}^-$ and ${}_n Z_r^-$ on $\Ubf_n^>[1]$ is given by
\begin{align}
1_{{}_n \Scal_i}^-\bullet {}_n\onebf^{\mathbf{ss}, +}_{1,\, d} &=\delta_{\nfrapart{d+i-1},0}\, {}_n\onebf^{\mathbf{ss}, +}_{1,\, d-1} \ ,\\[4pt]
\kbf_{(0, \mathbf c)}^-\bullet {}_n \onebf^{\mathbf{ss}, +}_{1, d} &=\upsilon^{c_{n-\nfrapart{d}}-c_{n-\nfrapart{d}+1}}\, {}_n \onebf^{\mathbf{ss}, +}_{1, d}\ ,\\[4pt]
{}_n Z_r^-\bullet  {}_n \onebf^{\mathbf{ss}, +}_{1, d} &=({}_n Z_r, {}_n \theta_{0, rn})_G\, \onebf^{\mathbf{ss}, +}_{1, d-rn}\ .
\end{align}
Thus the assertion follows.
\end{proof}

\subsection{Tensor and symmetric tensor representations of $\Ubf_\upsilon(\glfrakhat(n))$}\label{sec:fundamentalsymmetric}

We can extend the $\Ubf_\upsilon(\glfrakhat(n))$-action to $\Ubf^>_n[1]^{\otimes\, r}$, which we call \textit{$r$-th tensor representation} of $\Ubf_\upsilon(\glfrakhat(n))$, and to $\Ubf^>_n[1]^{\widehat \otimes\, r}$. We have the following.
\begin{proposition}
Let $r\in \Z_{>0}$. Then the symmetrization operator $\Psi_{n, r}$ 
\begin{align}
\Psi_{n, r}\colon \Ubf^>_n[1]^{\otimes\, r}\to \Ubf^>_n[1]^{\widehat \otimes\, r}\ , \quad  \vec u_1\otimes \cdots \otimes \vec u_r \mapsto J_r(\vec u_1\star \cdots \star \vec u_r)
\end{align}
is a $\Ubf_\upsilon(\glfrakhat(n))$-intertwiner. 
\end{proposition}
We call the image of $\Psi_{n, r}$, which is $\Ubf^>_n[r]$, the \textit{symmetric tensor representation} of $\Ubf_\upsilon(\glfrakhat(n))$ (of genus $g_X$).

\bigskip\section{Relations between different Hall algebras}\label{sec:hallalgebramn}

Let $m,n\in \Z_{\geq 2}$ be such that $n\vert m$. The exact functor  $\pi_{m,n}^\ast$ (cf.\ Section~\ref{sec:infinity}) induces injective maps $\Ksfnum(X_n) \hookrightarrow \Ksfnum(X_m)$ and $\mathcal{M}_{\alpha,n} \hookrightarrow \mathcal{M}_{\alpha,m}$, which we will still denote $\pi_{m,n}^\ast$. Consider the Hall algebras $\Hbf_n$ and $\Hbf_m$. The maps $\pi_{m,n}^\ast$ give rise to an injective map
\begin{align}
\Omega_{m,n}\colon \Hbf_n\hookrightarrow \Hbf_m 
\end{align}
defined by
\begin{align}
\Omega_{m,n}(f) \coloneqq (\pi_{m,n}^\ast)_!(f)
\end{align}
(extension by zero) or, at the level of characteristic functions,
\begin{align}
\Omega_{m,n}(1_\Fcal)\coloneqq 1_{\pi_{m,n}^\ast(\Fcal)}\ , \ \text{ for all $\Fcal \in \Coh(X_n)$} \ .
\end{align}
It is easy to see that $\Omega_{m,n}$ is a homomorphism of algebras and extends to a similar homomorphism $\Omega_{m,n}\colon \Hbf_n^{\mathsf{tw}}\hookrightarrow \Hbf_m^{\mathsf{tw}}$. Moreover, $\Omega_m=\Omega_{m,n}\circ \Omega_n$ and $\Omega_{m,n}\circ T_n^{\, i}=T_m^{\, i\,m/n}\circ \Omega_{m,n}$.

We have (cf.\ Section~\ref{sec:different-coh})
\begin{align}\label{eq:different-coh}
\Omega_{m,n}\big(1_{{}_n \Scal_i^{(j)}}\big)=1_{{}_{m}\Scal_{i\, m/n}^{(j\, m/n)}}\ , \ \Omega_{m,n}\big({}_n \theta_{0,\, e}\big)={}_m\theta_{0,\, e\,m/n} \ , \  \Omega_{m,n}\big({}_n\onebf^{\mathbf{ss}}_{1,\, d}\big)={}_m\onebf^{\mathbf{ss}}_{1,\, d\,m/n}\ .
\end{align}
By restriction, there are induced injective algebra homomorphisms $\Hbf_{n, p_n}^{\mathsf{tw}, \tor}\hookrightarrow \Hbf_{m, p_m}^{\mathsf{tw}, \tor}$, $\Ubf_n^0\hookrightarrow \Ubf_m^0$, $\Ubf_n^>\hookrightarrow \Ubf_m^>$ and $\Ubf_n\hookrightarrow \Ubf_m$.

Finally, the following morphism of shuffle algebras
\begin{align}
\mathbf{Sh}_{n, \, h_X(z)}^{\mathsf{vec}} &\to \mathbf{Sh}_{m, \, h_X(z)}^{\mathsf{vec}} \ ,\\
P(x_1, \ldots, x_r)\, \vec v_{j_1} \otimes \cdots \otimes \vec v_{ j_r} &\mapsto P(x_1^{m/n}, \ldots, x_r^{m/n})\, \vec v_{j_1 \,m/n} \otimes \cdots  \otimes \vec v_{j_r\, m/n}\ ,
\end{align}
fits into the commutative diagram
\begin{align}
  \begin{tikzpicture}[xscale=1.5,yscale=-1.2]
    \node (A0_0) at (0, 0) {$\Ubf_n^>$};
    \node (A0_2) at (2, 0) {$\Ubf_m^>$};
    \node (A1_1) at (1, 1) {$ $};
    \node (A2_0) at (0, 2) {$\mathbf{Sh}_{n, \, h_X(z)}^{\mathsf{vec}}$};
    \node (A2_2) at (2, 2) {$\mathbf{Sh}_{m, \, h_X(z)}^{\mathsf{vec}}$};
    \node (Comma) at (2.7, 1) {$.$};
    \path (A0_0) edge [->]node [left] {$\scriptstyle{}$} (A2_0);
    \path (A0_0) edge [->]node [auto] {$\scriptstyle{\Omega_{m,n}}$} (A0_2);
    \path (A0_2) edge [->]node [auto] {$\scriptstyle{}$} (A2_2);
    \path (A2_0) edge [->]node [auto] {$\scriptstyle{}$} (A2_2);
  \end{tikzpicture} 
\end{align}

There is a dual surjective map
\begin{align}
\Omega^{m,n}\colon \Hbf_m\to \Hbf_n
\end{align}
defined by
\begin{align}
\Omega^{m,n}(f)\coloneqq(\pi_{m,n}^\ast)^\ast(f)
\end{align}
or, at the level of characteristic functions
\begin{align}
\Omega^{m,n}(1_{\mathcal{G}})\coloneqq
\begin{cases}
1_{\mathcal{F}}& \textit{if }\; \pi_{m,n}^\ast(\Fcal)\simeq \Gcal\ , \\[4pt]
0 & \textit{if }\;\nexists \;\Fcal ,\; \;\Gcal \simeq  \pi_{m,n}^\ast(\Fcal)\ .
\end{cases}
\end{align}
The map $\Omega^{m,n}$ is a morphism of coalgebras, and induces a similar morphism $ \Omega^{m,n}\colon \Hbf_m^{\mathsf{tw}}\to \Hbf_n^{\mathsf{tw}}$. Note that by construction $\Omega^{m,n} \circ \Omega_{m,n}=Id$, i.e. $\Omega_{m,n}$ is a section of $\Omega^{m,n}$.

\bigskip\section{Hall algebra of the infinite root stack over a curve}\label{sec:hallalgebrainfinite}

\subsection{Preliminaries on Hall algebras}\label{sec:Hallalgebrainfty}

As before, we let $k=\F_q$ and let $X$ be a smooth geometrically connected projective curve over $k$. Let $\pi_\infty\colon X_\infty\to X$ be the infinite root stack $\displaystyle \lim_{\genfrac{}{}{0pt}{}{\longleftarrow}{n}}\, X_n$ introduced in Section~\ref{sec:infiniterootstack}.

Let $(r,f)\in \Ksfnum(X_\infty)\simeq \Z\oplus\Z_0^{S^1_\Q}$ and denote by $\Mcal_{(r,f)}$ the set of isomorphism classes of coherent sheaves on $X_\infty$ of class $(r, f)$. The Hall algebra of $X_\infty$ is, as vector space,
\begin{align}
\Hbf_\infty\coloneqq\bigoplus_{(r,f) \in \Ksfnum(X_\infty)}\, \Hbf_\infty[r, f]\ ,
\end{align}
where
\begin{align}
 \Hbf_\infty[r, f]\coloneqq\{x\colon \Mcal_{(r,f)}\to \widetilde \Q\, \vert\, \text{$\mathsf{supp}(x)$ is finite}\}=\bigoplus_{\Fcal\in  \Mcal_{(r,f)}}\,  \widetilde \Q \, 1_\Fcal\ .
\end{align}
Here $1_{\Fcal}$ denotes the characteristic function of $\Fcal\in \Mcal_{(r,f)}$. The product on $\Hbf_\infty$ is defined as in Equation \eqref{eq:Hallprod}. We define $\Hbf_\infty^{\mathsf{tw}}$ as $\Hbf_\infty\otimes_{\widetilde \Q}\Q[\Ksfnum(X_\infty)]$, with multiplication defined as
\begin{align}
\kbf_\alpha\kbf_\beta= \kbf_{\alpha+\beta}\ , \kbf_0=1\ , \kbf_\alpha 1_{\Fcal} \kbf_\alpha^{-1}= \upsilon^{(\alpha, \overline \Fcal)}\,1_\Fcal\ ,
\end{align}
for $\Fcal\in \Coh(X_\infty)$ and $\alpha, \beta\in \Ksfnum(X_\infty)$. 
\begin{remark}
Since for any $n,m\in \Z$, $n\vert m$, the functor $\pi_{m,n}^\ast$ is fully faithful and the relation \eqref{eq:Ext1different} for Ext-groups hold, we have 
\begin{align}
\Hbf_\infty \simeq \lim_{\genfrac{}{}{0pt}{}{\to}{n}}\,\Hbf_n \quad\text{and}\quad \Hbf_\infty^{\mathsf{tw}} \simeq \lim_{\genfrac{}{}{0pt}{}{\to}{n}}\,\Hbf_n^{\mathsf{tw}}
\end{align}
in the category of associative algebras. The exact functor $\pi_{\infty,n}^\ast$ gives rise, as before, to maps $\Omega_{\infty,n}\colon \Hbf_n \hookrightarrow \Hbf_\infty$ and $\Omega^{\infty,n}\colon \Hbf_{\infty} \to \Hbf_n$.
\end{remark}

In order to define a coproduct, we need to introduce a suitable completion. This is clear from the fact that it is the \textit{inverse} limit
$\underset{\leftarrow}{\lim}\;\Hbf_n$ with respect to the maps $\Omega^{m,n}$ which carries a natural coalgebra structure (see Section~\ref{sec:hallalgebramn}). Note that the sections $\Omega_{m,n}$ provide an embedding $\Hbf_{\infty} \hookrightarrow \underset{\leftarrow}{\lim}\;\Hbf_n$ and we could have defined the coproduct directly in this fashion. To make things more explicit, we will instead follow the same setting as in Section~\ref{sec:hallalgebran}.

Define
\begin{align}\label{eq:Hallcompletion-infinity}
\widehat \Hbf_\infty\coloneqq\bigoplus_{(r,f)\in \Ksfnum(X_\infty)}\, \widehat \Hbf_\infty[r, f]\quad\text{and}\quad \widehat \Hbf_\infty[r, f]\coloneqq\{x\colon \Mcal_{(r,f)}\to \widetilde \Q\, \}\ ,
\end{align}
We will identify elements in $\widehat\Hbf_\infty[r,f]$ with (possibly infinite) series
\begin{align}
\sum_{\overline \Fcal=(r,f)}\, a_\Fcal 1_{\Fcal}
\end{align}
with $a_\Fcal\in \widetilde \Q$. Similarly, we define
\begin{align}
\Hbf_\infty[r, f]\widehat{\otimes}\Hbf_\infty[s, g]\coloneqq \{f\colon \Mcal_{(r,f)}\times \Mcal_{(s,g)} \to \widetilde{\Q}\}
\end{align}
and
\begin{align}
\big(\Hbf_\infty\widehat{\otimes}\Hbf_\infty\big)[m, h]&\coloneqq\prod_{(m,\, h)=(r,f)+(s,g)}\, \Hbf_\infty[r,f]\widehat{\otimes}\Hbf_\infty[s,g]\ , \\[2pt]
\Hbf_\infty\widehat{\otimes}\Hbf_\infty&\coloneqq\bigoplus_{(m,\, h)\in \Ksfnum(X_\infty)}\, \big(\Hbf_\infty\widehat{\otimes}\Hbf_\infty\big)[m, h] \ . 
\end{align}
Denote by ${}_n \Delta$ the coproduct introduced in \eqref{eq:Hallcoprod}. For $(r, f), (r_1, f_1), (r_2, f_2)\in \Ksfnum(X_\infty)$, such that $(r, f)=(r_1, f_1)+(r_2, f_2)$, and $x\in \Hbf_\infty[r, f]$ we define
\begin{align}
\Delta_{(r_1, f_1), (r_2, f_2)}(x)\coloneqq\lim_{\genfrac{}{}{0pt}{}{\leftarrow}{n}}\, {}_n \Delta_{(r_1, \mathbf{d}_{f_1}), (r_2, \mathbf d_{f_2})}(\Omega^{\infty,n}(x)) \in\Hbf_\infty[r_1, f_1]\widehat{\otimes}\Hbf_\infty[r_2, f_2]\ .
\end{align}
Thus, we define the coproduct as
\begin{align}
\Delta\colon \Hbf_\infty[r, f]&\to \prod_{(r_1, f_1)+(r_2, f_2)=(r, f)}\, \Hbf_\infty[r_1, f_1]\widehat{\otimes}\Hbf_\infty[r_2, f_2]\ ,\\[4pt]
 x&\mapsto \Delta(x)\coloneqq\sum_{(r_1, f_1)+(r_2,f_2)=(r, f)}\, \Delta_{(r_1, f_1), (r_2, f_2)}(x)\ ,
\end{align}
 and we extend to the whole $\Hbf_\infty$.
\begin{proposition}
The following properties hold:
\begin{itemize}\itemsep0.1cm
\item $\widehat \Hbf_\infty$ and $\Hbf_\infty\widehat{\otimes}\Hbf_\infty$ are naturally equipped with the structure of associative algebras.
\item $\Delta_{(r, f),(s, g)}\big(\Hbf_\infty[(r, f)+(s, g)]\big)\subset \Hbf_\infty[r, f]\otimes\Hbf_\infty[s, g]$.
\item The coproduct $\Delta$ takes values in $\Hbf_\infty\widehat{\otimes}\Hbf_\infty$; it extends to a coassociative coproduct $\Delta\colon \widehat \Hbf_\infty\to \Hbf_\infty\widehat{\otimes}\Hbf_\infty$.
\end{itemize}
\end{proposition}
\begin{proof} The  first and the second points of the Proposition are both consequences of Corollaries~\ref{cor:finiteness-infty-one} and \ref{cor:finiteness-infty-two}. The third part of the Proposition is obvious.
\end{proof}


There exist an extended coproduct $\tilde \Delta$ and a Green pairing $(\cdot, \cdot)_G$ such that $(\Hbf_\infty^{\mathsf{tw}}, \cdot, \tilde \Delta)$ is a Hopf bialgebra. Therefore, there exist a Drinfeld double $\tildeDHbf_\infty$ and its reduced version $\DHbf_\infty$. Similarly, we can define an extended version $\widehat \Hbf_\infty^{\mathsf{tw}}$ of the completion $\widehat \Hbf_\infty$ such that $(\widehat \Hbf_\infty^{\mathsf{tw}}, \cdot, \tilde \Delta)$ is a topological bialgebra.

Thanks to Formula \eqref{eq:sn-n}, we have an isomorphism of algebras $T_\infty^{\, x}\colon \Hbf_\infty\to \Hbf_\infty$ for any $x\in \Q\cap [0,1[$ induced by $T_n^{\, \nfrapart{d}}\colon \Hbf_n\to \Hbf_n$, if $x=d/n$ with $\mathsf{gcd}(d,n)=1$.

\subsection{Hecke algebra and $\Ubf_\upsilon(\mathfrak{sl}(S^1_\Q))$}

Let $\Tor(X_\infty)$ be the category of zero-dimensional coherent sheaves on $X_\infty$ and
\begin{align}
\Hbf_\infty^{\tor}\coloneqq\bigoplus_{\Tcal\in \Tor(X_\infty)}\, \widetilde \Q 1_\Tcal \quad\text{and}\quad \Hbf_\infty^{\mathsf{tw}, \tor}\coloneqq\Hbf_\infty^{\tor}\otimes_{\widetilde \Q} \Kbf_\infty\ .
\end{align}
Define for any $x\in X_\infty$
\begin{align}
\Hbf_{\infty, x}^{\tor}\coloneqq\bigoplus_{\Tcal\in \Tor_{x}(X_\infty)}\, \widetilde\Q\, 1_\Tcal \quad\text{and}\quad \Hbf_{\infty, x}^{\mathsf{tw}, \tor}\coloneqq\Hbf_{\infty, x}^{\tor}\otimes_{\widetilde \Q} \Kbf_\infty\ .
\end{align}
The decomposition of $\Tor(X_\infty)$ over the points of $X_n$ gives a decomposition at the level of Hall algebras
\begin{align}
\Hbf_\infty^{\tor}=\bigotimes_{x\in X_\infty}\, \Hbf_{\infty, x}^{\tor} \quad\text{and}\quad \Hbf_\infty^{\mathsf{tw}, \tor}=\bigotimes_{x\in X_\infty}\,\Hbf_{\infty, x}^{\mathsf{tw}, \tor} \ .
\end{align}

\subsubsection{Hecke algebra at the stacky point and $\Ubf_\upsilon(\mathfrak{sl}(S^1_\Q))$}\label{sec:Heckealgebrastackyinfinity}

The elements $\upsilon^{1/2}\, 1_{{}_n\Scal_{i}^{(\ell)}}$  generate the whole $\Hbf_{n, p_n}^{\tor}$, for $i\in\{1, \ldots, n\}$ and $\ell\in \Z_{\geq 1}$. For any two positive integers $n, k$, we have 
\begin{align}
\Omega_{kn, n}\colon \Hbf_{n, p_n}^{\tor} \to \Hbf_{k n, p_{k n}}^{\tor} \ , \upsilon^{1/2}\, 1_{{}_n\Scal_{i}^{(\ell)}} \mapsto \upsilon^{1/2}\,1_{{}_{k n}\Scal_{k i}^{(k \ell)}} \ .
\end{align} 
The pair $(\Hbf_{n, p_n}^{\tor}, \Omega_{kn, n})$ forms a directed system where the order is given by divisibility. The corresponding direct limit, $\displaystyle \lim_{\to}\, \Hbf_{n, p_n}^{\tor} $ is the Hall algebra $\Hbf_{\infty, p_\infty}^{\tor}$. One can argue similarly for $\Hbf_{n, p_n}^{\mathsf{tw}, \tor}$ and have  $\displaystyle \lim_{\to}\, \Hbf_{n, p_n}^{\mathsf{tw},\tor}  \simeq \Hbf_{\infty, p_\infty}^{\mathsf{tw}, \tor}$. In this section, we will investigate in detail these algebras.

For a strict rational interval $J=[a,b[ \subset S^1_\Q$, with $a<b$, define
\begin{align}
E_J\coloneqq\upsilon^{1/2}\, 1_{\Scal_J}\ .
\end{align}
where $\Scal_J$ is defined in Formula~\eqref{eq:SJ}. Similarly, for any rational interval $I \subseteq S^1_\Q$, define
\begin{align}
K_I^{\pm 1}\coloneqq\kbf_{(0,\pm\chi_I)}\ .
\end{align}
Let $\Cbf_\infty$ be the subalgebra of $\Hbf_{\infty, p_\infty}^{\mathsf{tw},\tor}$ generated by the $E_J$'s and by $K_I^\pm$'s for all strict rational intervals $J$ and all rational intervals $I$. By the previous discussion, 
\begin{align}
\Cbf_\infty \simeq \lim_\to \Cbf_n \simeq \lim_\to \Ubf^+_\upsilon\big(\slfrakhat(n)\big)\ .
\end{align}
\begin{theorem}\label{thm:Cinfty}
The algebra $\Cbf_\infty$ is isomorphic to the algebra generated by elements $E_J$ and $K_I^\pm$ for all strict rational intervals $J$ and all rational interval $I$ in $S^1_\Q$ subject to the following relations:
\begin{itemize}\itemsep0.4cm
\item \textit{Drinfeld-Jimbo relations}: for any two rational intervals $I, J$, with $J$ strict,
\begin{align}\label{eq:Serre1-I}
K_{\emptyset}=1\ ,\ K_I\, E_{J}\, K_I^{-1}= \upsilon^{(\chi_I, \chi_{J})}\,E_{J} \ ;
\end{align}
\item \textit{join relations}:
\begin{itemize}\itemsep0.2cm
\item for any pair of strict rational intervals $J_1, J_2$ of the form $J_1=[a,b[$ and $J_2=[b,c[$ such that $J_1\cup J_2$ is again an interval,
\begin{align}\label{eq:Serre1-II}
K_{J}\, K_{{J'}}= K_{{J\cup J'}}\ ;
\end{align}
\item for any pair of strict rational intervals $J_1, J_2$ of the form $J_1=[a,b[$ and $J_2=[b,c[$ such that $J_1\cup J_2$ is again a strict rational interval,
\begin{align}\label{eq:Joining1}
E_{J_1\cup J_2}=\upsilon^{1/2}\, E_{J_1}\, E_{J_2}-\upsilon^{-1/2}\, E_{J_2}\, E_{J_1} \ ;
\end{align}
\end{itemize}
\item \textit{nest relations}:
\begin{itemize}\itemsep0.2cm
\item for any two strict rational intervals $J_1, J_2$ such that $\overline{J_1} \cap\overline{J_2} =\emptyset$, 
\begin{align}\label{eq:Serre3}
\left[E_{J_1}, E_{J_2}\right]=0\ ;
\end{align}
\item for any two strict rational intervals $J_1, J_2$ such that $J_1\subset J_2$,
\begin{align}\label{eq:nest}
\upsilon^{\langle \chi_{J_1}, \chi_{J_2}\rangle} \, E_{J_1}\, E_{J_2}=\upsilon^{\langle \chi_{J_2}, \chi_{J_1}\rangle} \, E_{J_2}\, E_{J_1} \ .
\end{align}
\end{itemize}
\end{itemize}
\end{theorem}
\begin{proof}
Let $\Acal$ be the algebra generated by $\{E_J, K_I^\pm\,\vert \, I, J \subset S^1_\Q\ , J\neq S^1_\Q\}$ modulo the relations \eqref{eq:Serre1-I}, \eqref{eq:Serre1-II}, \eqref{eq:Joining1}, \eqref{eq:Serre3}, and \eqref{eq:nest}. Let us prove that the following Serre relations automatically hold in $\Acal$:
\begin{align}\label{eq:Serre2}
\begin{aligned}
E_{J_1}^2\, E_{J_2} - [2]_\upsilon\, E_{J_1}\, E_{J_2}\, E_{J_1} + E_{J_2}\, E_{J_1}^2 &=0\ ;\\
E_{J_2}^2\, E_{J_1} - [2]_\upsilon\, E_{J_2}\, E_{J_1}\, E_{J_2} + E_{J_1}\, E_{J_2}^2 &=0\ ;
\end{aligned}
\end{align}
Indeed, let $J_1=[a,b[$, $J_2=[b,c[$ be strict rational intervals such that $J_1 \cup J_2$ is strict: by multiplying first by the left Formula  \eqref{eq:Joining1} by $\upsilon^{-1/2}\, E_{J_1}$ and then by the right the same formula by $\upsilon^{1/2}\, E_{J_1}$, we get
\begin{align}
E_{J_1}^2\, E_{J_2} -\upsilon^{-1}\, E_{J_1}\, E_{J_2}\, E_{J_1} &= \upsilon^{-1/2}\, E_{J_1}\, E_{J_1 \cup J_2}\ ,\\
E_{J_2}\, E_{J_1}^2-\upsilon\, E_{J_1}\, E_{J_2}\, E_{J_1} &=-\upsilon^{1/2}\, E_{J_1 \cup J_2}\,  E_{J_1}\ .
\end{align}
In the second formula, by using \eqref{eq:nest} for the intervals $J_1$ and $J_1\cup J_2$ we get $\upsilon^{1/2}\, E_{J_1 \cup J_2}\,  E_{J_1}=\upsilon^{-1/2}\,  E_{J_1}\, E_{J_1 \cup J_2}$. Then we obtain the first of the relations \eqref{eq:Serre2}. The second one can be derived in a similar fashion.

There is an obvious map $\phi\colon \Acal \to \Cbf_\infty$, which is surjective. To prove that $\phi$ is injective, let us consider the subalgebra $\Acal_n \subset \Acal$ generated by $E_J, K_{J}^\pm$ for $J=[a,b[$, $a,b \in \frac{1}{n}\Z$. Then $\Acal=\bigcup_n \Acal_n$ (and $\Acal_n \subseteq \Acal_{\ell n}$ for all $n, \ell\in \Z$). It is enough to show that $\phi\vert_{\Acal_n}$ is injective for all $n$, but by \eqref{eq:Joining1} we observe that $\Acal_n$ is generated by $E_{[p/n,(p+1)/n[}, K_{{[(p/n,(p+1)/n[}}^\pm$ for $p \in \{0, \ldots, n-1\}$ and by \eqref{eq:Serre1-I}, \eqref{eq:Serre1-II}, \eqref{eq:Serre2}, and \eqref{eq:Serre3} we have a map $\Cbf_n \to \Acal_n$ which is a section of $\phi\vert_{\Acal_n}$ and the assertion follows.
\end{proof}
\begin{rem}
One of the consequences of the previous theorem is that $\Cbf_\infty$ is only defined by quadratic (!) relations. \hfill$\triangle$
\end{rem}

A direct computation gives the following formulas for the coproduct $\tilde \Delta$:
\begin{align}\label{eq:coproductCinfty}
\begin{aligned}
\tilde \Delta(K_I)&=K_I\otimes K_I\ ,\\
\tilde \Delta (E_{[a,\, b[})&=E_{[a,\, b[} \otimes 1 + \sum_{a<c<b} \upsilon^{-1/2}\, (\upsilon-\upsilon^{-1})\, E_{[a,\, c[}\, K_{[c,\, b[} \otimes E_{[c,\, b[} + K_{{[a,\, b[}} \otimes E_{[a,\, b[} \ ,
\end{aligned}
\end{align}
(here the sum on the right-hand-side runs over all possible \textit{rational values} $c\in [a, b[$). In addition, the Green pairing gives a non-degenerate Hopf pairing:
\begin{align}\label{eq:GreenpairingCinfty}
(E_J\, K_I, E_{J'}\, K_{{I'}})_G=\frac{\delta_{J, J'}}{\upsilon-\upsilon^{-1}}\, \upsilon^{(\chi_I,\chi_{I'})}\ ,
\end{align}
for any rational intervals $J, J', I, I'\subset S^1_\Q$, with $J, J'$ strict. Thus, $\Cbf_\infty$ is a topological Hopf algebra.

Recall that for any $n\geq 2$, $\Zbf_n\otimes_{\widetilde \Q} \Cbf_n \simeq \Hbf_{n, p_n}^{\mathsf{tw}, \tor}$, where $\Zbf_n\simeq \widetilde \Q[{}_n z_1, {}_n z_2, \ldots, {}_n z_r, \ldots]$ is the center of $\Hbf_{n, p_n}^{\mathsf{tw}, \tor}$ (cf.\ Section~\ref{sec:Heckealgebrastacky}). It is natural to wonder what happens for $\Hbf_{\infty, p_\infty}^{\mathsf{tw}, \tor}$.
\begin{proposition}\label{prop:trivialcenter}
The center of $\Hbf_{\infty, p_\infty}^{\mathsf{tw}, \tor}$ is trivial.
\end{proposition}
\begin{proof}
It is enough to define for any $x\in \Hbf_{\infty, p_\infty}^{\mathsf{tw}, \tor}$ an element $y\in \Hbf_{\infty, p_\infty}^{\mathsf{tw}, \tor}$ such that $[x,y]\neq 0$. Assume that $x=\Omega_{\infty, n}(x_n)$ for some $x_n \in \Hbf_n[\mathbf d]$, where $n$ is a positive integer and $\mathbf d\in \Z^n$. Among all the objects of $\Tor_{p_n}(X_n)$ over which $x$ takes a nonzero value, let us choose one which contains an indecomposable summand, say ${}_n \Scal_p^{(k)}$, of maximal length. Let us set $y\coloneqq1_{{}_{2n}\Scal_{2p+1}}$. Thus the product $y \cdot x$ take a nonzero value on a torsion sheaf on $X_\infty$, which has as a direct summand $\pi_{\infty, 2n}^\ast\big({}_{2n}\Scal_{2p+1}^{(2k+1)}\big)$. On the other hand, the product $x\cdot y$ takes nonzero values only on torsion sheaves on $X_\infty$ of the form $\pi_{\infty, 2n}^\ast\big({}_{2n} \Scal_{2p+1}\oplus \pi_{2n,n}^\ast\Mcal\big)$ for some $\Mcal\in \Tor_{p_n}(X_n)$. In particular, $x\cdot y\neq y\cdot x$.
\end{proof}

By passing to the reduced Drinfeld double of $\Cbf_{\infty}$, we obtain the following characterization.
\begin{theorem} 
The algebra $\DCbf_{\infty}$ is isomorphic to the algebra generated by elements $E_J, F_J, K_{J'}^{\pm 1}$, where $J$ (resp.\ $J'$) runs over all strict rational intervals (resp.\ rational intervals) in $S^1_\Q$, modulo the following set of relations:
\begin{itemize}\itemsep0.4cm
\item \textit{Drinfeld-Jimbo relations}:
\begin{itemize}\itemsep0.2cm
\item for any rational intervals $I, I_1, I_2$ and strict rational interval $J$,
\begin{align}
[K_{I_1},K_{I_2}] &=0\ ,\\
K_{I}\, E_{J}\, K_{I}^{-1} &=\upsilon^{( \chi_{I},\chi_{J})}\, E_{J}\ ,\\[2pt]
K_{I}\, F_{J}\, K_{I}^{-1} &=\upsilon^{-( \chi_{I},\chi_{J})}\, F_{J} \ ;
\end{align}
\item if $J_1, J_2$ are strict rational intervals such that $J_1 \cap J_2 = \emptyset$, 
\begin{align}\label{eq:commutatorEF-1}
[F_{J_1},E_{J_2}]=0\ ;
\end{align}
\item for any strict rational interval $J$,
\begin{align}\label{eq:commutatorEF-2}
[E_J,F_J]=\frac{K_J-K_J^{-1}}{\upsilon-\upsilon^{-1}} \ ;
\end{align} 
\end{itemize}
\item \textit{join relations}:
\begin{itemize}\itemsep0.2cm
\item if $J_1,J_2$ are strict rational intervals of the form $J_1=[a,b[$ and $J_2=[b,c[$ such that $J_1 \cup J_2$ is again a rational interval,
\begin{align}
K_{J_1}\, K_{J_2}=K_{J_1\cup J_2}\ ;
\end{align}
\item if $J_1, J_2$ are strict rational intervals  of the form $J_1=[a,b[$ and $J_2=[b,c[$ such that $J_1 \cup J_2$ is again a strict rational interval,
\begin{align}
\begin{aligned}
E_{J_1\cup J_2} &=\upsilon^{1/2}\,E_{J_1}\,E_{J_2} - \upsilon^{-1/2}\,E_{J_2}\, E_{J_1}\ ,  \\[2pt]
F_{J_1\cup J_2} &=\upsilon^{-1/2}\,F_{J_2}\,F_{J_1} - \upsilon^{1/2}\, F_{J_1}\, F_{J_2}  \ ;
\end{aligned}
\end{align}
\end{itemize}
\item \textit{nest relations}:
\begin{itemize}\itemsep0.2cm
\item if $J_1, J_2$ are strict rational intervals such that $\overline{J_1} \cap \overline{J_2} =\emptyset$,
\begin{align}
[E_{J_1},E_{J_2}]=0\quad\text{and}\quad [F_{J_1}, F_{J_2}]=0\ ;
\end{align}
\item if $J_1, J_2$ are strict rational intervals such that $J_1\subseteq J_2$,
\begin{align}
\begin{aligned}
\upsilon^{\langle \chi_{J_1}, \chi_{J_2}\rangle} \, E_{J_1}\, E_{J_2} &=\upsilon^{\langle \chi_{J_2}, \chi_{J_1}\rangle} \, E_{J_2}\, E_{J_1} \ ,\\
\upsilon^{\langle \chi_{J_1}, \chi_{J_2}\rangle} \, F_{J_1}\, F_{J_2} &=\upsilon^{\langle \chi_{J_2}, \chi_{J_1}\rangle} \, F_{J_2}\, F_{J_1} \ .
\end{aligned}
\end{align}
\end{itemize}
\end{itemize}
It is a topological Hopf algebra, with coproduct given by the following formulas:
\begin{align}
\tilde \Delta(K_J)&=K_J\otimes K_J\ ,\\
\tilde \Delta (E_{[a,\, b[})&=E_{[a,\, b[} \otimes 1 + \sum_{a<c<b} \upsilon^{-1/2}\, (\upsilon-\upsilon^{-1})\, E_{[a,\, c[}\, K_{[c,\, b[} \otimes E_{[c,\, b[} + K_{[a,\, b[} \otimes E_{[a,\, b[} \ ,\\
\tilde \Delta (F_{[a,\, b[})&=1\otimes F_{[a, \, b[} - \sum_{a<c<b} \upsilon^{-1/2}\, (\upsilon-\upsilon^{-1})\, F_{[c,\, b[}\otimes F_{[a,\, c[}\, K_{[c,\, b[}^{-1}  + F_{[a,\, b[}\otimes K_{[a,\, b[}^{-1}\ .
\end{align}
Here the sums on the right-hand-side run over all possible rational values $c\in [a, b[$.
\end{theorem}
\begin{proof}
Denote temporarily by $\Acal$ the algebra generated by $E_J, F_J, K_{J'}^{\pm 1}$ modulo the Drinfeld-Jimbo, join and nest relations. A direct computation shows that the assignment
\begin{align}
E_J\mapsto \upsilon^{1/2}\, 1_{\Scal_J}^+\ , \quad  F_J \mapsto -\upsilon^{1/2}\, 1_{\Scal_J}^-  \ , \quad  K_J\mapsto K_{J}^+ 
\end{align}
induces a surjective morphism $\Psi\colon \Acal \to \DCbf_{\infty}$.  By Theorem~\ref{thm:Cinfty}, the restriction of $\Psi$ to the subalgebras generated respectively by $\{E_J\}_J$ and by $\{F_J\}_J$ is injective. Thus we only need to check that the defining relations of the Drinfeld double follow from the Drinfeld-Jimbo, join and nest relations. In fact, it is enough to check the Drinfeld double relations for two elements $u$ and $v$ belonging to a family of generators, such as the $E_J$. It is not hard to see that these relations follow from \eqref{eq:commutatorEF-1} and \eqref{eq:commutatorEF-2} together with the join relations. 
\end{proof}
\begin{definition}
We call the Hopf algebra $\DCbf_{\infty}$ the \textit{quantum group $\Ubf_\upsilon(\slfrak(S^1_\Q))$}.
\end{definition}

\subsubsection{Completion of the Hecke algebra at the stacky point }\label{sec:completion}

By Proposition \ref{prop:trivialcenter}, one of the discrepancies between $\Hbf_{n, p_n}^{\mathsf{tw}, \tor}$, for a fixed positive integer $n\geq 2$, and $\Hbf_{\infty, p_\infty}^{\mathsf{tw}, \tor}$ is that the former has a nontrivial center, while the center of the latter is trivial. As seen in Corollary \ref{cor:DC_n}, the reduced Drinfeld double of $\Cbf_n$ realizes geometrically $\Ubf_\upsilon(\slfrakhat(n))$, while by Corollary \ref{cor:heckealgebrastackypoint-n} $\Ubf_\upsilon(\glfrakhat(n))$ is given by the reduced Drinfeld double of $\Hbf_{n, p_n}^{\mathsf{tw}, \tor}$: the center of $\Hbf_{n, p_n}^{\mathsf{tw}, \tor}$ forms the positive part of the quantum Heisenberg algebra of $\Ubf_\upsilon(\glfrakhat(n))$. In the limit $n \to \infty$ this quantum Heisenberg algebra still exists, but it does not lie in the center anymore.

As an attempt to recover a center, we introduce a metric completion of $\Hbf_{\infty, p_\infty}^{\mathsf{tw}, \tor}$. First, we define on $\Hbf_{\infty, p_\infty}^{\tor}$ a metric using the collection of subalgebras $\Hbf_{n, p_n}^{\tor}$. To unburden the notation, we will drop the second index $p_n, p_{\infty}$. Observe first that
$\Hbf_{n}^{\tor} \cdot \Hbf_{m}^{\tor} \subset \Hbf_{\mathsf{l.c.m.}(m,n)}^{\tor} \subset \Hbf_{mn, p_{mn}}^{\tor}$. 
\begin{definition}
We call \textit{valuation} of $x\in \Hbf_{\infty}^{\tor}$ the positive integer
\begin{align}
\mathsf{val}(x)\coloneqq\min\{n\, \vert \, x \in \Hbf_{n}^{\tor}\}\ .
\end{align}
\end{definition}
The function $\mathsf{val}\colon \Hbf_{\infty}^{\tor}\to \Z_{>0}$ is \textit{submultiplicative}, i.e., 
\begin{align}
\mathsf{val}(xy)\leq \mathsf{l.c.m.}\big(\mathsf{val}(x), \mathsf{val}(y)\big)\leq \mathsf{val}(x)\mathsf{val}(y)\ .
\end{align}
Let $\rho\in ]0,1[$. Define the metric on each $\Hbf_{\infty}^{\tor}[f]$, for $f\in \N_0^{S^1_\Q}$, by
\begin{equation}
d(x,y)\coloneqq\rho^{\mathsf{val}(x-y)}\ ,
\end{equation}
and denote by $\underline \Hbf_{\infty}^{\tor}[f]$ its completion. We set 
\begin{align}
\underline \Hbf_{\infty}^{\tor}\coloneqq\bigoplus_{f\in \N_0^{S^1_\Q}}\, \underline \Hbf_{\infty}^{\tor}[f]\ .
\end{align}
\begin{example}
For any $f\in \N_0^{S^1_\Q}$, the element $1_f\coloneqq\sum_\Mcal\, 1_\Mcal$, with the sum ranging over all objects in $\Tor_{}(X_\infty)$ of dimension $f$, belongs to $\underline \Hbf_{\infty}^{\tor}$.
\end{example}
\begin{proposition}\label{prop:completion}
The algebra structure on $\Hbf_{\infty}^{\tor}$ extends uniquely to an algebra structure on $\underline \Hbf_{\infty}^{\tor}$.
\end{proposition}
\begin{proof}
Let us fix $f_1,f_2 \in \N_0^{\Q/\Z}$ and set $f=f_1+f_2$. We need to prove that the multiplication map  $\Hbf_{\infty}^{\tor}[f_1] \otimes \Hbf_{\infty}^{\tor}[f_2] \to \Hbf_{\infty}^{\tor}[f]$ is uniformly continuous. In concrete terms, this means that we need to bound from below the valuation of a product in terms of the valuation of each term. For this, it is enough to show that for any object $\Mcal$ of dimension $f$, there exists $\ell \in \Z_{>0}$ such that if $1_\Mcal$ appears in a product $1_{\Mcal_1} \cdot 1_{\Mcal_2}$ with $1_{\Mcal_i} \in \Hbf_{\infty}^{\tor}[f_i]$, then $M_i \in \Hbf_{\ell}^{\tor}[f_i]$ for $i=1,2$. Let us assume that a representative of $\Mcal$ is a torsion sheaf on $X_n$ for some positive integer $n$, i.e., let us assume that $\Mcal \in \Tor_{p_n}(X_n)$. Define $\ell$ as the smallest integer in $n\,\Z$ for which $f_1, f_2$ are locally constant on every interval of the form $[p/\ell,(p+1)/\ell[$. Let $\Mcal_1,\Mcal_2$ be objects of $\Tor_{p_\infty}(X_\infty)$ of dimension $f_1,f_2$. If $\Mcal_1$ has not a representative in $\Tor_{p_\ell}(X_\ell)$, there exist indecomposable direct summands $\Scal_J, \Scal_K$ of $M_1$ such that $J=[a,b[, K=[b,c[$ with $b=p/q \notin (1/\ell)\,\Z$. This means that $\Mcal_1$ admits a nonzero morphism $\Mcal_1\to \pi_{\infty, q}^\ast\big({}_q\Scal_{q\{1/q+a\}}\big)$. Likewise, if $\Mcal_2 \notin \Tor_{p_\ell}(X_\ell)$ then there exist indecomposable direct summands $\Scal_{J'}, \Scal_{K'}$ of $\Mcal_2$ such that $J'=[a',b'[, K'=[b',c'[$ with $b=p'/q' \notin (1/\ell)\,\Z$, and this means that there exists a nonzero morphism $\pi_{\infty, q}^\ast\big({}_{q'}\Scal_{p'}\big)\to \Mcal_2$. If either $\Mcal_1$ or $\Mcal_2$ do not belong to $\Tor_{p_\ell}(X_\ell)$, then for any extension
\begin{align}
0\to \Mcal_2\to \Mcal'\to \Mcal_1\to 0\ ,
\end{align}
we either have $\Hom(\Mcal', \pi_{\infty, q}^\ast\big({}_q\Scal_{q\{1/q+a\}}) \neq 0$ or $\Hom(\pi_{\infty, q}^\ast\big({}_{q'}\Scal_{p'}\big) ,\Mcal') \neq 0$. In both cases we conclude that $\Mcal'$ does not belong to $\Tor_{p_\ell}(X_\ell)$, and thus \textit{a fortiori} not to $\Tor_{p_n}(X_n)$. In particular, we may not have $\Mcal \simeq \Mcal'$ and the assertion follows.
\end{proof}

\begin{lemma}
The composition subalgebra $\Cbf_\infty$ is dense in $\underline \Hbf_{\infty}^{\tor}$.
\end{lemma}
\begin{proof}
Let $n$ be a positive integer. By \cite[Lemma~4.3]{art:schiffmann2004}, $\Hbf_{n}^{\tor}$ is generated by $1_{{}_n \Scal_i}$, for $1\leq i\leq n$, and the elements $1_{{}_n \Scal_n^{(n)}},  1_{{}_n \Scal_n^{(2n)}}, \ldots$. Observe that in $\Hbf_{\infty}^{\tor}$ we have 
\begin{align}
\Omega_{\infty, n}\big(1_{{}_n \Scal_n^{(kn)}}\big)=\Omega_{\infty, n\ell}\big(1_{{}_{n\ell} \Scal_{n\ell}^{(kn\ell)}}\big)\ ,
\end{align}
for any pair of positive integers $n, \ell$. Let us simply denote by $x_k$ this element. The lemma will be proved once we show that $x_k$ belongs to the closure $\overline{\Cbf_\infty}$ of $\Cbf_\infty$ in $\underline \Hbf_{\infty}^{\tor}$ for any positive integer $k$. We will prove this by induction on $k$ and by some explicit computations. For $k=1$, we have
\begin{align}
x_1=\lim_{n \to \infty}\, E_{[0,\frac{1}{n}[}\,E_{[\frac{1}{n},1[}\ .
\end{align}
Assume that for any $i<k$, we have $x_i\in \overline{\Cbf_\infty}$. Then $\underline \Hbf_{\infty}^{\tor}[f]\subset \overline{\Cbf_\infty}$ for all $f<k\delta$. On the other hand, from 
\begin{align}
x_k=\lim_{n \to \infty}  \upsilon^{1/2}\, E_{[0,\frac{1}{n}[}\cdot 1_{{}_n\Scal_n^{(kn-1)}}\ ,
\end{align}
we deduce that $x_k\in \overline{\Cbf_\infty}$ as wanted.
\end{proof}
We denote by $\underline \Hbf_{\infty}^{\mathsf{tw}, \tor}$ the extended version of the completion $\underline \Hbf_{\infty}^{\tor}$.

Let $r\in \Z_{>0}$ and consider the sequence $\left\{\Omega_{\infty, n\ell}\big({}_{n\ell} c_r\big)\right\}_{\ell}$ in $\Hbf_{\infty}^{\mathsf{tw}, \tor}$. It follows from the definition of the ${}_m c_r$'s (see Equation \eqref{eq:nc_r}) that such a sequence is Cauchy since the functor $\pi_{m,k}^\ast$ preserves the square freeness of the socle of a torsion sheaf for any $m, k \in \Z_{>0}$, $k\vert m$. Thus, there exists the limit
\begin{align}
c_r\coloneqq\lim_{\ell\to +\infty}\, \Omega_{\infty, n\ell}\big({}_{n\ell} c_r\big)\in  \underline \Hbf_{\infty}^{\mathsf{tw}, \tor}[r\, \delta]\ .
\end{align}
Since the elements ${}_n c_1, {}_n c_2, \ldots$ commute and are algebraically independent for any positive integer $n$, the same holds for their respective limits $c_1, c_2, \ldots$. Thanks to Formula \eqref{eq:nz_r}, we can define for any $r\in \Z_{>0}$ the element $z_r\in  \underline \Hbf_{\infty}^{\mathsf{tw}, \tor}[r\, \delta]$. Also, the elements $z_1, z_2, \ldots$ commute and are algebraically independent.
\begin{proposition}\label{prop:center}
The center of $\underline \Hbf_{\infty}^{\mathsf{tw}, \tor}$ is $\widetilde \Q[z_1, z_2, \ldots]$.
\end{proposition}
\begin{proof}
For any $x \in \Hbf_{\infty}^{\mathsf{tw}, \tor}$ such that $x$ has a representative $x_n\in  \Hbf_{n}^{\mathsf{tw}, \tor}$ for some positive integer $m$, we have
\begin{align}
z_r\cdot \Omega_{\infty, n}(x_n)- \Omega_{\infty, n}(x_n)\cdot z_r=\lim_{\ell\to +\infty}\, \Omega_{\infty, n\ell}\big({}_{n\ell} z_r\cdot \Omega_{n\ell, n}(x_n)-\Omega_{n\ell, n}(x_n)\cdot {}_{n\ell} z_r\big)=0\ .
\end{align}
 This shows that $z_1, z_2, \ldots$ belong to the center of $\underline \Hbf_{\infty}^{\mathsf{tw}, \tor}$. Conversely, let $z\in \underline \Hbf_{\infty}^{\mathsf{tw}, \tor}$ be a central element. We may write $z=\lim_{k\to +\infty}\, {}_k z$ with ${}_k z\in \Hbf_{n}^{\mathsf{tw}, \tor}$. We claim that for any element $y\in \Hbf_{k}^{\mathsf{tw}, \tor}$, we have $y\cdot {}_k z={}_k z\cdot y$. Since $\Tor_{p_k}(X_k)$ is a full abelian subcategory of $\Tor_{p_\infty}(X_\infty)$, any short exact sequence
\begin{align}
0\to \Mcal' \to \Mcal\to \Mcal''\to 0
\end{align}
of objects of $\Tor_{p_\infty}(X_\infty)$, in which two among $\Mcal, \Mcal', \Mcal''$ have a representative in $\Tor_{p_k}(X_k)$, comes from a short exact sequence in $\Tor_{p_k}(X_k)$. In particular, all three objects $\Mcal, \Mcal', \Mcal''$ belong to $\Tor_{p_k}(X_k)$. This implies that the restrictions of $\Omega_{\infty, k}(y)\cdot z$ and $z\cdot \Omega_{\infty, k}(y)$ to objects of $\Tor_{p_k}(X_k)$ are equal to $y\cdot {}_k z$ and ${}_k z\cdot y$ respectively. Hence from the equality $\Omega_{\infty, k}(y)\cdot z=z\cdot \Omega_{\infty, k}(y)$ we deduce that $y\cdot {}_k z={}_k z\cdot y$, therefore ${}_k z$ is a central element of $ \Hbf_{k}^{\mathsf{tw}, \tor}$, and so  ${}_k z\in \widetilde \Q[{}_k z_1, {}_k z_2, \ldots]$ (cf.\ Section~\ref{sec:Heckealgebrastacky}). Thus $z\in \widetilde \Q[z_1, z_2, \ldots]$ since the restriction maps $\widetilde \Q[z_1, z_2, \ldots]\to \widetilde \Q[{}_k z_1, {}_k z_2, \ldots]$ are isomorphisms for all $k$.
\end{proof}

The coproduct, being uniformly continuous, extends naturally to $\underline \Hbf_{\infty}^{\mathsf{tw}, \tor}$ to yield a map
\begin{align}
\tilde \Delta\colon \underline \Hbf_{\infty}^{\mathsf{tw}, \tor}[f]\to \prod_{f_1+f_2=f}\, \underline \Hbf_{\infty}^{\mathsf{tw}, \tor} [f_1]\widehat{\otimes} \underline \Hbf_{\infty}^{\mathsf{tw}, \tor}[f_2] \ .
\end{align}
Equipped with this coproduct, $\underline \Hbf_{\infty}^{\mathsf{tw}, \tor}$ is a topological bialgebra. In particular,
\begin{align}
\tilde \Delta(z_r)=z_r\otimes 1+\kbf_{(0, r\, \delta)}\otimes z_r \ .
\end{align}

\begin{remark}\label{rem:completiongreen}
The extension of the Green pairing to $\underline \Hbf_{\infty}^{\mathsf{tw}, \tor}$ is problematic. For instance, the sequence of elements $\{x_\ell\}_\ell$, where
\begin{align}
x_\ell\coloneqq\sum_{n\leq l}\, 1_{{}_n \Scal_n^{(n-1)}}\oplus 1_{{}_n \Scal_{n-1}^{(1)}} \ ,
\end{align}
is Cauchy, hence it converges to an element in $\underline \Hbf_{\infty}^{\mathsf{tw}, \tor}[\delta]$. However, 
\begin{align}
(x_\ell, x_\ell)_G=\sum_{n\leq l}\, (\upsilon-\upsilon^{-1})^2=\ell\, (\upsilon-\upsilon^{-1})^2
\end{align}
does not converge. The elements $1_f$ for $f\in \N_0^{S^1_\Q}$, $f \neq 0$, are likewise of infinite norm.
\end{remark}
At this point, we could define a positive part of $\Ubf_\upsilon(\glfrak(S^1_\Q))$ as $\widetilde \Q[z_1, z_2, \ldots]\otimes_{\widetilde \Q} \Cbf_\infty$. Because of the remark above, it is not clear to us how to get a geometric definition of the whole $\Ubf_\upsilon(\glfrak(S^1_\Q))$ by considering a reduced Drinfeld double of the latter algebra.

We end this section by describing possible canonical and crystal bases for $\underline \Hbf_{\infty}^{\tor}$. For any positive integer $n$ and $f \in \N_0^{S^1_\Q}$, let ${}_n\Xcal_f$ stand for the stack of nilpotent $\overline{\F_q}$-representations of $\Z/n\Z$ of dimension ${}_n\mathbf d_f$ (cf.\ Formula \eqref{eq:df} for the definition of ${}_n\mathbf d_f$). For any $\ell\geq 1$ there is a canonical open embedding ${}_n\Xcal_f \subset {}_{n\ell}\Xcal_f$ induced by the exact functor $\pi_{n\ell, n}^\ast\colon \Coh(X_n) \to \Coh(X_{n \ell})$. In concrete terms,  $\pi_{n\ell, n}^\ast$ identifies the stack of representation of $\Z/n\Z$ of dimension ${}_n \mathbf d_f$ with the open substack of representations of $\Z/n\ell\Z$ of dimension ${}_{n\ell} \mathbf d_f$ for which the edges $k \to k+1$ are isomorphisms whenever $k \neq 0 \bmod n$. We obtain in this way a presentation of the stack of torsion sheaves on $X_\infty$ supported at the stacky point $p_\infty$ as the direct limit 
\begin{align}
\Xcal_f=\lim_{\genfrac{}{}{0pt}{}{\to}{n}}\, {}_n \Xcal_f\ .
\end{align}
Now let $\Mcal$ be an object of $\Tor_{p_\infty}(X_\infty)$ of dimension $f$, and let $\Osf_\Mcal$ be the substack of $\Xcal_f$ parametrizing sheaves isomorphic to $\Mcal$. There is a unique perverse sheaf $\Ical \Ccal\big(\Osf_\Mcal\big) \in \Db_{\mathsf{cst}}\big(\Xcal_f)$\footnote{By definition, $\displaystyle \Db_{\mathsf{cst}}\big(\Xcal_f)=\lim_{\genfrac{}{}{0pt}{}{\leftarrow}{n}}\, \Db_{\mathsf{cst}}\big({}_n \Xcal_f)$.}  whose restriction to each ${}_n\Xcal_f$ is equal to $\Ical \Ccal\big({}_n\Osf_\Mcal,\overline{\Q_\ell})$, where  ${}_n\Osf_\Mcal\subset {}_n\Xcal_f$ is the substack parametrizing objects isomorphic to $\Mcal$ (which may well be empty). We define the \textit{canonical basis} of $\underline \Hbf_{\infty}^{\tor}$ as the collection of elements $\Bbf\coloneqq\{\bbf_\Mcal \, \vert \, \Mcal \in \Tor_{p_\infty}(X_\infty)\}$, where
\begin{align}
\bbf_\Mcal=\sum_\Ncal\, \sum_i\, \upsilon^{-2i}\, \dim \big(H^i\big(\Ical \Ccal\big(\Osf_\Mcal\big)\vert_{\Osf_\Ncal}\big)1_\Ncal\in \underline  \Hbf_{\infty}^{\tor} \ .
\end{align}
Here the sum runs over all $\Ncal \in \Tor_{p_\infty}(X_\infty)$ satisfying $\underline{\dim}(\Ncal)=\underline{\dim}(\Mcal)$. The following is immediate.
\begin{proposition} 
The set $\Bbf$ is a topological basis of $\underline \Hbf_{\infty}^{\tor}$, i.e., it spans a dense subspace of $\underline \Hbf_{\infty}^{\tor}$.
\end{proposition}
\begin{remark}
There is a similar extension to the setting of $\underline \Hbf_{\infty}^{\tor}$ of the theory of crystal bases, as expounded by Kashiwara-Saito, and based on the geometry of the cotangent stack. Indeed, $T^\ast \Xcal_f$ (or its \textit{underived} component) is equal to a direct limit of open finite type substacks $T^\ast {}_n \Xcal_f$, and we may define a Lagrangian substack $\Lambda_f$ as the direct limit of the Lusztig nilpotent Lagrangian varieties ${}_n\Lambda_f \subset T^\ast {}_n\Xcal_f$. Then $\Lambda_f$ is a pure-dimensional stack, locally of finite type, and the set of irreducible components of $\Lambda=\bigsqcup_f \Lambda_f$ is in bijection with $\Bbf$ and carries the structure of a crystal graph with operators $\tilde{e}_J,\tilde{f}_J$ labelled by rational intervals $J \subsetneq S^1_\Q$. 
\end{remark}

\subsubsection{Hecke algebra of the infinite root stack over a curve and $\Ubf_\upsilon(\glfrak(S^1_\Q))$}\label{sec:Heckealgebrainfinite}

Let $\Ubf_\infty^0$ be the subalgebra of $\Hbf_\infty^{\mathsf{tw}}$ generated by $\Omega_{\infty, 1}(\Cbf)$ and $\Cbf_\infty$. One can easily show that it is a topological Hopf algebra, hence there exists a reduced Drinfeld double $\DUbf_\infty^0$.

As before, we can introduce a completion $\underline \Ubf_\infty^0$ which is an associative algebra by the same arguments as in the proof of  Proposition~\ref{prop:completion}. Let $r\in \Z_{>0}$. First note that the ${}_n Z_r$'s belong to $\Ubf_\infty^0$ (cf.\ the proof of \cite[Proposition~6.3]{art:burbanschiffmann2013}). The sequence $\{\Omega_{\infty, n\ell}\big({}_{n\ell} Z_r\big)\}_\ell$  in $\Ubf_\infty^0$ is Cauchy, hence there exists a limit
\begin{align}
Z_r\coloneqq\lim_{\ell\to +\infty}\, \Omega_{\infty, n\ell}\big({}_{n\ell} Z_r\big)\in \underline \Ubf_\infty^0\ .
\end{align} 
Since the elements ${}_n Z_1, {}_n Z_2, \ldots$ commute and are algebraically independent for any positive integer $n$, the same holds for their respective limits $Z_1, Z_2, \ldots$. In addition, the dimension of $Z_r$ is $r\, \delta$. By using the same arguments as in the proof of Proposition~\ref{prop:center}, the elements $Z_1, Z_2, \ldots$ generate the center of $\underline \Ubf_\infty^0$. 
%

\begin{rem}
A possible candidate for being $\Ubf_\upsilon(\glfrak(S^1_\Q))$ is $\DUbf_\infty^0$, but it is not true that we can decompose it as a commuting tensor product of a Heisenberg algebra and $\Ubf_\upsilon(\slfrak(S^1_\Q))$. On the other hand, as we will show in Section~\ref{sec:comparisons}, it contains as a subalgebra the quantum enveloping algebra $\Ubf_\upsilon(\glfrakhat(\infty))$ of the affinization of $\glfrak(\infty)$.

Another candidate for being at least a positive part of $\Ubf_\upsilon(\glfrak(S^1_\Q))$ is $\underline \Ubf_\infty^{0}$, but at the moment we do not know how to extend the Green pairing to the completion in order to take the reduced Drinfeld double of it, because of elements of infinite norm as the $Z_r$'s (see also Remark~\ref{rem:completiongreen}).\hfill $\triangle$
\end{rem}

\subsection{Shuffle algebra presentation of $\Ubf_\infty^>$}\label{sec:positivesphericalinfinite}

For $x\in \Q$, define
\begin{align}
\onebf^{\mathbf{ss}}_{1, \, x}\coloneqq\sum_{\genfrac{}{}{0pt}{}{M\in \Pic(X)}{\deg(M)=\lfloor x\rfloor}}\, 1_{\pi_\infty^\ast M\otimes \Lcal_\infty^{\otimes\, \{x\}}}\ .
\end{align}
Let $\Ubf_\infty^>$ be the subalgebra of $\Hbf_\infty^{\mathsf{tw}}$ generated by $\{\onebf^{\mathbf{ss}}_{1,\, x}\, \vert\, x\in \Q\}$. We call it the \textit{positive spherical algebra} of $\Hbf_\infty^{\mathsf{tw}}$. As before, we will give a \textit{shuffle presentation} of $\Ubf_\infty^>$.

For $y\in \Q_{\geq 0}$, define
\begin{align}
\theta_{0,\, y}\coloneqq\Omega_{\infty, m}\big({}_m \theta_{0,\, e}\big)
\end{align}
if $y=e/m$ with $\mathsf{gcd}(e,m)=1$. One can check that this quantity is well defined by Formulas \eqref{eq:different-coh}.
\begin{proposition}\label{prop:delta-infty}
Let $x\in\Q$ be such that $x=d/k$ with $\mathsf{gcd}(d,k)=1$. Then
\begin{align}
\tilde \Delta\big(\onebf^{\mathbf{ss}}_{1, \, x}\big)=\onebf^{\mathbf{ss}}_{1, \, x}\otimes 1 + \sum_{y\in\Q_{\geq 0}}\,
T_\infty^{\, \{x\}}\, \big(\theta_{0,\, y}\big)\,\kbf_{(1, \, x-y)}\otimes \onebf^{\mathbf{ss}}_{1, \, x-y}\ .
\end{align} 
\end{proposition}
\begin{proof}
By construction, $\onebf^{\mathbf{ss}}_{1, \, x}=\Omega_{\infty, k}({}_k \onebf^{\mathbf{ss}}_{1, \, d})$ if $x=d/k$ with $\mathsf{gcd}(d,k)=1$. Therefore, by definition of $\tilde \Delta$, we need to compute $\displaystyle \lim_{\genfrac{}{}{0pt}{}{\leftarrow}{n}}\, {}_n \tilde \Delta\big(\pi_{n, k}^\ast ({}_k\onebf^{\mathbf{ss}}_{1, \, d})\big)$. By Proposition~\ref{prop:coproduct}, we have
\begin{align}
{}_n \tilde \Delta\big(\pi_{n, k}^\ast ({}_k\onebf^{\mathbf{ss}}_{1, \, d})\big)={}_n \tilde \Delta\big({}_n\onebf^{\mathbf{ss}}_{1,\, d n/k}\big)={}_n\onebf^{\mathbf{ss}}_{1,\, d n/k}\otimes 1 + \sum_{e\in\Z_{\geq 0}}
T_n^{\, n\left\{d/k\right\}}\big({}_n\theta_{0,\, e}\big)\,\kbf_{(1,\, d n/k-e)}\otimes {}_n\onebf^{\mathbf{ss}}_{1,\, d n/k-e}\ .
\end{align}
By passing to the limit, the assertion follows.
\end{proof}
Denote by $\Hbf_\infty^{\bun}$ the Hall algebra associated with the exact subcategory consisting of locally free sheaves on $X_\infty$. Let us define $\omega_\infty\colon \Hbf_\infty^{\mathsf{tw}}\to \Hbf_\infty^{\bun}$ as
\begin{align}
\omega_\infty(x)\coloneqq \Omega_{\infty,n}\omega_n(x_n)\ ,
\end{align}
if $x=\Omega_{\infty,n}(x_n)$. The following version of Proposition~\ref{prop:omega} is straightforward to prove.
\begin{lemma}\label{lem:omega-infty}
Let $x\in\Q$. Then
\begin{align}
\omega_n\big(\sum_{e\in \Q_{\geq 0}}\, \theta_{0,\, e}\, {}_n\onebf^{\mathbf{ss}}_{1,\, x}\big)=\sum_{\alpha\in \Q_{\geq 0}}\, \xi_\alpha^{(x)}\, {}_n\onebf^{\mathbf{ss}}_{1,\, x+\alpha}\ , 
\end{align}
where
\begin{align}
\xi_\alpha^{(d)}\coloneqq
\begin{cases}
\displaystyle \xi_{\alpha} & \text{if } \{x\}=0, \, \{\alpha\}=0\ , \\[8pt]
\displaystyle \xi_{\lfloor\alpha\rfloor}-q^{-1}\,\xi_{\lfloor\alpha\rfloor}^\circ  & \text{if } \{x\}\neq 0, \, \{x+\alpha\}=0\ , \\[8pt]
\displaystyle \xi_{\lfloor\alpha\rfloor}^\circ  & \text{if } \{x\}\neq 0, \, \{\alpha\}=0\ ,\\[8pt]
0 & \text{otherwise}\ .
\end{cases}
\end{align}
\end{lemma}
Now we shall follow similar approach as in Section~\ref{sec:shuffle}. For $r\geq 1$, define
\begin{align}
\Hbf_\infty^{\mathsf{tw}}[r]\coloneqq\bigoplus_{\genfrac{}{}{0pt}{}{\alpha\in \Ksfnum(X_\infty)}{\rk(\alpha)=r}}\, \Hbf_\infty^{\mathsf{tw}}[\alpha] \quad\text{and}\quad \Ubf_\infty^>[r]\coloneqq\Ubf_\infty^>\cap \Hbf_\infty^{\mathsf{tw}}[r]\ .
\end{align}
As before, we define for $r\geq 1$
\begin{align}
J_r\colon \Ubf_\infty^>[r] \to \left( \Ubf_\infty^>[1]\right)^{\widehat{\otimes} \, r}\ , \quad u \mapsto \omega_\infty^{\otimes\, r}\, \Delta_{1, \ldots, 1}(u)\ ,
\end{align}
and $J\colon \Ubf_\infty^>\to \bigoplus_r \left( \Ubf_\infty^>[1]\right)^{\widehat{\otimes} \, r}$, given by the sum of all the maps $J_r$, is injective. 

Let $V_\infty$ be an infinite-dimensional vector space with basis $\{\vec v_x\,\vert\, x\in S^1_\Q\}$. Let us consider the following assignment: for any $a\in \Q$
\begin{align}\label{eq:assignment-infty}
\onebf_{1, a}^{\mathbf{ss}}\mapsto x^{\left\lfloor a\right\rfloor}\otimes \vec v_{\left\{a\right\}}\ .
\end{align}
Thus we have
\begin{align}
\left( \Ubf_\infty^>[1]\right)^{\otimes \, r}& \simeq \C[x_1^{\pm 1}, \ldots, x_r^{\pm 1}]\otimes V_\infty^{\, \otimes r}\ , \\[4pt]
\left( \Ubf_\infty^>[1]\right)^{\widehat{\otimes} \, r} & \simeq \C[x_1^{\pm 1}, \ldots, x_r^{\pm 1}][[x_1/x_2,\ldots, x_{r-1}/x_r]]\otimes V_\infty^{\, \otimes r}\ .
\end{align}
For any rational function $h(z)\in \C(z)$, define the map $\varpi^h$ which assigns to $x_1^{d_1}x_2^{d_2}\, \vec v_{1,\, a_1}\otimes \vec v_{2, \, a_2}$, the following:
\begin{align}
\begin{cases}
 h(x_1x_2^{-1})\, \upsilon^{-1}\, \frac{1-x_1 x_2^{-1}}{1-\upsilon^{-2}\, x_1 x_2^{-1}}\,x_1^{d_2}x_2^{d_1}\, \vec v_{1, \, a_2}\otimes \vec v_{2, \, a_1} 
 +h(x_1x_2^{-1})\, \frac{1-\upsilon^{-2}}{1-\upsilon^{-2}x_1x_2^{-1}} \, x_1^{d_2}x_2^{d_1}\, \vec v_{1, \, a_1}\otimes \vec v_{2, \, a_2} & \text{if $a_1>a_2$}\ , \\[15pt]
 h(x_1 x_2^{-1})\,x_1^{d_2}x_2^{d_1}\, \vec v_{1, \, a_1}\otimes \vec v_{2, \, a_2} & \text{if $a_1=a_2$} \ ,\\[15pt]
 h(x_1x_2^{-1})\, \upsilon^{-1}\, \frac{1-x_1 x_2^{-1}}{1-\upsilon^{-2}\, x_1 x_2^{-1}}\,x_1^{d_2}x_2^{d_1}\, \vec v_{1, \, a_2}\otimes \vec v_{2, \, a_1}   +h(x_1x_2^{-1})\, \frac{1-\upsilon^{-2}}{1-\upsilon^{-2}x_1x_2^{-1}} \, x_1^{d_2+1}x_2^{d_1-1}\, \vec v_{1, \, a_1}\otimes \vec v_{2, \, a_2}  & \text{if $a_1<a_2$}\ .
\end{cases}
\end{align}
For any permutation $\sigma$, we define $\varpi_\sigma^h$ as in Formula \eqref{eq:varpi}.
\begin{proposition}
Set
\begin{align}
h_X(z)\coloneqq\frac{\upsilon^{2(1-g_X)}\zeta_X(z)}{\zeta_X(\upsilon^{-2} z)}\ .
\end{align}
Then 
\begin{itemize}[leftmargin=0.2cm]
\item the map $\varpi^{h_X}$ satisfies the braid relation. 
\item the vector space
\begin{align}
\mathbf{Sh}_{\infty, \, h_X(z)}^{\mathsf{vec}}\coloneqq\C\oplus \bigoplus_{r\geq 1}\, \C[x_1^{\pm 1}, \ldots, x_r^{\pm 1}][[x_1/x_2,\ldots, x_{r-1}/x_r]]\otimes V_\infty^{\otimes\, r}\ ,
\end{align}
with the multiplication given by 
\begin{multline}
P(x_1, \ldots, x_r)\, \bigotimes_{i=1}^r\, \vec v_{i, \ell_i}\star Q(x_1, \ldots, x_s)\, \bigotimes_{k=1}^s\, \vec v_{k, j_k} \coloneqq\\
\sum_{\sigma\in\mathrm{Sh}_{r,s}}\, \varpi^{h_X}_\sigma\big(P(x_1, \ldots, x_r)\, Q(x_{r+1}, \ldots, x_{r+s})\, \otimes\bigotimes_{i=1}^{r+s}\, \vec v_{i, c_i}\big)
\end{multline}
where $c_i=\ell_i$ for $i=1, \ldots, r$ and $c_{r+k}=j_k$ for $k=1, \ldots, s$, is an associative algebra. Moreover, $\mathbf{Sh}_{\infty, \, h(z)}^{\mathsf{vec}}$ is equipped with a coproduct $\Delta\colon \mathbf{Sh}_{\infty, h(z)}^{\mathsf{vec}}\to \mathbf{Sh}_{\infty, h(z)}^{\mathsf{vec}}\widehat{\otimes} \mathbf{Sh}_{\infty, h(z)}^{\mathsf{vec}}$ given by
\begin{multline}
\Delta_{u,w}\big(x_1^{d_1}\cdots x_r^{d_r}\, \vec v_{1, \ell_1}\otimes \cdots \vec v_{r, \ell_r}\big)\coloneqq\\
(x_1^{d_1}\cdots x_u^{d_u}\, \vec v_{1, \ell_1}\otimes \cdots \vec v_{u, \ell_u})\otimes (x_1^{d_{u+1}}\cdots x_w^{d_{r}}\, \vec v_{1, \ell_{u+1}}\otimes \cdots \vec v_{w, \ell_r})\quad\text{and}\quad \Delta\coloneqq\bigoplus_{r=u+w}\, \Delta_{u,w}\ .
\end{multline}
\item the constant term map $J\colon \Ubf_\infty^>\to \mathbf{Sh}_{\infty, h_X(z)}^{\mathsf{vec}}$ is an algebra morphism such that
\begin{align}
(J_u\otimes J_w)\circ (\omega_n\otimes\omega_n)\circ \Delta_{u,w}=\Delta_{u,w}\circ J_{u+w}\ .
\end{align}
$\Ubf_\infty^>$ is isomorphic to the subalgebra $\Sbf_{\infty, \, h_X(z)}^{\mathsf{vec}}\subset \mathbf{Sh}_{\infty, \, h_X(z)}^{\mathsf{vec},\mathsf{rat}}$ generated by the degree one component in $x$. Moreover, the algebra $\Omega_\infty(\Ubf^>)$ is isomorphic to the subalgebra generated by the degree one component in $x$ of the image of the morphism of algebras
\begin{align}
\Omega_\infty\colon  \mathbf{Sh}_{1, \, h(z)}^{\mathsf{vec}} \to \mathbf{Sh}_{\infty, \, h_X(z)}^{\mathsf{vec}} \ , \quad P(x_1, \ldots, x_r) \mapsto \, P(x_1, \ldots, x_r)\otimes \bigotimes_{i=1}^r\, \vec v_{i, 0}\ . 
\end{align}
\end{itemize}
\end{proposition} 
\begin{proof}
We follow arguments similar to those in the proof of Proposition~\ref{prop:shuffle-vec}. By using Proposition~\ref{prop:delta-infty} and Lemma~\ref{lem:omega-infty}, we can derive the following formula for the ``standard factor" in $J_2$:
\begin{multline}
\sum_{\genfrac{}{}{0pt}{}{e\in\Z_{\geq 0}}{m\vert k}}\, \omega_\infty^{\otimes\, 2 }\big(T_\infty^{\{d/m\}}\big(\theta_{0,\, e/k}\big)\,\kbf_{(1,\, d/m-e/k)}\, \onebf^{\mathbf{ss}}_{1,\, \ell/n}\otimes \onebf^{\mathbf{ss}}_{1,\, d/m-e/k}\big)=\\
\shoveright{=\delta_{\left\{\ell/n-d/m\right\},0}\, \upsilon^{2-2g_X}\, \sum_{s\in\Z_{\geq 0}}\, \xi_s\, \onebf^{\mathbf{ss}}_{1,\, \ell/n+s}\otimes \onebf^{\mathbf{ss}}_{1,\, d/m-s}+(1-\delta_{\left\{\ell/n-d/m\right\},0})\upsilon^{2-2g_X}}\\
\times\,  \sum_{s\in\Z_{\geq0}}\Big(\upsilon^{-1}\, \xi_s^\circ \, \onebf^{\mathbf{ss}}_{1,\, \ell/n+s}\otimes \onebf^{\mathbf{ss}}_{1,\, d/m-s}
+(\xi_s-q^{-1}\, \xi_s^\circ)\,\onebf^{\mathbf{ss}}_{1,\, \ell/+s+\{d/m-\ell/n\}}\otimes \onebf^{\mathbf{ss}}_{1,\, d/m-s-\{d/m-\ell/n\}}\Big) \ .
\end{multline}
Let us introduce the automorphism $\gamma^{\pm m}$ of $\Ubf_n^>$ defined as
\begin{align}
\gamma^{\pm m}\big(\onebf_{1, \, x}^{\mathbf{ss}}\big)\coloneqq \onebf_{1, \, x\pm m}^{\mathbf{ss}}
\end{align}
for $m\in \Q$. Denote by $\gamma_i^\bullet$ the operator $\gamma^\bullet$ acting on the $i$-th component of the tensor product. By using Corollaries~\ref{cor:xie} and \ref{cor:xiecirc} the previous quantity is equal to
\begin{multline}
\delta_{\left\{\ell/n-d/m\right\},0}\, \upsilon^{2-2g_X}\, \sum_{s\in\Z_{\geq 0}}\, \xi_s\, (\gamma_1\gamma_2^{-1})^s\, \onebf^{\mathbf{ss}}_{1,\, \ell/n}\otimes \onebf^{\mathbf{ss}}_{1,\, d/m}+(1-\delta_{\left\{\ell/n-d/m\right\},0})\\
\times\, \upsilon^{2-2g_X}\, \sum_{s\in\Z_{\geq0}}\Big(\upsilon^{-1}\, \xi_s^\circ \, (\gamma_1\gamma_2^{-1})^s +\, (\xi_s-q^{-1} \, \xi_s^\circ)\,(\gamma_1\gamma_2^{-1})^s\, (\gamma_1\gamma_2^{-1})^{\{d/m-\ell/n\}}\Big) \onebf^{\mathbf{ss}}_{1,\, \ell/n}\otimes \onebf^{\mathbf{ss}}_{1,\, d/m} \ .
\end{multline}
The assertion follows from the assignment \eqref{eq:assignment-infty}.
\end{proof}

\subsection{Hecke operators on line bundles and the fundamental representation of $\Ubf_\upsilon(\slfrak(S^1_\Q))$}\label{sec:heckeoperatorslinebundles}

We define the  \textit{fundamental representation} of $\Ubf_\upsilon(\slfrak(S^1_\Q))$ as the $\widetilde{\Q}$-vector space
\begin{align}
\V_{S^1_\Q}=\bigoplus_{y \in \Q}\, \widetilde \Q\, \vec u_y \ ,
\end{align} 
with action given by
\begin{align}
F_{[a,b[}\bullet \vec u_y =& \delta_{\{b+y\}, 0}\, \upsilon^{1/2}\,\vec u_{y+b-a}\ , \\[4pt]
E_{[a,b[}\bullet \vec u_y =& \delta_{\{a+y\}, 0}\, \upsilon^{-1/2}\,\vec u_{y+a-b}\ , \\[4pt]
K_{[a',b'[}^\pm\bullet \vec u_y =& \upsilon^{\pm(\delta_{\{b'+y\},0}-\delta_{\{a'+y\},0})}\, \vec u_y\ ,
\end{align}
for rational intervals $[a,b[, [a',b'[\subseteq S^1_\Q$, with $[a,b[$ strict, and $y\in \Q$. 

The action of the Hall algebra $\Hbf_\infty^{\mathsf{tw}, \tor}$ on $\Hbf_\infty^{\bun}$ is given as
\begin{align}
\Hbf_\infty^{\mathsf{tw}, \tor}\otimes \Hbf_\infty^{\bun}\to \Hbf_\infty^{\bun}\ , \quad u_0\otimes u\mapsto u_0\bullet u\coloneqq\omega_\infty(u_0\, u)\ .
\end{align}
Thus we have the following.
\begin{theorem} 
The natural action of $\Ubf_\upsilon(\slfrak(S^1_\Q))$ on $\Ubf^>_\infty[1]$ is isomorphic to $\V_{S^1_\Q}$ after applying the automorphism
\begin{align}\label{eq:automorphism}
\Phi(E_J)=F_J\ , \quad \Phi(F_J)=E_J\ , \quad \Phi(K_{J'})=K_{J'}\ , \quad \Phi(\upsilon)=\upsilon^{-1}\ .
\end{align}
\end{theorem}
\begin{proof}
By using Corollary~\ref{cor:heckeaction-II}, one can easily derive the action of $F_{[a,b[}$ and $E_{[a,b[}$. By Formula \eqref{eq:actionK}, we have
\begin{align}
K_{\chi_{[a',b'[}}^\pm\bullet \onebf^{\mathbf{ss}}_{1,\, y}=\upsilon^{\pm(\delta_{\{a'+y\},0}-\delta_{\{b'+y\},0})}\, \onebf^{\mathbf{ss}}_{1,\, y}\ ,
\end{align}
for a rational interval $[a',b'[\subseteq S^1_\Q$ and $y\in \Q$. 
\end{proof}

\subsection{Tensor and symmetric tensor representations of $\Ubf_\upsilon(\slfrak(S^1_\Q))$}

Also in this case, we can extend the $\Ubf_\upsilon(\slfrak(S^1_\Q))$-action to $\Ubf^>_\infty[1]^{\otimes\, r}$, which we call \textit{$r$-th tensor representation} of $\Ubf_\upsilon(\slfrak(S^1_\Q))$, and to $\Ubf^>_\infty[1]^{\widehat \otimes\, r}$. We have the following.
\begin{proposition}
Let $r\in \Z_{>0}$. Then the symmetrization operator $\Psi_{\infty, r}$ given by the multiplication
\begin{align}
\Psi_{\infty, r}\colon \Ubf^>_\infty[1]^{\otimes\, r}\to \Ubf^>_\infty[1]^{\widehat \otimes\, r}
\end{align}
is a $\Ubf_\upsilon(\slfrak(S^1_\Q))$-intertwiner. 
\end{proposition}
We call the image of $\Psi_{\infty, r}$, which is $\Ubf^>_\infty[r]$, the \textit{symmetric tensor representation} of $\Ubf_\upsilon(\slfrak(S^1_\Q))$ (of genus $g_X$).

\subsection{Representation of the completed Hecke algebra}

Let us now consider the subalgebra $\Lambda\otimes_{\widetilde \Q} \Cbf_\infty$ of the completed Hecke algebra $\underline \Ubf_\infty^0$, where $\Lambda\coloneqq\widetilde \Q[Z_1, Z_2, \ldots, Z_r, \ldots]$ is the center of $\underline \Ubf_\infty^0$.

First note that, by Proposition~\ref{prop:representative-infty} and Corollaries~\ref{cor:completion-I} and \ref{cor:completion-II}, there does not exist a completion of $\Hbf_\infty^{\bun}$ similar to $\Hbf_{\infty, p_\infty}^{\mathsf{tw}, \tor}$.

We define the  \textit{fundamental representation} of $\Lambda\otimes_{\widetilde \Q} \Cbf_\infty$ and $\underline \Ubf_\infty^0$ as the $\widetilde{\Q}$-vector space
\begin{align}
\V=\bigoplus_{y \in \Q}\, \widetilde \Q\, \vec u_y \ ,
\end{align} 
with action given by
\begin{align}
1_{\Scal_{[a,b[}}\bullet \vec u_y =& \delta_{\{b+y\}, 0}\, \upsilon^{-1}\,\vec u_{y+b-a}\ , \\[4pt]
K_{\chi_{[a,b[}}^\pm\bullet \vec u_y =& \upsilon^{\pm(\delta_{\{a+y\},0}-\delta_{\{b+y\},0})}\, \vec u_y\ ,\\[4pt]
Z_r \bullet  \vec u_y =& ({}_n Z_r, {}_n \theta_{0, rn})_G\,  \vec u_{y+r} \ , 
\end{align}
for any rational interval $[a,b[\subset S^1_\Q$ and $y=d/n\in \Q$, with $\mathsf{g.c.d.}(d,n)=1$. 

Thanks to Corollary~\ref{cor:completion-II}, the action of the Hall algebra $\Ubf_\infty^0$ on $\Hbf_\infty^{\bun}$ extends to $\underline \Ubf_\infty^0$. We have:
\begin{theorem} 
The natural action of $\Lambda\otimes_{\widetilde \Q} \Cbf_\infty$ and $\underline \Ubf_\infty^0$ on $\Ubf^>_\infty[1]$ is isomorphic to $\V$. 
\end{theorem}
%

\bigskip\section{Comparisons with $\Ubf_\upsilon\big(\slfrakhat(+\infty))$ and $\Ubf_\upsilon\big(\slfrakhat(\infty))$}\label{sec:comparisons}

Let
\begin{align}
  \begin{tikzpicture}[xscale=1.5,yscale=-0.5]
    \node (A0_0) at (0, 0) {$1$};
    \node (A0_1) at (1, 0) {$2$};
    \node (A0_2) at (2, 0) {$3$};
    \node (A0_3) at (3, 0) {$4$};
    \node (A0_5) at (5, 0) {$n-1$};
    \node (A0_6) at (6, 0) {$n$};
    \node (A1_0) at (0, 1) {$\bullet$};
    \node (A1_1) at (1, 1) {$\bullet$};
    \node (A1_2) at (2, 1) {$\bullet$};
    \node (A1_3) at (3, 1) {$\bullet$};
    \node (A1_5) at (5, 1) {$\bullet$};
    \node (A1_6) at (6, 1) {$\bullet$};
    \path (A1_0) edge [-]node [auto] {$\scriptstyle{}$} (A1_1);
    \path (A1_1) edge [-]node [auto] {$\scriptstyle{}$} (A1_2);
    \path (A1_2) edge [-]node [auto] {$\scriptstyle{}$} (A1_3);
    \path (A1_3) edge [dashed, -]node [auto] {$\scriptstyle{}$} (A1_5);
    \path (A1_5) edge [-]node [auto] {$\scriptstyle{}$} (A1_6);
  \end{tikzpicture} 
\end{align}
be the Dynkin diagram of type $A_n$. We can consider two different limits for $n\to +\infty$: the infinite Dynkin diagram
\begin{align}
  \begin{tikzpicture}[xscale=1.5,yscale=-0.5]
    \node (A1_0) at (0, 1) {$\bullet$};
    \node (A1_1) at (1, 1) {$\bullet$};
    \node (A1_2) at (2, 1) {$\bullet$};
    \node (A1_3) at (3, 1) {$\bullet$};
    \node (A1_5) at (5, 1) {$ $};
    \path (A1_0) edge [-]node [auto] {$\scriptstyle{}$} (A1_1);
    \path (A1_1) edge [-]node [auto] {$\scriptstyle{}$} (A1_2);
    \path (A1_2) edge [-]node [auto] {$\scriptstyle{}$} (A1_3);
    \path (A1_3) edge [dashed, -]node [auto] {$\scriptstyle{}$} (A1_5);
  \end{tikzpicture} 
\end{align}
which gives rise to the infinite dimensional Lie algebra $\slfrak(+\infty)$, and the infinite Dynkin diagram
\begin{align}\label{eq:limit-II}
  \begin{tikzpicture}[xscale=1.5,yscale=-0.5]
    \node (A1_0') at (-2, 1) {$ $};
  \node (A1_0) at (0, 1) {$\bullet$};
    \node (A1_1) at (1, 1) {$\bullet$};
    \node (A1_2) at (2, 1) {$\bullet$};
    \node (A1_3) at (3, 1) {$\bullet$};
    \node (A1_5) at (4, 1) {$\bullet$};
    \node (A1_6) at (6, 1) {$ $};
    \path (A1_0) edge [-]node [auto] {$\scriptstyle{}$} (A1_1);
    \path (A1_1) edge [-]node [auto] {$\scriptstyle{}$} (A1_2);
    \path (A1_2) edge [-]node [auto] {$\scriptstyle{}$} (A1_3);
    \path (A1_3) edge [-]node [auto] {$\scriptstyle{}$} (A1_5);
    \path (A1_5) edge [dashed, -]node [auto] {$\scriptstyle{}$} (A1_6);
   \path (A1_0') edge [dashed, -]node [auto] {$\scriptstyle{}$} (A1_0); 
  \end{tikzpicture} 
\end{align}
which corresponds to the infinite dimensional Lie algebra $\slfrak(\infty)$. One can define the quantum groups $\Ucal_\upsilon(\slfrak(+\infty))$ and $\Ucal_\upsilon(\slfrak(\infty))$ using the Drinfeld-Jimbo presentation (see for example \cite[Section~2]{art:arikijaconlecouvey2012}). In the present section, we relate the quantum affinization of such quantum groups to our algebras $\Ubf_\upsilon(\slfrak(S^1_\Q))=\DCbf_\infty$ and $\DUbf_\infty^0$. 

Let us start by realizing geometrically $\Ucal_\upsilon(\slfrakhat(+\infty))$ as a direct limit of associative algebras. Let $n$ be a positive integer and consider the Serre subcategory $\mathsf{add}\big({}_{(n+1)}\Scal_{n+1}^{(1)}\big)$ of $\Tor_{p_{n+1}}(X_{n+1})$ generated by the simple torsion sheaf ${}_{(n+1)}\Scal_{n+1}^{(1)}$. As described in \cite[Section~4.10]{art:burbanschiffmann2013}, the Serre quotient $\Tor_{p_{n+1}}(X_{n+1})/\mathsf{add}\big({}_{(n+1)}\Scal_{n+1}^{(1)}\big)$ is equivalent to $\Tor_{p_n}(X_n)$: there exists an exact fully faithful functor 
\begin{align}
\Tor_{p_n}(X_n)\to \Tor_{p_{n+1}}(X_{n+1})
\end{align}
such that
\begin{align}
{}_n \Scal_n\mapsto {}_{(n+1)}\Scal_1^{(2)} \quad\text{and}\quad {}_n\Scal_i\mapsto {}_{(n+1)}\Scal_{i+1} \text{ for } i=1, \ldots, n-1 \ .
\end{align}
Let us denote by $\imath_{n, n+1}\colon \Tor_{p_n}(X_n)\to \Tor_{p_{n+1}}(X_{n+1})$ the composition of such a functor with the equivalence given by tensoring by $\Lcal_{n+1}$. We have
\begin{align}\label{eq:imath}
\imath_{n, n+1}\big({}_n \Scal_n\big)= {}_{(n+1)}\Scal_{(n+1)}^{(2)} \quad\text{and}\quad \imath_{n, n+1}\big({}_n\Scal_i\big)= {}_{(n+1)}\Scal_i \text{ for } i=1, \ldots, n-1 \ .
\end{align}
By \cite[Theorem~3.3]{art:burbanschiffmann2013}, we have an injective algebra homomorphism of the reduced Drinfeld doubles $\imath_{n, n+1}\colon \DHbf^{\mathsf{tw}, \tor}_{n, p_n}\to \DHbf^{\mathsf{tw}, \tor}_{n+1, p_{n+1}}$ and therefore an injective algebra homomorphism
\begin{align}
\imath_{n, n+1}\colon \DCbf_n\to \DCbf_{n+1}\ .
\end{align}
Here and in the following, by abuse of notation we denote in the same way the functor \eqref{eq:imath} and the homomorphisms of algebras induced by it.

We define $\Ubf_\upsilon(\slfrakhat(+\infty))$ as the direct limit
\begin{align}
\Ubf_\upsilon\big(\slfrakhat(+\infty))\coloneqq\lim_{\genfrac{}{}{0pt}{}{\to}{n}}\, \DCbf_n
\end{align}
with respect to the directed system of associative algebras $(\DCbf_n, \imath_{n, n+1})$. Now, let us relate it to $\Ubf_\upsilon(\slfrak(S^1_\Q))$. Define a discrete family $\{\alpha_n\}_{n\in \Z}$ of rational numbers $\alpha_n$ by
\begin{align}
\alpha_n\coloneqq
\begin{cases}
\displaystyle 1-\frac{1}{2^{n+1}} & \text{for } n\geq 0 \ , \\[8pt]
\displaystyle \frac{1}{2^{-n+1}} & \text{for } n\leq 0 \ .
\end{cases}
\end{align}
Define the injective homomorphism $\Phi_{n, +\infty}\colon \Cbf_n\to \Hbf_{\infty, p_\infty}^{\mathsf{tw}, \tor}$ of algebras sending
\begin{align}
1_{{}_n \Scal_n} &\mapsto 1_{\Scal_{[\alpha_{n-1},\, 1/2[}}\quad\text{and}\quad
1_{{}_n \Scal_i} \mapsto 1_{\Scal_{[\alpha_{i-1},\, \alpha_i[}}\ \text{ for } i=1, \ldots, n-1\ ,\\[2pt] 
\kbf_{(0, \pm\mathbf e_{n, n})} & \mapsto K_{[\alpha_{n-1}, 1/2[}^\mp   \quad\text{and}\quad \kbf_{(0, \pm\mathbf e_{n, i})}\mapsto K_{[\alpha_{i-1}, \alpha_i[}^\mp\ \text{ for } i=1, \ldots, n-1\ .
\end{align}
The homomorphism $\Phi_{n, +\infty}$ is compatible with the homomorphism $\imath_{n, n+1}\colon \Cbf_n\to \Cbf_{n+1}$, i.e., we have a commutative diagram of associative algebras
\begin{align}
  \begin{tikzpicture}[xscale=1.5,yscale=-1.2]
    \node (A0_0) at (0, 0) {$\Cbf_n$};
    \node (A0_2) at (2, 0) {$\Hbf_{\infty, p_\infty}^{\mathsf{tw}, \tor}$};
    \node (A1_1) at (1, 1) {$\square$};
    \node (A2_0) at (0, 2) {$\Cbf_{n+1}$};
    \node (A2_2) at (2, 2) {$\Hbf_{\infty, p_\infty}^{\mathsf{tw}, \tor}$};
    \node (Comma) at (2.7, 1) {$.$};
    \path (A0_0) edge [->]node [left] {$\scriptstyle{\imath_{n, n+1}}$} (A2_0);
    \path (A0_0) edge [->]node [auto] {$\scriptstyle{\Phi_{n, +\infty}}$} (A0_2);
    \path (A0_2) edge [->]node [auto] {$\scriptstyle{\mathsf{id}}$} (A2_2);
    \path (A2_0) edge [->]node [auto] {$\scriptstyle{\Phi_{n+1, +\infty}}$} (A2_2);
  \end{tikzpicture} 
\end{align}
Such a compatibility holds at the level of the reduced Drinfeld doubles. In particular, we obtain the following.
\begin{proposition}
The homomorphisms $\Phi_{n, +\infty}$ induce an injective homomorphism of associative algebras
\begin{align}
\Phi_{+\infty}\colon \Ubf_\upsilon\big(\slfrakhat(+\infty)) \to \Ubf_\upsilon(\slfrak(S^1_\Q)) \ .
\end{align}
\end{proposition}

If we see $\mathsf{add}\big({}_{(n+1)}\Scal_{n+1}^{(1)}\big)$ as a Serre subcategory of $\Tor(X_{n+1})$ and consider the corresponding Serre quotient, we obtain an exact fully faithful functor $\imath_{n, n+1}\colon \Tor(X_n)\to \Tor(X_{n+1})$ and therefore an injective homomorphism of algebras (cf.\ \cite[Theorem~5.25]{art:burbanschiffmann2013})
\begin{align}
\imath_{n, n+1}\colon \DUbf_n^0\to \DUbf_{n+1}^0 \ \ \text{, i.e.,}\ \ \imath_{n, n+1}\colon \Ubf_\upsilon(\glfrakhat(n))\to \Ubf_\upsilon(\glfrakhat(n+1))\ .
\end{align}
We call $\Ubf_\upsilon(\glfrakhat(+\infty))$ the direct limit
\begin{align}
\Ubf_\upsilon\big(\glfrakhat(+\infty))\coloneqq\lim_{\genfrac{}{}{0pt}{}{\to}{n}}\, \DUbf_n^0
\end{align}
with respect to the directed system of associative algebras $(\DUbf_n^0, \imath_{n, n+1})$. By using similar arguments than before, we obtain injective homomorphisms $\Phi_{n, +\infty}\colon \DUbf_n^0\to \DUbf_\infty^0$ (note that the generators $\Omega_n(T_r)$ are sent to $\Omega_\infty(T_r)$).
\begin{proposition}
The homomorphisms $\Phi_{n, +\infty}$ induce an injective homomorphism of associative algebras
\begin{align}
\Phi_{+\infty}\colon \Ubf_\upsilon\big(\glfrakhat(+\infty)) \to   \DUbf_\infty^0\ .
\end{align}
\end{proposition}

\begin{remark} The construction of the above embedding $\Ubf_\upsilon\big(\slfrakhat(+\infty)) \to \Ubf_\upsilon(\slfrak(S^1_\Q))$ depends on a choice of a discrete sequence $(\alpha_n)_n$ as above; we could have constructed a similar embedding using any discrete sequence with a single accumulation point. 
\end{remark}

Now we turn to the second type of limit of the $A_n$ Dynkin quiver \eqref{eq:limit-II}. By using again Serre quotients, we can define an exact fully faithful functor $\jmath_{n, n+2}\colon \Tor_{p_{n}}(X_{n})\to \Tor_{p_{n+2}}(X_{n+2})$ sending
\begin{align}
{}_{n} \Scal_{n}  \mapsto \imath_{n+1, n+2}\big({}_{n+1}\Scal_1^{(2)}\big)\quad \text{and}\quad {}_{n} \Scal_i \mapsto {}_{n+2} \Scal_{i+1} \ \text{ for } i=1, \ldots, n-1\ .
\end{align}
Then there is a corresponding injective homomorphism of algebras $\jmath_{n, n+2}\colon \DCbf_{n}\to \DCbf_{n+2}$ and a directed system $(\DCbf_{n}, \jmath_{n, n+2})$. We define $\Ubf_\upsilon(\slfrakhat(\infty))$ as the direct limit
\begin{align}
\Ubf_\upsilon\big(\slfrakhat(\infty))\coloneqq\lim_{\genfrac{}{}{0pt}{}{\to}{n}}\, \DCbf_{n}\ .
\end{align} 
As before, by working over the category of all torsion sheaves on $X_{n}$ and $X_{n+2}$, we can obtain a directed system $(\DUbf_{n}^0, \jmath_{n, n+2})$. We define $\Ubf_\upsilon(\glfrakhat(\infty))$ as the direct limit
\begin{align}
\Ubf_\upsilon\big(\glfrakhat(\infty))\coloneqq\lim_{\genfrac{}{}{0pt}{}{\to}{n}}\, \DUbf_{n}^0\ .
\end{align} 
\begin{rem}
In \cite{art:hernandez2011}, Hernandez defined the quantum group $\Ucal_\upsilon(\slfrakhat(\infty))$ as a quantum affinization of $\Ucal_\upsilon(\slfrak(\infty))$ by using a ``Drinfeld presentation" (see also \cite{art:mansuy2015}). Our approach to this quantum affinization is somewhat different and only involves Drinfeld-Jimbo type relations. \hfill $\triangle$
\end{rem}
Fixing the discrete family $\{\alpha_n\}_{n \in \Z}$ as before, let us consider the assignment 
\begin{align}
1_{{}_{n}\Scal_n} &\mapsto 1_{\Scal_{[\alpha_{\lfloor n/2\rfloor +2 \{n/2\}},\, \alpha_{-\lfloor n/2\rfloor+1}[}}\ \text{ and }\ 1_{{}_{n}\Scal_i}\mapsto 1_{\Scal_{[\alpha_{-\lfloor n/2\rfloor+i}, \,\alpha_{-\lfloor n/2\rfloor+i+1}[}}\ \text{ for } i=1, \ldots, n-1\ ,\\[2pt] 
\kbf_{(0, \pm\mathbf e_{n, n})} &\mapsto K_{[\alpha_{\lfloor n/2\rfloor +2 \{n/2\}},\, \alpha_{-\lfloor n/2\rfloor+1}[}^\mp   \quad\text{and}\quad \kbf_{(0, \pm\mathbf e_{n, i})}\mapsto K_{[\alpha_{-\lfloor n/2\rfloor+i}, \,\alpha_{-\lfloor n/2\rfloor+i+1}[}^\mp\ \text{ for } i=1, \ldots, n-1\ .
\end{align}
These define injective homomorphism of associative algebras $\Phi_{n, \infty}\colon \DCbf_{n}\to \DCbf_\infty$ and $\Phi_{n, \infty}\colon \DUbf_{n}^0\to \DUbf_\infty^0$ compatible with the $\jmath_{n, n+2}$. As before, we have the following. 
\begin{proposition}
The homomorphisms $\Phi_{n, \infty}$ induce, respectively, injective homomorphisms of associative algebras
\begin{align}
\Phi_{\infty}\colon \Ubf_\upsilon\big(\slfrakhat(\infty)) \to \Ubf_\upsilon(\slfrak(S^1_\Q))
\quad \text{and}\quad \Phi_{\infty}\colon \Ubf_\upsilon\big(\glfrakhat(\infty)) \to \DUbf_\infty^0 \ .
\end{align}
\end{proposition}
\appendix

\bigskip\section{Some results from \cite{art:lin2014}}

In this section we provide the proofs of some of the results in \cite{art:lin2014}, which are used in the main body of the paper. This is done in order to avoid possible confusions since our assumptions differ sometimes from \textit{loc.cit.} (for example, the orientation of the cyclic quiver we chose in the main body of the paper is the opposite of the one in \textit{loc.cit.}). We will follow the notation introduced in Sections~\ref{sec:preliminariesorbcurves-nrootstack} and \ref{sec:hallalgebran}.
\begin{lemma}\label{lem:Lin-1}
Let $1\leq i\leq n$ and $1<j<n$. Then
\begin{align}
1_{{}_n \Scal_i^{(j)}}=\upsilon\, 1_{{}_n \Scal_{\nfrapart{i+1-j}}^{(1)}}\, 1_{{}_n \Scal_{i}^{(j-1)}} -1_{{}_n \Scal_{i}^{(j-1)}}\,  1_{{}_n \Scal_{\nfrapart{i+1-j}}^{(1)}}\ .
\end{align}
Here, we set formally ${}_n\Scal_0^{(\ell)}\coloneqq{}_n \Scal_n^{(\ell)}$ for any positive integer $\ell$.
\end{lemma}
\begin{proof}
First note, that because of our assumption on $j$, we have
\begin{align}
\langle [{}_n \Scal_i^{(j-1)}], [{}_n \Scal_{\nfrapart{i+1-j}}^{(1)}]\rangle &= 0\\
\langle [{}_n \Scal_{\nfrapart{i+1-j}}^{(1)}], [{}_n \Scal_i^{(j-1)}]\rangle& =
\begin{cases}
-1 & \text{if $\nfrapart{i+1-j}=0$}\ ,\\[4pt]
0 & \text{otherwise}\ .
\end{cases}
\end{align}
By applying the functor $\Hom(\cdot, {}_n \Scal_{\nfrapart{i+1-j}}^{(1)})$ to the short exact sequence
\begin{align}\label{eq:ses}
0\to \Lcal_n^{\otimes\, -i}\to \Lcal_n^{\otimes\, j-1-i}\to {}_n \Scal_i^{(j-1)}\to 0\ ,
\end{align}
we obtain
\begin{align}
0\to \Hom({}_n \Scal_i^{(j-1)}, {}_n \Scal_{\nfrapart{i+1-j}}^{(1)})\to \Hom(\Lcal_n^{\otimes\, j-1-i}, {}_n \Scal_{\nfrapart{i+1-j}}^{(1)})\to  \cdots
\end{align}
Since $\Hom(\Lcal_n^{\otimes\, j-1-i}, {}_n \Scal_{\nfrapart{i+1-j}}^{(1)})\simeq H^0(X_n, {}_n\Scal_n^{(1)})=0$, we get 
\begin{align}
 \Hom({}_n \Scal_i^{(j-1)}, {}_n \Scal_{\nfrapart{i+1-j}}^{(1)})=0 \ , \  \Ext^1({}_n \Scal_i^{(j-1)}, {}_n \Scal_{\nfrapart{i+1-j}}^{(1)})=0\ .
\end{align}
Now, by applying the functor $\Hom({}_n \Scal_{\nfrapart{i+1-j}}^{(1)}, \cdot)$ to the short exact sequence \eqref{eq:ses}, we obtain
\begin{align}
\cdots \to \Ext^1({}_n \Scal_{\nfrapart{i+1-j}}^{(1)}, \Lcal_n^{\otimes\, j-1-i})\to \Ext^1({}_n \Scal_{\nfrapart{i+1-j}}^{(1)},  {}_n \Scal_i^{(j-1)})\to 0 \ .
\end{align}
In this case, by Serre duality $\Ext^1({}_n \Scal_{\nfrapart{i+1-j}}^{(1)}, \Lcal_n^{\otimes\, j-1-i})\simeq H^0(X_n, {}_n \Scal_1^{(1)})^\vee \simeq k$, we get
\begin{align}
\Hom({}_n \Scal_{\nfrapart{i+1-j}}^{(1)}, {}_n \Scal_i^{(j-1)})=0 \ , \ \Ext^1({}_n \Scal_{\nfrapart{i+1-j}}^{(1)},  {}_n \Scal_i^{(j-1)})\simeq k\ .
\end{align}
Thus we get
\begin{align}
1_{{}_n \Scal_{\nfrapart{i+1-j}}^{(1)}}\, 1_{{}_n \Scal_{i}^{(j-1)}} &=\upsilon^{-1}\Big(1_{{}_n \Scal_i^{(j)}}+1_{{}_n \Scal_{\nfrapart{i+1-j}}^{(1)}\oplus {}_n \Scal_{i}^{(j-1)}}\Big)\ ,\\
1_{{}_n \Scal_{i}^{(j-1)}}\,  1_{{}_n \Scal_{\nfrapart{i+1-j}}^{(1)}} &=1_{{}_n \Scal_{\nfrapart{i+1-j}}^{(1)}\oplus {}_n \Scal_{i}^{(j-1)}} \ .
\end{align}
\end{proof}
By using similar arguments as in the proof above, one can also prove the following lemma.
\begin{lemma}\label{lem:BS}
Let $0\leq j\leq n-1$ and $m\in \Z_{>0}$. Then
\begin{align}
1_{{}_n \Scal_j^{(m n+j)}}=
\begin{cases}
1_{{}_n \Scal_n^{(m n)}} & \text{if $j=0$}\ , \\[4pt]
1_{{}_n \Scal_j^{(j)}}\, 1_{{}_n \Scal_j^{(m n)}}-\upsilon^{-2}\,1_{{}_n \Scal_j^{(m n)}}\, 1_{{}_n \Scal_j^{(j)}} & \text{otherwise} \ .
\end{cases}
\end{align}
\end{lemma}
\begin{lemma}\label{lem:Lin-2}
Let $1\leq i\leq n$ and $d\in \Z$. Then
\begin{align}
1_{{}_n \Scal_i}\, {}_n\onebf^{\mathbf{ss}}_{1,\, d}=\upsilon^{\delta_{n\{(i+d-1)/n\},0}-\delta_{n\{(i+d)/n\},0}}\, {}_n\onebf^{\mathbf{ss}}_{1,\, d}\, 1_{{}_n \Scal_{i}}+\delta_{n\{(i+d)/n\},0}\, \upsilon^{-1}\, {}_n\onebf^{\mathbf{ss}}_{1,\, d+1}\ .
\end{align}
\end{lemma}
\begin{proof}
Let $M$ be a line bundle on $X$ and $1\leq k\leq n$ an integer. First compute $1_{{}_n \Scal_k}\, 1_{\pi_n^\ast M}$. Note that
\begin{align}
\langle \overline{{}_n \Scal_k}, \overline{\pi_n^\ast M}\rangle &=\langle \overline{{}_n \Scal_k},  \overline \Ocal_{X_n}+\deg(M)\, \delta_n\rangle =-\delta_{k,n}\ ,\\[2pt]
\langle \overline{\pi_n^\ast M}, \overline{{}_n \Scal_k}\rangle &=\langle\overline \Ocal_{X_n}+\deg(M)\, \delta_n, \overline{{}_n \Scal_k} \rangle =\delta_{k,1}\ .
\end{align}
In addition, 
\begin{align}
\Hom({}_n \Scal_k, \pi_n^\ast M)& \simeq \Ext^1(\pi_n^\ast M, {}_n \Scal_k\otimes\omega_{X_n})^\vee =0\\[2pt]
\Ext^1(\pi_n^\ast M, {}_n \Scal_k)& =0\ .
\end{align}
Hence,
\begin{align}
1_{\pi_n^\ast M}\, 1_{{}_n \Scal_k} &=\upsilon^{\delta_{k,1}}\, 1_{{}_n \Scal_k\oplus \pi_n^\ast M} \ ,\\[4pt]
1_{{}_n \Scal_k}\, 1_{\pi_n^\ast M} &=
\begin{cases}
\upsilon^{-1}\big(1_{{}_n \Scal_k\oplus \pi_n^\ast M}+1_{\pi_n^\ast M\otimes \Lcal_n}\big) & \text{ for $k=n$} \ ,\\[4pt]
\upsilon^{\, 2\delta_{k,1}}\,  1_{{}_n \Scal_k\oplus \pi_n^\ast M} & \text{otherwise}\ .
\end{cases}
\end{align}
Thus
\begin{align}
1_{{}_n \Scal_k}\, 1_{\pi_n^\ast M} &=\upsilon^{\delta_{k,1}-\delta_{k,n}}\, 1_{\pi_n^\ast M}\, 1_{{}_n \Scal_k}+\delta_{k,n}\, \upsilon^{-1}\, 1_{\pi_n^\ast M\otimes \Lcal_n} \\[4pt]
&=\upsilon^{\delta_{n\{k/n\},1}-\delta_{n\{k/n\},0}}\, 1_{\pi_n^\ast M}\, 1_{{}_n \Scal_k}+\delta_{n\{k/n\},0}\, \upsilon^{-1}\, 1_{\pi_n^\ast M\otimes \Lcal_n}\ .
\end{align}
Now,
\begin{multline}
1_{{}_n \Scal_i}\, {}_n\onebf^{\mathbf{ss}}_{1,\, d} = 1_{{}_n \Scal_i}\, \sum_{\genfrac{}{}{0pt}{}{M\in \Pic(X)}{\deg(M)=\nintpart{d}}}\, 1_{\pi_n^\ast M\otimes \Lcal_n^{\otimes\, \nfrapart{d}}}=\sum_{\genfrac{}{}{0pt}{}{M\in \Pic(X)}{\deg(M)=\nintpart{d}}}\,1_{{}_n \Scal_i}\,1_{\pi_n^\ast M\otimes \Lcal_n^{\otimes\, \nfrapart{d}}}\\
\shoveleft{=\sum_{\genfrac{}{}{0pt}{}{M\in \Pic(X)}{\deg(M)=\nintpart{d}}}\, T_n^{\nfrapart{d}}\big(1_{{}_n \Scal_{i+\nfrapart{d}}}\,1_{\pi_n^\ast M}\big)}\\
=\sum_{\genfrac{}{}{0pt}{}{M\in \Pic(X)}{\deg(M)=\nintpart{d}}}\, T_n^{\nfrapart{d}}\big(\upsilon^{\delta_{\nfrapart{i+d},1}-\delta_{\nfrapart{i+d},0}}\, 1_{\pi_n^\ast M}\, 1_{{}_n \Scal_{i+\nfrapart{d}}}+\delta_{\nfrapart{i+d},0}\, \upsilon^{-1}\, 1_{\pi_n^\ast M\otimes \Lcal_n}\big)\ ,
\end{multline}
and we get the assertion.
\end{proof}

Thanks to the previous lemmas, we can provide a characterization of the commutation relations between the $1_{{}_n \Scal_i^{(j)}}$'s and the ${}_n\onebf^{\mathbf{ss}}_{1,\, d}$'s.
\begin{proposition}
Let $1\leq i \leq n$ and $1\leq j\leq n-1$ be integers, let $d\in \Z$. Then
\begin{multline}
1_{{}_n \Scal_i^{(j)}}\, {}_n\onebf^{\mathbf{ss}}_{1,\, d}=\\[3pt]
\begin{cases}
\upsilon^{-1}\, {}_n\onebf^{\mathbf{ss}}_{1,\, d}\, 1_{{}_n \Scal_i^{(j)}}+\upsilon^{-1}\,  {}_n\onebf^{\mathbf{ss}}_{1,\, d+j} & \text{if } \nfrapart{i+d}=0\ , \\[4pt]
{}_n\onebf^{\mathbf{ss}}_{1,\, d}\, 1_{{}_n \Scal_i^{(j)}} & \text{if } \nfrapart{i+d}>j \text{ and } \nfrapart{i+d}\neq 0 \ ,\\[4pt]
\upsilon\, {}_n\onebf^{\mathbf{ss}}_{1,\, d}\, 1_{{}_n \Scal_i^{(j)}} & \text{if } \nfrapart{i+d}=j \text{ and } \nfrapart{i+d}\neq 0\ , \\[4pt]
{}_n \onebf^{\mathbf{ss}}_{1,\, d}\, 1_{{}_n \Scal_i^{(j)}}+(\upsilon-\upsilon^{-1})\, {}_n \onebf^{\mathbf{ss}}_{1,\, \nintpart{d} n+j-\nfrapart{i+d}}\, 1_{{}_n \Scal_{\nfrapart{i+d}}^{(\nfrapart{i+d})}} & \text{if } \nfrapart{i+d}<j \text{ and } \nfrapart{i+d}\neq 0\ . 
\end{cases}
\end{multline}
\end{proposition}
\begin{proof}
First, let us assume that $d=\alpha n$ for some $\alpha\in \Z$ and let $k\in \{1, \ldots, n\}$. Then the statement reduces to
\begin{align}
1_{{}_n \Scal_k^{(j)}}\, {}_n\onebf^{\mathbf{ss}}_{1,\, \alpha n}=
\begin{cases}
\upsilon^{-1}\, {}_n\onebf^{\mathbf{ss}}_{1,\, \alpha n}\, 1_{{}_n \Scal_k^{(j)}}+\upsilon^{-1}\,  {}_n\onebf^{\mathbf{ss}}_{1,\, \alpha n+j} & \text{if $k=n$}\ , \\[2pt]
{}_n\onebf^{\mathbf{ss}}_{1,\, \alpha n}\, 1_{{}_n \Scal_k^{(j)}} & \text{if $n-1\geq k>j$}\ ,\\[2pt]
\upsilon\, {}_n\onebf^{\mathbf{ss}}_{1,\, \alpha n}\, 1_{{}_n \Scal_k^{(j)}} & \text{if $k=j$}\ , \\[2pt]
{}_n \onebf^{\mathbf{ss}}_{1,\, \alpha n}\, 1_{{}_n \Scal_k^{(j)}}+(\upsilon-\upsilon^{-1})\, {}_n \onebf^{\mathbf{ss}}_{1,\, \alpha n+j-k}\, 1_{{}_n \Scal_k^{(k)}} &\text{if $k<j$}\ . 
\end{cases}
\end{align}
Assume that it is true for all $j'< j$ and let us prove it for $j$, by using Lemmas~\ref{lem:Lin-1} and \ref{lem:Lin-2}. For $n-1\geq k >j$, we have
\begin{align}
1_{{}_n \Scal_k^{(j)}}\, {}_n\onebf^{\mathbf{ss}}_{1,\, \alpha n} &=\Big(\upsilon\, 1_{{}_n \Scal_{\nfrapart{k+1-j}}^{(1)}}\, 1_{{}_n \Scal_{k}^{(j-1)}} -1_{{}_n \Scal_{k}^{(j-1)}}\,  1_{{}_n \Scal_{\nfrapart{k+1-j}}^{(1)}}\Big)\, {}_n\onebf^{\mathbf{ss}}_{1,\, \alpha n}\\
&=\upsilon\, 1_{{}_n \Scal_{\nfrapart{k+1-j}}^{(1)}}\, {}_n\onebf^{\mathbf{ss}}_{1,\, \alpha n}\, 1_{{}_n \Scal_{k}^{(j-1)}}-1_{{}_n \Scal_{k}^{(j-1)}}\Big({}_n\onebf^{\mathbf{ss}}_{1,\, \alpha n}\, 1_{{}_n \Scal_{\nfrapart{k+1-j}}^{(1)}}\Big)\\
&={}_n\onebf^{\mathbf{ss}}_{1,\, \alpha n}\, 1_{{}_n \Scal_k^{(j)}}\ .
\end{align}
For $k=n$ and $j>1$, we get
\begin{align}
1_{{}_n \Scal_k^{(j)}}\, {}_n\onebf^{\mathbf{ss}}_{1,\, \alpha n} &=\Big(\upsilon\, 1_{{}_n \Scal_{n-j+1}^{(1)}}\, 1_{{}_n \Scal_{n}^{(j-1)}} -1_{{}_n \Scal_{n}^{(j-1)}}\,  1_{{}_n \Scal_{n-j+1}^{(1)}}\Big)\, {}_n\onebf^{\mathbf{ss}}_{1,\, \alpha n}\\
&= \upsilon\, 1_{{}_n \Scal_{n-j+1}^{(1)}}\, \Big(\upsilon^{-1}\, \onebf^{\mathbf{ss}}_{1,\, \alpha n} \, 1_{{}_n \Scal_n^{(j-1)}}+\upsilon^{-1}\, \onebf^{\mathbf{ss}}_{1,\, \alpha n+j-1}\Big)-1_{{}_n \Scal_n^{(j-1)}}\, \Big({}_n \onebf^{\mathbf{ss}}_{1,\, \alpha n}\, 1_{{}_n \Scal_{n-j+1}^{(1)}}\Big)\\
&=\upsilon^{-1}\, {}_n \onebf^{\mathbf{ss}}_{1,\, \alpha n}\, 1_{{}_n \Scal_n^{(j)}}+T_n^{j-1}\Big(1_{{}_n \Scal_{n}^{(1)}}\,{}_n \onebf^{\mathbf{ss}}_{1,\, \alpha n}\Big)-\upsilon^{-1}\, {}_n \onebf^{\mathbf{ss}}_{1,\, \alpha n+j-1}\, 1_{{}_n \Scal_{n-j+1}^{(1)}}\\
&=\upsilon^{-1}\, {}_n\onebf^{\mathbf{ss}}_{1,\, \alpha n}\, 1_{{}_n \Scal_k^{(j)}}+\upsilon^{-1}\,  {}_n\onebf^{\mathbf{ss}}_{1,\, \alpha n+j}  \ .
\end{align}
For $j=k+1$, we obtain
\begin{align}
1_{{}_n \Scal_k^{(j)}}\, {}_n\onebf^{\mathbf{ss}}_{1,\, \alpha n} =&\Big(\upsilon\, 1_{{}_n \Scal_{n}^{(1)}}\, 1_{{}_n \Scal_{k}^{(k)}} -1_{{}_n \Scal_{k}^{(k)}}\,  1_{{}_n \Scal_{n}^{(1)}}\Big)\, {}_n\onebf^{\mathbf{ss}}_{1,\, \alpha n}\\
=&\upsilon\, 1_{{}_n\Scal_n^{(1)}}\, \Big(\upsilon\, {}_n\onebf^{\mathbf{ss}}_{1,\, \alpha n} \, 1_{{}_n \Scal_k^{(k)}}\Big)-1_{{}_n \Scal_k^{(k)}}\Big(\upsilon^{-1}\, {}_n\onebf^{\mathbf{ss}}_{1,\, \alpha n} \, 1_{{}_n \Scal_n^{(1)}}+\upsilon^{-1}\, {}_n\onebf^{\mathbf{ss}}_{1,\, \alpha n+1}\Big)\\
=&\Big(\upsilon^{-1}\, {}_n\onebf^{\mathbf{ss}}_{1,\, \alpha n} \, 1_{{}_n \Scal_n^{(1)}}+\upsilon^{-1}\,{}_n \onebf^{\mathbf{ss}}_{1,\, \alpha n+1}\Big)\, 1_{{}_n \Scal_k^{(k)}}- {}_n\onebf^{\mathbf{ss}}_{1,\, \alpha n} \,1_{{}_n \Scal_k^{(k)}}\, 1_{{}_n \Scal_n^{(1)}}\\
&-\upsilon^{-1}\, T_n\Big(1_{{}_n \Scal_{k+1}^{(k)}}\,{}_n\onebf^{\mathbf{ss}}_{1,\, \alpha n}\Big)={}_n \onebf^{\mathbf{ss}}_{1,\, \alpha n}\, 1_{{}_n \Scal_k^{(j)}}+(\upsilon-\upsilon^{-1})\, {}_n \onebf^{\mathbf{ss}}_{1,\, \alpha n+1}\, 1_{{}_n \Scal_k^{(k)}}\ ,
\end{align}
while for $k<j-1$, one has
\begin{align}
1_{{}_n \Scal_k^{(j)}}\, {}_n\onebf^{\mathbf{ss}}_{1,\, \alpha n} = &\Big(\upsilon\, 1_{{}_n \Scal_{n+k+1-j}^{(1)}}\, 1_{{}_n \Scal_{k}^{(j-1)}} -1_{{}_n \Scal_{k}^{(j-1)}}\,  1_{{}_n \Scal_{n+k+1-j}^{(1)}}\Big)\, {}_n\onebf^{\mathbf{ss}}_{1,\, \alpha n}\\
= &\upsilon\, 1_{{}_n \Scal_{n+k+1-j}^{(1)}}\Big({}_n \onebf^{\mathbf{ss}}_{1,\, \alpha n}\, 1_{{}_n \Scal_k^{(j-1)}}+(\upsilon-\upsilon^{-1})\, {}_n \onebf^{\mathbf{ss}}_{1,\, \alpha n+j-1-k}\, 1_{{}_n \Scal_k^{(k)}}\Big)\\
&-1_{{}_n \Scal_k^{(j-1)}}\, {}_n\onebf^{\mathbf{ss}}_{1,\, \alpha n} \,  1_{{}_n \Scal_{n+k+1-j}^{(1)}}
={}_n \onebf^{\mathbf{ss}}_{1,\, \alpha n}\, 1_{{}_n \Scal_k^{(j)}}+(\upsilon-\upsilon^{-1})\, {}_n \onebf^{\mathbf{ss}}_{1,\, \alpha n+j-k}\, 1_{{}_n \Scal_k^{(k)}}\\
&+(\upsilon-\upsilon^{-1})\,  {}_n \onebf^{\mathbf{ss}}_{1,\, \alpha n+j-1-k}\, \Big(1_{{}_n \Scal_{n+k+1-j}^{(1)}}\, 1_{{}_n \Scal_k^{(k)}}-1_{{}_n \Scal_k^{(k)}}\,1_{{}_n \Scal_{n+k+1-j}^{(1)}}\Big) \ .
\end{align} 
Because of the assumptions on $j$ and since $j-1>k$, we have 
\begin{align}
1_{{}_n \Scal_{n+k+1-j}^{(1)}}\, 1_{{}_n \Scal_k^{(k)}}=1_{{}_n \Scal_k^{(k)}}\,1_{{}_n \Scal_{n+k+1-j}^{(1)}}\ ,
\end{align}
hence we get the result. The case $k=j$ is straightforward. 

Finally, the assertion follows from the fact that
\begin{align}
1_{{}_n \Scal_k^{(j)}}\, {}_n\onebf^{\mathbf{ss}}_{1,\, d}=1_{{}_n \Scal_k^{(j)}}\, {}_n\onebf^{\mathbf{ss}}_{1,\, \alpha n+\beta}=T_n^{\,\beta}\Big(1_{{}_n \Scal_{k+\beta}^{(j)}}\, {}_n\onebf^{\mathbf{ss}}_{1,\, \alpha n}\Big)
\end{align}
for $d\in \Z$ such that $d=\alpha n+\beta$, with $\alpha, \beta\in \Z$ and $0\leq \beta\leq n-1$.
\end{proof}
\begin{corollary}\label{cor:heckeaction}
Let $1\leq i \leq n$ and $1\leq j\leq n-1$ be integers, let $d\in \Z$. Then
\begin{align}
\omega_n\big(1_{{}_n \Scal_i^{(j)}}\, {}_n\onebf^{\mathbf{ss}}_{1,\, d}\big)=
\begin{cases}
\upsilon^{-1}\,  {}_n\onebf^{\mathbf{ss}}_{1,\, d+j} & \text{if } \nfrapart{i+d}=0\ , \\[2pt]
0 & \text{otherwise}\ . 
\end{cases}
\end{align}
\end{corollary}
\begin{lemma}
Let $1\leq i\leq n$ and $m\neq 0$ be integers and $d\in \Z$. Then
 \begin{align}
\omega_n\big(1_{{}_n \Scal_i^{(mn)}}\, {}_n\onebf^{\mathbf{ss}}_{1,\, d}\big)=
\begin{cases}
\upsilon^{-m}\,  {}_n\onebf^{\mathbf{ss}}_{1,\, d+mn} & \text{if } \nfrapart{i+d}=0\ , \\[2pt]
0 & \text{otherwise}\ . 
\end{cases}
\end{align}
\end{lemma}
\begin{proof}
Let $M$ be a line bundle on $X$. Let us first compute $1_{{}_n \Scal_i^{(mn)}}\, 1_{\pi_n^\ast M}$. By definition
\begin{align}
1_{{}_n \Scal_i^{(mn)}}\, 1_{\pi_n^\ast M}=\upsilon^{\langle [{}_n \Scal_i^{(mn)}], [\pi_n^\ast M]\rangle}\, \sum_\Ecal \, \frac{\Pbf_{{}_n \Scal_i^{(m)}, \pi_n^\ast M}^\Ecal}{\abf({}_n \Scal_i^{(mn)}) \abf(\pi_n^\ast M)}\, 1_\Ecal \ ,
\end{align}
where we denote by $\Pbf_{\Ecal_1, \Ecal_2}^\Fcal$ the cardinality of the set of short exact sequences $0\to \Ecal_2\to \Fcal\to \Ecal_1\to 0$ for any fixed triple of coherent sheaves $\Fcal, \Ecal_1, \Ecal_2$; by $\abf(\Ecal)$ the cardinality of the automorphism group of a coherent sheaf $\Ecal$. Now, $\langle [{}_n \Scal_i^{(mn)}], [\pi_n^\ast M]\rangle=-m$, 
\begin{align}
\frac{\Pbf_{{}_n \Scal_i^{(mn)}, \pi_n^\ast M}^\Ecal}{\abf(\pi_n^\ast M)}=
\begin{cases}
\# \Hom^{\mathsf{surj}}(\Ocal_{X_n}, {}_n\Scal_n^{(mn)})=q^m-q^{m-1} & \text{if $i=n$}\ ,\\
0 & \text{otherwise}\ ,
\end{cases}
\end{align}
and $\abf({}_n \Scal_n^{(mn)})=q^m-q^{m-1}$. Thus, $\omega_n\big(1_{{}_n \Scal_i^{(mn)}}\, 1_{\pi_n^\ast M}\big)=\upsilon^{-m}\, 1_{\pi_n^\ast M\otimes \Lcal_n^{\otimes\, mn}}$. Therefore
\begin{align}
\omega_n\big(1_{{}_n \Scal_i^{(mn)}}\, {}_n\onebf^{\mathbf{ss}}_{1,\, d}\big)=T_n^{\, \beta}\big(\omega_n\big(1_{{}_n \Scal_{i+\beta}^{(mn)}}\, {}_n\onebf^{\mathbf{ss}}_{1,\, \alpha n}\big)\big)=T_n^{\, \beta}(\onebf^{\mathbf{ss}}_{1,\, \alpha n+mn})
\end{align}
and we obtain the assertion. Here, $\alpha, \beta\in\Z$, $0\leq \beta\leq n-1$, are such that $d=\alpha n+\beta$.
\end{proof}
\begin{corollary}\label{cor:heckeaction-II}
Let $1\leq i\leq n-1$ and $m$ be integers and $d\in \Z$. Then
\begin{align}
\omega_n\big(1_{{}_n \Scal_i^{(m n+i)}}\, {}_n\onebf^{\mathbf{ss}}_{1,\, d}\big)= \delta_{\nfrapart{i+d}, 0}\, \upsilon^{-m-1}\, {}_n\onebf^{\mathbf{ss}}_{1,\, d+m n+i}\ .
\end{align}
\end{corollary}
\begin{definition}
Let $e\in \Z_{\geq 0}$. Define
\begin{align}
& {}_1 \theta_{0, \, e} =\theta_{0, \, e}\coloneqq\upsilon^{e}\, \sum_{\genfrac{}{}{0pt}{}{(m_x)_x}{\sum_x\, m_x\deg(x)=e}}\, \prod_{\genfrac{}{}{0pt}{}{x\in X}{m_x\neq 0}}\, (1-\upsilon^{-2\deg(x)})\, 1_{T_x^{(m_x)}}\ ,\\
& {}_n\theta_{0,\, e} \coloneqq\\
& 
\begin{cases}
\Omega_n\big(\theta_{0,\, e/n}\big) & \text{for } \nfrapart{e}=0 \ ,\\[8pt]
\displaystyle \upsilon^{\, \nintpart{e}+1}\, \sum_{\genfrac{}{}{0pt}{}{(m_x)_x}{\sum_x\, m_x\deg(x)= \nintpart{e}}}\, \prod_{\genfrac{}{}{0pt}{}{x\in X\setminus\{p\}}{m_x\neq 0}}\, (1-\upsilon^{-2\deg(x)})\, 1_{\Tcal_x^{(m_x)}}\, (1-\upsilon^{-2})\, 1_{{}_n \Scal_{\nfrapart{e}}^{(m_p n+\nfrapart{e})}} & \text{otherwise} \ .
 \end{cases}
\end{align}
\end{definition}
\begin{remark}
Note that the definition of $\theta_{0,e}$ as above agrees with the definition in \eqref{eq:torsionrelations}.
\end{remark}
\begin{proposition}\label{prop:omega}
Let $d\in \Z$. Then
\begin{align}
\omega_n\Big(\sum_{e\in \Z_{\geq 0}}\, {}_n\theta_{0,\, e}\, {}_n\onebf^{\mathbf{ss}}_{1,\, d}\Big)=\sum_{\alpha\in \Z_{\geq 0}}\, \xi_\alpha^{(d)}\, {}_n\onebf^{\mathbf{ss}}_{1,\, d+\alpha}\ , 
\end{align}
where
\begin{align}
\xi_\alpha^{(d)}\coloneqq
\begin{cases}
\displaystyle \xi_{\alpha/n} & \text{if } \nfrapart{d}=0, \, \nfrapart{\alpha}=0\ , \\[8pt]
\displaystyle \xi_{\nintpart{\alpha}}-q^{-1}\xi_{\nintpart{\alpha}}^\circ  & \text{if } \nfrapart{d}\neq 0, \, \nfrapart{\alpha+d}=0\ , \\[8pt]
\displaystyle \xi_{\nintpart{\alpha}}^\circ  & \text{if } \nfrapart{d}\neq 0, \, \nfrapart{\alpha}=0\ ,\\[8pt]
0 & \text{otherwise}\ .
\end{cases}
\end{align}
Here the complex numbers $\xi_s$ and the complex numbers $\xi_s^\circ$, for $s\in \Z_{\geq0}$, are given by
\begin{align}
\xi_{s}\coloneqq\sum_{\genfrac{}{}{0pt}{}{(m_x)_{x\in X}, m_x\in \Z_{\geq 0}}{\sum_x\,m_x\deg(x)=s}}\, \prod_{\genfrac{}{}{0pt}{}{x\in X}{m_x\neq 0}}\, \Big(1-\upsilon^{-2\deg(x)}\Big) \quad\text{and}\quad \xi_{s}^\circ\coloneqq\sum_{\genfrac{}{}{0pt}{}{(m_x)_{x\in X}, m_x\in \Z_{\geq 0}}{m_p=0, \, \sum_x\,m_x\deg(x)=s}}\, \prod_{\genfrac{}{}{0pt}{}{x\in X}{m_x\neq 0}}\, \Big(1-\upsilon^{-2\deg(x)}\Big)\ .
\end{align}
\end{proposition}
\begin{proof}
Let us first observe that
\begin{multline}
\omega_n\Big(\sum_{e\in \Z_{\geq 0}}\, {}_n\theta_{0,\, e}\, {}_n\onebf^{\mathbf{ss}}_{1,\, d}\Big)=\omega_n\Big(\Big(\sum_{s\geq 0}\, \upsilon^{\,s}\, \sum_{\genfrac{}{}{0pt}{}{(m_x)_x}{m_p=0 \ ,\ \sum_x\, m_x\deg(x)=s}}\, \prod_{\genfrac{}{}{0pt}{}{x\in X\setminus\{p\}}{m_x\neq 0}}\, (1-\upsilon^{-2\deg(x)})\, 1_{\Tcal_x^{(m_x)}}\Big)\\
\times\, \Big(1+\sum_{t\geq 1}\, \upsilon^{\, \nintpart{t}+1-\delta_{\nfrapart{t},0}}\, (1-\upsilon^{-2})\, 1_{{}_n \Scal_{\nfrapart{t}}^{(t)}}\Big)\,  {}_n\onebf^{\mathbf{ss}}_{1,\, d}\Big)\Big)\ .
\end{multline}
By applying Corollary~\ref{cor:heckeaction-II}, the above quantity is equal to
\begin{multline}
\omega_n\Big(\Big(\sum_{s\geq 0}\, \upsilon^{\,s}\, \sum_{\genfrac{}{}{0pt}{}{(m_x)_x}{m_p=0 \ ,\ \sum_x\, m_x\deg(x)=s}}\, \prod_{\genfrac{}{}{0pt}{}{x\in X\setminus\{p\}}{m_x\neq 0}}\, (1-\upsilon^{-2\deg(x)})\, 1_{\Tcal_x^{(m_x)}}\Big)\\
\times \, ({}_n\onebf^{\mathbf{ss}}_{1,\, d}+(1-\upsilon^{-2})\, \sum_{t\geq 1}\, \delta_{\nfrapart{d+t}, 0}\, {}_n\onebf^{\mathbf{ss}}_{1,\, d+t}\Big)\Big)\ .
\end{multline}
Recall that $\omega_n\Big(1_{\Tcal_x^{(m_x)}}\, {}_n\onebf^{\mathbf{ss}}_{1,\, d}\Big)=\upsilon^{\, -m_x\deg(x)}\, \, {}_n\onebf^{\mathbf{ss}}_{1,\, d+m_x\deg(x)n}$ we get the assertion with
\begin{align}
\xi_\alpha^{(d)}\coloneqq
\begin{cases}
\displaystyle \xi_{\nintpart{\alpha}}^\circ+(1-\upsilon^{-2})\, \sum_{\beta=0}^{\nintpart{\alpha}-1}\,\xi_\beta^\circ & \text{if } \nfrapart{d}=0, \, \nfrapart{\alpha}=0\ , \\[8pt]
\displaystyle (1-\upsilon^{-2})\, \sum_{\beta=0}^{\nintpart{\alpha}}\,\xi_\beta^\circ& \text{if } \nfrapart{d}\neq 0, \, \nfrapart{\alpha+d}=0\ , \\[8pt]
\displaystyle \xi_{\nintpart{\alpha}}^\circ  & \text{if } \nfrapart{d}\neq 0, \, \nfrapart{\alpha}=0\ ,\\[8pt]
0 & \text{otherwise}\ .
\end{cases}
\end{align}
The proposition follows by osserving that
\begin{align}
 \xi_{m}^\circ+(1-\upsilon^{-2})\, \sum_{\beta=0}^{m-1}\,\xi_\beta^\circ=\xi_{m}\ .
\end{align}
\end{proof}
\begin{corollary}[{See \cite[Corollary~1.4]{art:schiffmannvasserot2012}}]\label{cor:xie}
As a series in $\C[[z]]$, we have
\begin{align}
\sum_{e\geq 0}\, \xi_e\, z^e=\frac{\zeta_X(z)}{\zeta_X(\upsilon^{-2}\, z)}\ ,
\end{align}
where $\zeta_X(z)$ is the zeta function of $X$.
\end{corollary}
\begin{corollary}[{\cite[Lemma~3.14]{art:lin2014}}]\label{cor:xiecirc}
As a series in $\C[[z]]$, we have
\begin{align}
\sum_{e\geq 0}\, \xi_e^\circ\,  z^e=\frac{\zeta_{X\setminus\{p\}}(z)}{\zeta_{X\setminus\{p\}}(\upsilon^{-2}\, z)}=\frac{\zeta_X(z)}{\zeta_X(\upsilon^{-2}\, z)}\frac{1-z}{1-\upsilon^{-2}\, z}\ ,
\end{align}
where $\zeta_X(z)$ and $\zeta_{X\setminus\{p\}}(z)$ are the zeta functions of $X$ and $X\setminus\{p\}$ respectively.
\end{corollary}

We conclude this section, by computing the coproduct of the ${}_n\onebf^{\mathbf{ss}}_{1,\, d}$'s: this is a key result to compute the shuffle algebra presentation of $\Ubf_n^>$ in the main body of the paper.
\begin{proposition}\label{prop:coproduct}
Let $d\in\Z$. Then
\begin{align}
\tilde \Delta\big({}_n\onebf^{\mathbf{ss}}_{1,\, d}\big)={}_n\onebf^{\mathbf{ss}}_{1,\, d}\otimes 1 + \sum_{e\in\Z_{\geq 0}}
T_n^{\, \nfrapart{d}}\big({}_n\theta_{0,\, e}\big)\,\kbf_{(1,\, d-e)}\otimes {}_n\onebf^{\mathbf{ss}}_{1,\, d-e}\ .
\end{align}
\end{proposition}
\begin{proof}
Set $k\coloneqq\nintpart{d}$ and $i\coloneqq \nfrapart{d}$. Then
\begin{equation}
\tilde \Delta\big({}_n \onebf^{\mathbf{ss}}_{1,\, kn+i}\big)=\tilde \Delta\Big(\sum_{\genfrac{}{}{0pt}{}{M\in \Pic(X)}{\deg(M)=k}}\, 1_{\pi_n^\ast M\otimes \Lcal_n^{\otimes\, i}}\Big)\ ,
\end{equation}
and
\begin{align}
\tilde \Delta\big(1_{\pi_n^\ast M\otimes \Lcal_n^{\otimes\, i}}\big)&=\sum_{\Ecal_1,\Ecal_2}\, \upsilon^{\langle \Ecal_1, \Ecal_2\rangle}\, \Pbf_{\Ecal_1,\Ecal_2}^{\pi_n^\ast M\otimes \Lcal_n^{\otimes\, i}}\, \frac{\abf(\Ecal_1)\abf(\Ecal_2)}{\abf(\pi_n^\ast M\otimes \Lcal_n^{\otimes \, i})}1_{\Ecal_1}\, \kbf_{(\rk(\Ecal_2), \underline{\deg}(\Ecal_2))}\otimes 1_{\Ecal_2}\\
&=\sum_{\Ecal_1\simeq \pi_n^\ast M\otimes \Lcal_n^{\otimes\, i}}\, 1_{\Ecal_1}\otimes 1\\
&+\sum_{\rk(\Ecal_1)=0, \Ecal_2\simeq \pi_n^\ast N\otimes \Lcal_n^{\otimes\, s}}\, \upsilon^{\langle \Ecal_1, \Ecal_2\rangle}\, \Pbf_{\Ecal_1,\Ecal_2}^{\pi_n^\ast M\otimes \Lcal_n^{\otimes\, i}}\, \frac{\abf(\Ecal_1)\abf(\Ecal_2)}{\abf(\pi_n^\ast M\otimes \Lcal_n^{\otimes\, i})}1_{\Ecal_1}\, \kbf_{(1, \deg(N)n+ s)}\otimes 1_{\Ecal_2}\ .
\end{align}
Note that
\begin{multline}
\sum_{\rk(\Ecal_1)=0, \Ecal_2\simeq \pi_n^\ast N\otimes \Lcal_n^{\otimes\, s}}\, \upsilon^{\langle \Ecal_1, \Ecal_2\rangle}\, \Pbf_{\Ecal_1,\Ecal_2}^{\pi_n^\ast M\otimes \Lcal_n^{\otimes\, i}}\, \frac{\abf(\Ecal_1)\abf(\Ecal_2)}{\abf(\pi_n^\ast M\otimes \Lcal_n^{\otimes\, i})}1_{\Ecal_1}\, \kbf_{(1, \deg(N)n+ s)}\otimes 1_{\Ecal_2}=\\
\sum_{\rk(\Fcal_1)=0, \Fcal_2\simeq \pi_n^\ast L\otimes \Lcal_n^{\otimes\, \ell}}\, \upsilon^{\langle \Fcal_1, \Fcal_2\rangle}\, \Pbf_{\Fcal_1,\Fcal_2}^{\Ocal_{X_n}}\, \abf(\Fcal_1) 1_{\Fcal_1\otimes \pi_n^\ast M\otimes \Lcal_n^{\otimes\, i}}\, \kbf_{(1,\, \deg(L)n+ \ell+kn+ i)}\otimes 1_{\Fcal_2\otimes \pi_n^\ast M\otimes \Lcal_n^{\otimes\, i}}\ .
\end{multline}
Now, $\Fcal_1$ is a torsion sheaf which is a quotient of $\Ocal_{X_n}$, hence it can be only of the form
\begin{equation}\label{eq:decomFcal_1}
\Fcal_1\simeq \bigoplus_{x\in X\setminus\{p\}} \Tcal_x^{(m_x)}\oplus {}_n \Scal_j^{(m_p\, n +j)}
\end{equation}
for some $m_x\in \Z_{\geq 0}$, with $x\in  X$, and some $j\in \{0, \ldots, n-1\}$. Here we formally set $\Tcal_x^{(0)}={}_n \Scal_n^{(0)}=0$. Moreover,
since $\overline{\Ocal_{X_n}}=\overline{\Fcal_1} +\overline{\Fcal_2}$, we get
\begin{align}
\deg(L)&=-\sum_{x\in X}\, m_x\deg(x) -1+\delta_{\nfrapart{j},0}\ ,\\
\ell&=\nfrapart{-j}\ .
\end{align}
In particular, $\langle \Fcal_1, \Fcal_2\rangle = \deg(L)$. 

Moreover, $\Pbf_{\Fcal_1,\Fcal_2}^{\Ocal_{X_n}}\, \abf(\Fcal_1)$ is equal to the cardinality of $\Hom^{\mathsf{surj}}(\Ocal_{X_n}, \Fcal_1)$. Since a morphism $\Ocal_{X_n}\to \Fcal_1$ is surjective if and only if all its components, with respect to the decomposition \eqref{eq:decomFcal_1} of $\Fcal_1$ are surjective, we need to compute the cardinalities of the space of surjective morphisms from $\Ocal_{X_n}$ to each factor in \eqref{eq:decomFcal_1}. By using the exactness of the functor ${\pi_n}_\ast$ and the computations of \cite[Example~4.12]{book:schiffmann2012} we have for $m_x, m_p, j\neq 0$
\begin{align}
\# \Hom^{\mathsf{surj}}(\Ocal_{X_n}, \Tcal_x^{(m_x)})&=\# \Hom^{\mathsf{surj}}(\Ocal_{X}, T_x^{(m_x)})=q^{m_x\deg(x)}-q^{(m_x-1)\deg(x)} \ , \\
\# \Hom^{\mathsf{surj}}(\Ocal_{X_n}, {}_n\Scal_n^{(m_p\, n)})&=\# \Hom^{\mathsf{surj}}(\Ocal_{X}, T_p^{(m_p)})=q^{m_p}-q^{m_p-1} \ , \\
\# \Hom^{\mathsf{surj}}(\Ocal_{X_n}, \Scal_j^{(m_p\, n+j)})&=\# \Hom^{\mathsf{surj}}(\Ocal_{X}, T_p^{(m_p+1)})=q^{m_p+1}-q^{m_p}\ .
\end{align}
Summarizing, we get
\begin{multline}
\sum_{\rk(\Fcal_1)=0, \Fcal_2\simeq \pi_n^\ast L\otimes \Lcal_n^{\otimes\, \ell}}\, \upsilon^{\langle \Fcal_1, \Fcal_2\rangle}\, \Pbf_{\Fcal_1,\Fcal_2}^{\Ocal_{X_n}}\, \abf(\Fcal_1) 1_{\Fcal_1\otimes \pi_n^\ast M\otimes \Lcal_n^{\otimes\, i}}\, \kbf_{(1,\, \deg(L)n+ \ell+kn+ i)}\otimes 1_{\Fcal_2\otimes \pi_n^\ast M\otimes \Lcal_n^{\otimes\, i}}\\
=\sum_{\genfrac{}{}{0pt}{}{\Fcal_1\simeq \bigoplus_{x\in X} \pi_n^\ast T_x^{(m_x)}}{\Fcal_2\simeq \pi_n^\ast L}}\, \upsilon^{-\sum_x\, m_x\deg(x)}\, \prod_{\genfrac{}{}{0pt}{}{x\in X}{m_x\neq 0}}\, \big(q^{m_x\deg(x)}-q^{(m_x-1)\deg(x)}\big) \\
\shoveright{\times 1_{\pi_n^\ast T_x^{(m_x)}}\, \kbf_{(1,\, d-\sum_x\, m_x\deg(x))}\otimes 1_{\pi_n^\ast L\otimes \pi_n^\ast M\otimes \Lcal_n^{\otimes\, i}}}\\
+\sum_{j=1}^{n-1}\,\sum_{\genfrac{}{}{0pt}{}{\Fcal_1\simeq \bigoplus_{x\in X\setminus\{p\}} \Tcal_x^{(m_x)}\oplus {}_n \Scal_j^{(m_p\, n +j)}}{\Fcal_2\simeq \pi_n^\ast L\otimes \Lcal_n^{\otimes\, n-j}}}\, \upsilon^{-\sum_x\, m_x\deg(x)-1}\,\prod_{\genfrac{}{}{0pt}{}{x\in X\setminus\{p\}}{m_x\neq 0}}\, \big(q^{m_x\deg(x)}-q^{(m_x-1)\deg(x)}\big)\\
\times 1_{\pi_n^\ast T_x^{(m_x)}}\,q\big(q^m-q^{m-1}\big)\, T_n^{\, i}\big(1_{{}_n \Scal_j^{(m_p n+j)}}\big) \, \kbf_{(1,\, d-\sum_x\, m_x\deg(x)-1+n-j)}\otimes 1_{\pi_n^\ast L\otimes\Lcal_n^{\otimes\, n-j}\otimes \pi_n^\ast M\otimes \Lcal_n^{\otimes\, i}}\ .
\end{multline}
Thus, we get the assertion.
\end{proof}
\begin{remark}
For $n=1$, one recovers the non-stacky curve result (cf.\ \cite[Example~4.12]{book:schiffmann2012})
\begin{align}
\tilde \Delta(\onebf^{\mathbf{ss}}_{1, \, d})=\onebf^{\mathbf{ss}}_{1,\, d}\otimes 1+\sum_{e\geq 0}\, \theta_{0,\, e}\kbf_{(1,\, d-e)}\otimes \onebf^{\mathbf{ss}}_{1,\, d-e}\ .
\end{align}
\end{remark}

\bigskip\section{$\Ubf_\upsilon(\slfrak(S^1))$ from mirror symmetry}\label{sec:tatsuki}

\begin{center}
\textit{by Tatsuki Kuwagaki\footnote{\textsc{Kavli IPMU (WPI), UTIAS, The University of Tokyo, Kashiwa, Chiba 277-8583, Japan}. \newline \textit{Email address:} \texttt{tatsuki.kuwagaki@ipmu.jp}.}}
\end{center}

In the body of this paper, Sala--Schiffmann defined $\Ubf_\upsilon(\slfrak(S^1))$ by studying Hall algebras of coherent sheaves over infinite root stacks. In this note, we give a natural explanation of Sala--Schiffmann's description using mirror symmetry.

\subsection*{Acknowledgments}
The author thanks Francesco Sala for his intriguing talk at GTM seminar at IPMU, which motivated this note. The author also thanks Francesco and Olivier Schiffmann for their kindness to include this note as an appendix of their work. This work was supported by World Premier International Research Center Initiative (WPI), MEXT, Japan and JSPS KAKENHI Grant Number JP18K13405.

\subsection{Mirror symmetry for $\bP^1_{\infty}$}

Let $\bP^1_n$ be the projective line over $\bC$ with an $n$-gerbe at $\infty$. Hori--Vafa~\cite{HV} mirror of $\bP^1_n$ is given by a Landau--Ginzburg model $W=z+\frac{1}{z^n}$ over $\bC^\ast$. A traditional category associated with a Landau--Ginzburg model is Fukaya--Seidel category \cite{Seidelmutation}. Here, however, we will use a different description: the partially wrapped Fukaya category. We view $\bC^\ast$ as the cotangent bundle $T^\ast S^1$ of the circle $S^1=\bR/\bZ$. Let $p\colon \bR\rightarrow \bR/\bZ=S^1$ be the quotient map. The Lagrangian skeleton $\Lambda_n$ associated with $W$, is given by the FLTZ skeleton~\cite{FLTZ}
\begin{align}
\Lambda_n\coloneqq T^\ast_{S^1}S^1\cup \bigcup_{k\in \bZ}T^{\ast,{<0}}_{p(\frac{k}{n})}S^1
\end{align}
where $T^\ast_A S^1$ is the conormal bundle of $A$ inside $S^1$ and $T^{\ast,<0}_xS^1$ is the negative part of the cotangent fiber at $x$. We would like to consider the Fukaya category associated with $\Lambda_n$. However there exists a further different description. For a cotangent bundle $T^\ast X$, the derived category of the infinitesimally wrapped Fukaya category is equivalent to the bounded derived category of $\bR$-constructible sheaves over $X$ by Nadler--Zaslow \cite{NZ, Nad}: $\mathfrak{Fuk}(T^\ast X)\cong \cSh(X)$. Concerning this result, we can use certain subcategory of $\cSh(S^1)$ as a model for the Fukaya category associated with $\Lambda_n$ (more precise statement about this consideration is known as Kontsevich's conjecture. See \cite{Kon2, GPS} for further information). 

For a sheaf (or a complex of sheaves) $\Ecal$ over a manifold $X$, Kashiwara--Schapira~\cite{KS} defined the microsupport $\SS(\Ecal)\subset T^\ast X$. The definition of ``certain subcategory'' mentioned above is the full subcategory of $\cSh(S^1)$ spanned by the objects whose microsupports are contained in $\Lambda_n$. We denote it by $\cSh_{\Lambda_n}(S^1)$.

In the below, let us change the base field from $\bC$ to $k\coloneqq\bF_q$ i.e. $\bP^1_n$ is defined over $\bF_q$ and constructible sheaves are $\bF_q$-valued. Let $D(\Coh(\bP^1_{n}))$ be the bounded derived category of coherent sheaves.
We have the following mirror symmetry result:
\begin{theorem}[cf.\cite{FLTZ, FLTZ3}]\label{theorem2.1}
There exists an equivalence
\begin{align}
D(\Coh(\bP^1_{n}))\simeq \cSh_{\Lambda_n}(S^1)\ .
\end{align}
\end{theorem}
Over $\bC$, this mirror symmetry is known as ``coherent-constructible correspondence'' initiated by Bondal, formulated by Fang--Liu--Treumann--Zaslow, proved in full generality by the author~\cite{BO, FLTZ, Kuw2}. For $\bP^1_n$, the statement over $\bF_q$ remains true:
\begin{proof}[Proof of Theorem \ref{theorem2.1}]
There exists a Beilinson collection of $D(\Coh(\bP^1_{n}))$ over $\bF_q$ \cite{Canonaco}. By using the equivalence over $\bC$, one can also find the candidate for the exceptional collection on the RHS, which one can verify that it is indeed so also over $\F_q$. By calculating the endomorphism algebras of these exceptional collections, one can conclude the equivalence.
\end{proof}

The forgetful morphism $\bP^1_{nm}\rightarrow \bP^1_n$ induces the pull-back functor $D(\Coh(\bP^1_{n}))\rightarrow D(\Coh(\bP^1_{nm}))$. On the mirror side, this morphism realized by the inclusion $\Lambda_n\subset \Lambda_{nm}$. We set
\begin{align}
\Lambda_{\infty}\coloneqq T^\ast_{S^1}S^1\cup \bigcup_{q\in \bQ}T^{\ast,{<0}}_{p(q)}S^1 \ .
\end{align}
Then we have
\begin{corollary}\label{2.3}
There exists an equivalence
\begin{equation}
D(\Coh(\bP^1_{\infty}))\simeq \cSh_{\Lambda_\infty}(S^1)\ .
\end{equation}
\end{corollary}

One can deduce the following corollary easily from the construction of the equivalence above.
\begin{corollary}\label{CCC}
The abelian category obtained as the essential image of $\Coh_{\mathsf{tors}}(\bP^1_\infty)$, the category of torsion sheaves at $\infty$\footnote{In the main body of the paper, this category is denoted as $\Coh_{p_\infty}(\PP^1_\infty)$, with $p=\infty \in \PP^1$.}, in $\cSh(S^1)$ is the full subcategory spanned by finite direct sums of the following $k^{S^1}_J$ where $J$ is a half-interval of the form $(a,b]\subset \bR$ with $a, b\in \bQ$ and $k^{S^1}_J\coloneqq p_!k_J$.
\end{corollary}
We denote the category by $\Sh^\bQ_{<0}(S^1)$. Below, we use $k_J$ for $p_!k_J^{S^1}$ if the context is clear. Note that our intervals are open-closed but not closed-open as in the body of this paper.

\subsection{$\Ubf_\upsilon(\slfrak(S^1))$ from the mirror side}\label{sec:S1-mirror}

The category $\Sh^\bQ_{<0}(S^1)$ gives a natural explanation of Sala--Schiffmann's description Theorem~\ref{thm:Cinfty} as follows.

We give a direct proof of the following (although one can prove by considering Corollary~\ref{2.3}):
\begin{theorem}\label{theorem1}
Let $\Acal$ be the algebra generated by $\lc E_J, K_{J'}^\pm\relmid J, J'\subseteq S^1_{\Q}\ , J\neq S^1_{\Q} \rc$ where $J$ (resp.\ $J'$) runs over all strict rational open-closed intervals (resp. rational open-closed intervals) modulo the relations appearing in Theorem~\ref{thm:Cinfty}. There exists an isomorphism of associative algebras
\begin{align}
\phi\colon \Acal\rightarrow \fH^{\mathsf{tw}}(\Sh^\bQ_{<0}(S^1))\ .
\end{align}
\end{theorem}
\begin{proof}
We define the morphism by the assignment
\begin{align}
E_J\mapsto \upsilon^{1/2}1_{k_J}, K_J^{\pm 1}\mapsto \bold{k}^{\pm k_J} 
\end{align}
where $1_{k_J}$ is the characteristic function of $k_J$ and $\bold{k}$ is the indefinite of the group ring $\widetilde{\bQ}[\Ksfnum(\Sh^\bQ_{<0}(S^1))]$.

First, we will check the relations in Theorem~\ref{thm:Cinfty} for $\fH^{\mathsf{tw}}(\Sh^\bQ_{<0}(S^1))$.
\begin{itemize}\itemsep0.2cm
\item \eqref{eq:Serre1-I} This is the definition of the product of the twisted Hall algebra; Actually, the description of Euler form given in \eqref{eq:eulerform-introduction} is the Euler form of $\Sh^\bQ_{<0}(S^1)$, hence the symmetric one is the same.
\item \eqref{eq:Serre1-II} For any pair of disjoint $J_1,J_2$ of the form $J_1=(a,b]$ and $J_2=(b,c]$ with $a\neq c$ and $J_1\cup J_2$ is again a proper interval in $S^1_{\Q}$, we have an exact sequence
\begin{align}
0\rightarrow k_{(b,c]}\rightarrow k_{(a, c]}\rightarrow k_{(a,b]}\rightarrow 0\ .
\end{align}
In the Grothendieck group, this gives the desired relation $\bold{k}_{J_1\cup J_2}=\bold{k}_{J_1}+\bold{k}_{J_2}$. 
\item \eqref{eq:Joining1} Suppose the same assumption on $J_1, J_2$ as above. The Hall product is
\begin{align}
\upsilon^{1/2}1_{k_{J_1}}\cdot \upsilon^{1/2}1_{k_{J_2}}&=\upsilon(\upsilon^{-1}1_{k_{J_1\cup J_2}}+\upsilon^{-1}1_{k_{J_1}+k_{J_2}})\ ,\\
\upsilon^{1/2}1_{k_{J_2}}\cdot \upsilon^{1/2}1_{k_{J_1}}&=v1_{k_{J_1}+k_{J_2}}\ .
\end{align}
Hence we have
\begin{align}
\upsilon^{1/2}(\upsilon^{1/2}1_{k_{J_1}}\cdot \upsilon^{1/2}1_{k_{J_2}})-\upsilon^{-1/2}(\upsilon^{1/2}1_{k_{J_2}}\cdot \upsilon^{1/2}1_{k_{J_1}})=\upsilon^{1/2}1_{k_{J_1\cup J_2}}\ .
\end{align}

\item \eqref{eq:Serre3} For any $J_1, J_2$ such that $\overline{J_1}\cap \overline{J_2}=\varnothing$. there are no nontrivial extensions $k_{J_1}$ and $k_{J_2}$. This implies that the Hall product is $1_{k_{J_1}}\cdot 1_{k_{J_2}}=1_{k_{J_1}\oplus k_{J_2}}=1_{k_{J_2}}\cdot 1_{k_{J_1}}$. Hence $[1_{k_{J_1}},1_{k_{J_2}}]=0$. 
\item \eqref{eq:nest} Let $J_1=(a,b]\subsetneq J_2=(c,d]$. First assume that $a \neq c$ and $b\neq d$. There are no nontrivial $\Hom$ and $\Ext$ between $k_{J_1}$ and $k_{J_2}$. The relation \eqref{eq:nest} is trivial. Next we assume that $b=d$, then there exists exists $1$-dimensional hom-space $\Hom(k_{J_1}, k_{J_2})$. Hence we have
\begin{align}
\upsilon^{\la k_{J_1}, k_{J_2}\ra}\upsilon 1_{k_{J_1}}\cdot 1_{k_{J_2}}&=\upsilon^2 {1_{k_{J_1}+k_{J_2}}}=q{1_{k_{J_1}+k_{J_2}}}\ ,\\
\upsilon^{\la k_{J_2}, k_{J_1}\ra}\upsilon 1_{k_{J_2}}\cdot 1_{k_{J_1}}&=q{1_{k_{J_1}+k_{J_2}}}\ .
\end{align}
Since $\la k_{J_1}, k_{J_2}\ra=\la \chi_{J_1}, \chi_{{J}_2}\ra$ where the intervals are interpreted as closed-open intervals on the RHS, we get \eqref{eq:nest} for this case. Similarly, one can prove the case of $a=c$.
\end{itemize}
It is obvious that the morphism $\phi$ is surjective. One can follow the argument for the injectivity of Theorem~\ref{thm:Cinfty}.
\end{proof}
We can readapt the definitions of the coproduct \eqref{eq:coproductCinfty} and of the Green pairing \eqref{eq:GreenpairingCinfty} in our setting endowing $\fH^{\mathsf{tw}}(\Sh^\bQ_{<0}(S^1))$ with the structure of a topological $\widetilde \Q$-Hopf algebra. By passing to the reduced Drinfeld double, we obtain the following.
\begin{corollary}
We have an isomorphism
\begin{align}
\Ubf_\upsilon(\slfrak(S^1_{\Q}))\cong \mathbf{D}(\fH^{\mathsf{tw}}(\Sh^\bQ_{<0}(S^1)))\ .
\end{align}
\end{corollary}

In the above argument, the rationality of $S^1_{\Q}$ is not essential. We set
\begin{align}
\Lambda_{\bR}\coloneqq T^\ast_{S^1}S^1\cup \bigcup_{q\in \bR}T^{\ast,{<0}}_{p(q)}S^1\ .
\end{align}
Let $\Sh^\bR_{<0}(S^1)$ be the abelian full subcategory of $\cSh_{\Lambda_\bR}(S^1)$ spanned by finite direct sums of elements $k_J$, where $J$ is an open-closed interval of the form $(a,b]\subset \bR$ and $k_J\coloneqq p_!k_J$ by abuse of notation.

\begin{theorem}
Let $\Acal$ be the algebra generated by $\lc E_J, K_{J'}^\pm\relmid J, J'\subseteq S^1\ , J\neq S^1\rc$ where $J$ (resp.\ $J'$) runs over all strict open-closed intervals (resp. open-closed intervals)  modulo the relations appearing in Theorem~\ref{thm:Cinfty}. There exists an isomorphism
\begin{align}
\phi\colon \Acal\rightarrow \fH^{\mathsf{tw}}(\Sh^\bR_{<0}(S^1))\ .
\end{align}
\end{theorem}
\begin{proof}
We define the morphism by the assignment
\begin{equation}
E_J\mapsto \upsilon^{1/2}1_{k_J}, K_J^{\pm 1}\mapsto \bold{k}^{\pm k_J}\ .
\end{equation}
One can prove that this is a homomorphism by the same argument as in the proof of Theorem \ref{theorem1}. The surjectivity is obvious.

For the injectivity, we follow the argument of Theorem~\ref{thm:Cinfty}. Let $P\subset S^1$ be a finite subset. Let us consider the subalgebra $\Acal_P\subset\Acal$ generated by $F_J, K_J^{\pm 1}$ for $J$ ends in $P$. Since $\Acal_P\cong \Acal_{\vert P\vert}$ in the proof of Theorem~\ref{thm:Cinfty}, the restriction $\phi|_{\Acal_P}$ is injective. Since $\bigcup_{P}\Acal_P=\Acal$, we complete the proof.
\end{proof}

As before, we can suitably define a coproduct and a Green pairing giving rise to the structure of a topological $\widetilde \Q$-Hopf algebra on $\fH^{\mathsf{tw}}(\Sh^\bR_{<0}(S^1))$. By taking the reduced Drinfeld double, we obtain.
\begin{corollary}
We have an isomorphism
\begin{align}
\Ubf_\upsilon(\slfrak(S^1))\cong \mathbf{D}(\fH^{\mathsf{tw}}(\Sh^\bR_{<0}(S^1)))\ .
\end{align}
\end{corollary}

\subsection{Fundamental representation}
Along with Sala--Schiffmann's definition for $\Ubf_\upsilon(\slfrak(S^1_{\Q}))$, the fundamental representation of $\Ubf_\upsilon(\slfrak(S^1))$ is defined by the $\widetilde{\bQ}$-vector space, $\mathbb{V}_{S^1}\coloneqq\bigoplus_{y\in \bR}\widetilde{\bQ}\vec{u}_y$ with action given by
\begin{align}\label{4.1}
\begin{aligned}
F_{(a,b]}\bullet \vec{u}_y&\coloneqq\delta_{\{b+y\},0}\upsilon^{1/2}\vec{u}_{y+b-a}\ ,\\
E_{(a,b]}\bullet \vec{u}_y&\coloneqq\delta_{\{a+y\},0}\upsilon^{-1/2}\vec{u}_{y+a-b}\ ,\\
K^\pm_{(a,b]}\bullet \vec{u}_y&\coloneqq \upsilon^{\pm(\delta_{\{b+y\},0}-\delta_{\{a+y\},0})}\vec{u}_{y}\ .
\end{aligned}
\end{align}

Next we define $\fU_{\infty}^>[r]$. Let us set
\begin{align}
\fI^{\mathsf{ss}}_{x}&\coloneqq1_{p_!k_{[0,x]}} \text{ if $x\geq 0$}\ ,\\
\fI^{\mathsf{ss}}_{x}&\coloneqq1_{p_!k_{(x,0)}[1]} \text{ if $x< 0$}\ ,
\end{align}
for $x\in \bR$. Let $\fU_\infty^>$ be the Hall algebra generated by $\fI^\semi_x$'s. Let us denote by $\fU^>_\infty[r]$ the linear span of products
of $r$ generators $\fI^\semi_x$. We call the number $r$ the rank.

The following is an analog of Sala--Schiffmann's result:
\begin{theorem}
The natural Hecke action of $\Ubf_\upsilon(\slfrak(S^1))$ on $\fU^>_\infty[1]$ is isomorphic to $\bV_{S^1}$, after applying the automorphism \eqref{eq:automorphism}.
\end{theorem}
\begin{proof}
There exists a basis $\{\fI^\semi_x\}_{x\in \bR}$ of $\fU^>_\infty[1]$. This gives a natural identification with $\bV_{S^1}$ as vector spaces by $\vec{u}_y\mapsto \fI_{-y}^{\mathsf{ss}}$. The second and third formulas in \eqref{4.1} follow directly by computations. The first one is followed by the definition of Drinfeld double.
\end{proof}

\subsection{The composition Hall algebra of $S^1$}
Recall the definition of $\mathbf{C}_1$ (cf.\ Section~\ref{sec:sphericalHeckealgebracurve}), which is generated by
\begin{align}
\fI_{0,\,d}\coloneqq\sum_{\genfrac{}{}{0pt}{}{\Tcal\in \Tor(\bP^1)}{deg(\Tcal)=d}} {1}_\Tcal\ .
\end{align}
In the language of the constructible side, a torsion sheaf corresponds to a finite direct sum of the following:
\begin{enumerate}
\item $k$-local system on $S^1$,
\item $p_!k_{[0,n)}$ for some $n\in \bZ_{>0}$,
\item $p_!k_{(0,n]}$ for some $n\in \bZ_{\geq 0}$.
\end{enumerate}
We call such an object a torsion object. For a torsion object $E=L\oplus \bigoplus_{i=1}^Ip_!k_{[0,n_i)} \oplus \bigoplus_{j=1}^Jp_!k_{(0,m_j]}$, we set $\deg E\coloneqq\rk L+\sum_{i=1}^In_i+\sum_{j=1}^Jm_j$. Then
\begin{align}
\fI'_{0,\,d}\coloneqq\sum_{\genfrac{}{}{0pt}{}{E\text{ : torsion}}{\deg(E)=d}} {1}_E\ .
\end{align}
is the mirror of $\fI_{0,d}$.

Sala--Schiffmann defined $\Ubf^0_\infty$ as the subalgebra of $\mathbf{H}_\infty^\mathsf{tw}(X)$ generated by $\mathbf{C}_\infty=\Ubf_\upsilon^+(\slfrak(S^1_\Q))$ and the pull-back of $\mathbf{C}$. By imitating this definition, we obtain:
\begin{definition}
The composition Hall algebra $\Ubf^0_\bR$ is an algebra generated by $\mathbf{H}^{\mathsf{tw}}(\Sh^\bQ_{<0}(S^1))$ and $\{\fI'_{0,\,d}\}_{d\leq 0}$.
\end{definition}

\subsection{Other Dynkin types}

In this section, we treat other Dynkin types: $A_n, D_n, \hat{D}_n$. 

\subsubsection{$A_n$}

Set $I=[0,1]$. Consider the Lagrangian
\begin{align}
\Lambda_{A_n}\coloneqq T^\ast_II\cup \bigcup_{i=0}^nT_{i/n}^{\ast,<0}I
\end{align}
and the category of constructible sheaves microsupported on $\cSh_{\Lambda_{A_n}}(I)$. The full subcategory $\Sh_{\Lambda_{A_n}}(I)$ consisting of constant sheaves supported on open-closed intervals of the forms like $(i/n,j/n]$ forms an abelian category and is equivalent to the category of representations of the quiver $A_n$. The reduced Drinfeld double of the twisted Hall algebra of $\Sh_{\Lambda_{A_n}}(I)$ is isomorphic to the quantum group $\Ubf_\upsilon(\slfrak(n+1))$ of type $A_n$.

We define two subsets:
\begin{align}
\Lambda_{A_\bQ}&\coloneqq T^\ast_II\cup \bigcup_{a\in I\cap \bQ}T_{a}^{\ast,<0}I\\
\Lambda_{A_\bR}&\coloneqq T^\ast_II\cup \bigcup_{a\in I }T_{a}^{\ast,<0}I
\end{align}
Then one can define the full subcategory $\Sh_{\Lambda_{A_\bQ}}(I)$ (resp. $\Sh_{\Lambda_{A_{\bR}}}(I)$) consisting of constant sheaves supported on open-closed intervals of the forms like $(a,b]$ with $a,b\in I\cap \bQ$ (resp. $a,b\in I$) forms a finitary abelian category. We define
\begin{align}
\Ubf_\upsilon(\slfrak(I_{\Q}))&\coloneqq\mathbf{D}(\mathbf{H}^{\mathsf{tw}}(\Sh_{\Lambda_{A_\bQ}}(I)))\ ,\\
\Ubf_\upsilon(\slfrak(I))&\coloneqq\mathbf{D}(\mathbf{H}^{\mathsf{tw}}(\Sh_{\Lambda_{A_\bR}}(I)))\ .
\end{align}

\subsubsection{$D_n, \widehat{D}_n$}

Since one can easily imagine $\widehat{D}_n$ case from $D_n$, we only treat the latter one.

Consider the following subset in $\bR^2$:
\begin{align}\label{eq:D-case}
I\coloneqq\lc(x, 0)\relmid x\in[0,1]\rc \cup \lc (x, e^{1/x})\relmid -1< x< 0\rc\cup \lc (x, -e^{1/x})\relmid -1< x< 0\rc\ .
\end{align}
We call the components of this union $I_1, I_2, I_3$ from left to right. 
\begin{align}
\Lambda_{D_n}\coloneqq\bigcup_{i=1}^3T^\ast_{I_i}\bR^2 \cup \bigcup_{i=0}^{n-3}T^{\ast,<0}_{(i/(n-2),0)}\bR^2
\end{align}
Here $T^{\ast,<0}_{(x,y)}\bR^2$ is the subset of $T^\ast_{(x,y)}\bR^2=\{(x, y; \xi, \eta)\}$ ($\xi$ (resp.\ $\eta$) is the cotangent coordinate of $x$ (resp.\ $y$)) defined by $\xi<0$. 

Note that $I_1\cup I_2$ and $I_1\cup I_3$ are smooth curves isomorphic to closed intervals. A closed-open interval in $I$ is an interval of the form $(a,b]$ of either $I_1\cup I_2$ or $I_1\cup I_3$. Let $D$ be the closed disk with diameter $1$ in $\bR^2$. Let $\Sh_{\Lambda_{D_n}}(D)\subset \cSh_{\Lambda_{D_n}}(D)$ be the full abelian subcategory consisting of constant sheaves on open-closed intervals. One can easily see that this abelian category is equivalent to the category of representations of the quiver $D_n$. Hence the reduced Drinfeld double of the twisted Hall algebra of $\Sh_{\Lambda_{D_n}}(D)$ is isomorphic to the quantum group of type $D_n$.

We define two isotropic sets:
\begin{align}
\Lambda_{D_\bQ}&\coloneqq\bigcup_{i=1}^3T^\ast_{I_i}\bR^2\cup \bigcup_{a\in I_1\cap \bQ}T^{\ast,<0}_{a}\bR^2\ ,\\
\Lambda_{D_\bR}&\coloneqq\bigcup_{i=1}^3T^\ast_{I_i}\bR^2 \cup \bigcup_{a\in I_1}T^{\ast,<0}_{a}\bR^2\ .
\end{align}

Then one can define the full subcategory $\Sh_{\Lambda_{D_\bQ}}(D)$ (resp.\ $\Sh_{\Lambda_{D_{\bR}}}(D)$) consisting of constant sheaves supported on closed-open intervals with their ends in $\Lambda_{D_\bQ}$ (resp, $\Lambda_{D_{\bR}}$) forms a finitary abelian category. We define
\begin{align}
\Ubf_\upsilon(\mathfrak{so}(2I_{\Q}))&\coloneqq\mathbf{D}(\mathbf{H}^{\mathrm{tw}}(\Sh_{\Lambda_{D_\bQ}}(D)))\ ,\\
\Ubf_\upsilon(\mathfrak{so}(2I))&\coloneqq\mathbf{D}(\mathbf{H}^{\mathrm{tw}}(\Sh_{\Lambda_{D_\bR}}(D)))\ .
\end{align}
To make things more explicit, let us write up the Euler forms here. There are three new situations which did not appear in $A_n$-case. Let $a, b\in (0,1]$.
\begin{enumerate}
	
	\item[Case (T):]
	\begin{align}
	H^\bullet\R \Hom(k_{\overline{I_2} \cup (0, a]}, k_{\overline{I_3}})\simeq 0\quad\mbox{and}\quad
	H^\bullet\R \Hom(k_{\overline{I_3}}, k_{\overline{I_2}\cup (0, a]})\simeq k[-1]\ .
	\end{align}
	
	\item[Case (Y):]
	\begin{align}
	H^\bullet\, \R \Hom(k_{\overline{I_2}\cup (0, a]}, k_{\overline{I_3}\cup(0,b]})\simeq
	\begin{cases}
	 k & \text{if $a> b$}\ , \\
	0 &  \text{otherwise}\ .
	\end{cases}
	\end{align}

	\item[Case (V):]
	\begin{align}	
		H^\bullet\R \Hom(k_{\overline{I_2}}, k_{\overline{I_3}})\simeq 0\ .
	\end{align}
	
\end{enumerate}

\begin{remark}
	One can also consider the subset \eqref{eq:D-case} of $\R^2$ with the opposite orientation, that is,
	\begin{align}\label{eq:D-case-opposite}
	I\coloneqq\lc(x, 0)\relmid x\in[-1, 0]\rc \cup \lc (x, e^{1/x})\relmid 0< x\leq 1\rc\cup \lc (x, -e^{1/x})\relmid 0<x\leq 1\rc\ .
	\end{align}
	We call the components of this union $I_1, I_2, I_3$ from left to right. In such a case the formulas above may be rewritten as follows, for $a, b\in [-1, 0)$:
	\begin{enumerate}
		\item[Case (T'):]
		\begin{align}
		H^\bullet\, \R \Hom(k_{I_2\cup (a, 0]}, k_{I_3})\simeq k[-1]\quad\mbox{and}\quad	H^\bullet\, \R \Hom(k_{I_3}, k_{I_2\cup (a, 0]})\simeq 0\ .
		\end{align}
		
		\item[Case (Y'):]
		\begin{align}
		H^\bullet\R \Hom(k_{I_2\cup (a, 0]}, k_{I_3\cup(b,0]})\simeq
		\begin{cases}
		k&  \text{if $a< b$}\ , \\
		0 & \text{otherwise}\ .
		\end{cases}
		\end{align}
		
		\item[Case (V'):]
		\begin{align}
			H^\bullet\, \R \Hom(k_{I_2}, k_{I_3})\simeq 0\ .
		\end{align}
		
	\end{enumerate}
\end{remark}

\end{document}